\documentclass{article}
\usepackage{amsmath,amssymb,amsthm,scrtime,bbm,verbatim,comment}
\usepackage{bm}
\usepackage{csquotes}
\MakeOuterQuote{"}

\usepackage{graphics}
\usepackage{amsfonts,xcolor}
\usepackage{tikz}
\usetikzlibrary{patterns}

\numberwithin{equation}{section}

\newcommand{\sss}{\scriptscriptstyle}

\renewcommand{\P}{\mathbb{P}}
\newcommand{\N}{\mathbb{N}}

\newcommand{\R}{\mathbb{R}}
\newcommand{\eps}{\epsilon}
\newcommand{\M}{\mathcal{M}}
\newcommand{\F}{\mathcal{F}}

\newcommand{\la}{\langle}
\newcommand{\ra}{\rangle}
\newcommand{\B}{\mathcal{B}}
\newcommand{\ext}{\mathcal{E}}
\newcommand{\extinction}{\mathcal{E}_{{\rm fin}}}
\newcommand{\extinguishing}{\mathcal{E}_{{\rm lim}}}
\newcommand{\bbP}{{\bf{P}}}
\newcommand{\bbF}{{\boldsymbol{\mathcal{F}}}}
\newcommand{\mglX}{W_{\infty}^{\phi}(X)}
\newcommand{\mglZ}{W_{\infty}^{\phi/w}(Z)}
\newcommand{\tphi}{\widetilde{\phi}}
\newcommand{\bbeta}{\bar{\beta}}
\newcommand{\cem}{\dag}
\newcommand{\tree}{\mathcal{T}}
\newcommand{\alongbb}{\mathfrak{a}}
\newcommand{\bbpb}{\mathfrak{b}}
\newcommand{\ZGen}{G}

\newcommand{\n}{{\bf{n}}}
\newcommand{\m}{{\bf{m}}}
\newcommand{\D}{\mathcal{D}} 
\newcommand{\simplefct}{\mathcal{S}}
\newcommand{\T}{\mathbb{T}} 
\newcommand{\locbdpos}{lbp}
\newcommand{\wtw}{\widetilde{w}}

\newcommand{\res}{\kappa}
\newcommand{\OUp}{\gamma}
\newcommand{\exPa}{\vartheta}

\newcommand{\pint}[1]{\la #1\ra}
\newcommand{\bigpint}[1]{\big\la #1\big\ra}
\newcommand{\Bigpint}[1]{\Big\la #1\Big\ra}
\newcommand{\twonorm}[1]{\| #1 \|}
\newcommand{\supnorm}[1]{\| #1\|_{\infty}}
\newcommand{\supTnorm}[1]{\| #1\|_{\infty,T}}

\newtheorem{theorem}{Theorem}[section]
\newtheorem{lemma}[theorem]{Lemma}

\newtheorem{corollary}[theorem]{Corollary}
\newtheorem{proposition}[theorem]{Proposition}

\theoremstyle{definition}

\newtheorem{example}[theorem]{Example}
\newtheorem{remark}[theorem]{Remark}
\newtheorem{ass}{Assumption}

\newtheorem{notation}[theorem]{Notation}

\newtheorem{innercustomassump}{Assumption}
\newenvironment{customassump}[1]
  {\renewcommand\theinnercustomassump{#1}\innercustomassump}
  {\endinnercustomassump}

\title{ Spines, skeletons and the
Strong Law of Large Numbers\\
for superdiffusions}
\author{Maren Eckhoff\thanks{Department of Mathematical Sciences, University of Bath, Bath BA2 7AY, United Kingdom. Email: {\tt m.eckhoff@bath.ac.uk, a.kyprianou@bath.ac.uk}}
\and
Andreas E. Kyprianou$^*$
\and
Matthias Winkel\thanks{Department of Statistics, University of Oxford, Oxford OX1 3TG, United Kingdom. Email: {\tt winkel@stats.ox.ac.uk}}}

\begin{document}
\maketitle

\begin{abstract}
Consider a supercritical superdiffusion $(X_t)_{t \ge 0}$ on a domain $D \subseteq \R^d$ with branching mechanism
\[
(x,z)\mapsto -\beta(x) z+\alpha(x) z^2 + \int_{(0,\infty)} \big(e^{-yz}-1+yz\big) \, \Pi(x,dy).
\]
The skeleton decomposition provides a pathwise description of the process in terms of immigration along a branching particle diffusion. We use this decomposition to derive the Strong Law of Large Numbers (SLLN) for a wide class of superdiffusions from the corresponding result for branching particle diffusions. That is, we show that for suitable test functions $f$ and starting measures $\mu$,\vspace{-0.15cm}
\[
\frac{\la f,X_t\ra}{P_{\mu}[\la f,X_t\ra]} \to W_{\infty} \qquad P_{\mu}\text{-almost surely as $t\to\infty$}, 
\]
where $W_{\infty}$ is a finite, non-deterministic random variable characterised as a martingale limit. Our method is based on skeleton and spine techniques and offers structural insights into the driving force behind the SLLN for superdiffusions. The result covers many of the key examples of interest and, in particular, proves a conjecture by Fleischmann and Swart \cite{FleSwa03} for the super-Wright-Fisher diffusion.

\medskip


\smallskip


\end{abstract}

\section{Introduction}\label{intro}

The asymptotic behaviour of the total mass assigned to a compact set by a superprocess was first characterised by Pinsky \cite{Pin96} at the level of the first moment. Motivated by this study,
Engl\"ander and Turaev \cite{EngTur02} proved weak convergence of the ratio between the total mass in a compact set and its expectation. Others have further improved the mode of convergence; specifically, several authors conjectured an almost sure convergence result for a wide class of superprocesses \cite{Eng07,EngWin06,FleSwa03,LiuRenSon13}. However, up to now it has not been possible to deal with many of the classical examples of interest. In the existing literature, for almost sure convergence, either motion and branching mechanism have to obey restrictive conditions \cite{CheRenWan08} or the domain is assumed to be of finite Lebesgue measure \cite{LiuRenSon13}. In this article, we make a significant step towards closing the gap and establish the Strong Law of Large Numbers (SLLN) for a wide class of superdiffusions on arbitrary domains. In particular, we prove a conjecture by Fleischmann and Swart for the super-Wright-Fisher diffusion.

Methodologically, previous articles concerned with almost sure limit behaviour of superprocesses relied on Fourier analysis, functional analytic arguments or used the martingale formulation for superprocesses combined with stochastic analysis. We take a different approach. The core of our proof is the skeleton decomposition that represents the superprocess as an immigration process along a branching particle process, called the skeleton, where immigration occurs in a Poissonian way along the space-time trajectories and at the branch points of the skeleton. The skeleton may be interpreted as immortal particles that determine the long-term behaviour of the process.
We exploit this fact and carry the SLLN from the skeleton over to the superprocess. Apart from the result itself, this approach provides insights into the driving force behind the law of large numbers for superprocesses.

A more detailed literature review and discussion of the ideas of proof is deferred to Sections~\ref{sec:literature} and \ref{sec:pf_idea}. 
Before, we introduce the model in Section~\ref{sec:model}, our assumptions are stated in Section~\ref{sec:assumptions} and the main results are collected in Section~\ref{sec:results}.

\subsection{Model and notation}\label{sec:model}

Let $d \in \N$ and let $D \subseteq \R^d$ be a nonempty domain. For $k \in \N_0$, $\eta >0$, we write $C^{k,\eta}(D)$ for the space of real-valued functions on $D$, whose $k$-th order partial derivatives are locally $\eta$-H\"older continuous, $C^{\eta}(D):=C^{0,\eta}(D)$.
We denote by $\B(D)$ the Borel $\sigma$-algebra on $D$. The notation $B \subset\subset D$ means that $B \in \B(D)$ is bounded and there is an open set $B_1$ such that $B\subseteq B_1 \subseteq \overline{B}_1 \subseteq D$. 
The Lebesgue measure on $\B(D)$ is denoted by $\ell$; the set of finite (and compactly supported) measures on $\B(D)$ is denoted by $\M_f(D)$ (and  
$\M_c(D)$ resp.). When $\mu$ is a measure on $\B(D)$ and $f\colon D \to \R$ measurable, let $\la f,\mu \ra:=\int_D f(x) \,\mu(dx)$, whenever the right-hand side makes sense. 
If $\mu$ has a density $\rho$ with respect to $\ell$, we write $\pint{ f,\rho}=\pint{ f,\mu}$.
For any metric space $E$, we denote by $p(E)$ and $b(E)$ the sets of Borel measurable and, respectively, nonnegative and bounded functions on $E$, and let $bp(E)=b(E) \cap p(E)$.
\smallskip

Let $(\xi=(\xi_t)_{t\ge 0}; (\P_x)_{x \in D} )$ be a diffusion process on $D$ with generator
\[
L(x)=\frac{1}{2}\nabla \cdot a(x)\nabla +b(x) \cdot \nabla \qquad \text{on }D.
\]
The diffusion matrix $a\colon D \to \R^{d \times d}$ takes values in the set of symmetric, positive definite matrices. Moreover, all components of $a$ and $b\colon D \to \R^d$ belong to $C^{1,\eta}(D)$ for some $\eta \in (0,1]$ (the parameter $\eta$ remains fixed throughout the article). In other words, $\xi$ denotes the unique solution to the generalized martingale problem associated with $L$ on $D\cup\{\cem\}$, the one-point compactification of $D$ with cemetery state $\cem$, cf.\ Chapter I in \cite{Pin95}. We write $\tau_D=\inf\{t \ge 0\colon \xi_t \not\in D\}$.

Let $\beta \in C^{\eta}(D)$ be bounded and
\begin{equation}\label{eq:psiDef}
\psi_0(x,z):=\alpha(x)z^2 +\int_{(0,\infty)}\big(e^{-z y}-1+ zy\big) \,\Pi(x,dy),
\end{equation}
where $\alpha \in bp(D)$ and $\Pi$ is a kernel from $D$ to $(0,\infty)$ such that $x \mapsto \int_{(0,\infty)}(y \land y^2) \,\Pi(x,dy)$ belongs to $bp(D)$. The function $\psi_{\beta}(x,z):= -\beta(x) z+\psi_0(x,z)$ is called the \em branching mechanism\em. If $\Pi\equiv 0$, we say that the branching mechanism is quadratic. 
In Section~\ref{sec:quadratic}, we explain that our results carry over to a class of quadratic branching mechanisms with unbounded $\alpha$ and $\beta$.
\smallskip

The main process of interest in this article is the \em $(L,\psi_{\beta};D)$-superdiffusion\em, which we denote by $X=(X_t)_{t \ge 0}$. Its distribution is denoted by $P_{\mu}$ if the process is started in $\mu \in \M_f(D)$. That is, $X$ is a $\M_f(D)$-valued time-homogeneous Markov process such that for all $\mu \in \M_f(D)$, $f \in bp(D)$ and $t\ge 0$,
\begin{equation}\label{eq:logLaplace}
P_{\mu}[e^{-\la f,X_t\ra}]=e^{-\la u_f(\cdot,t), \mu\ra},
\end{equation}
where $u_f$ is the unique nonnegative solution to the {\it{mild equation}}
\begin{equation}\label{eq:mild}
u(x,t)=S_tf(x)-\int_0^{t} S_s[\psi_0(\cdot,u(\cdot,t-s))](x)\, ds\quad \text{for all }(x,t) \in D\times [0,\infty).
\end{equation}
Here $S_tg(x):=\P_x[e^{\int_0^{t} \beta(\xi_s) \, ds} g(\xi_{t}) \mathbbm{1}_{\{t <\tau_D\}}]$ for all $g\in p(D)$, i.e.\ $(S_t)_{t \ge 0}$ denotes the semigroup of the differential operator $L+\beta$. Every function $g$ on $D$ is automatically extended to $D \cup \{\cem\}$ by $g(\cem):=0$. Hence, 
\[
S_tg(x)=\P_x\big[e^{\int_0^{t} \beta(\xi_s) \, ds} g(\xi_{t})\big].
\]
We refer to $\xi$ as the \em underlying motion \em or just the \em motion \em in the \em space \em $D$. 
Informally, the $\M_f(D)$-valued process $X=(X_t)_{t\ge 0}$ describes a cloud of infinitesimal particles independently evolving according to the motion $\xi$ and branching in a spatially dependent way according to the branching mechanism $\psi_{\beta}$. The existence of the superprocess $X$ is guaranteed by \cite{Dyn93,Fit88} and it satisfies the branching property (see (1.1) in \cite{Fit88} for a definition). By Theorem~3.1 in \cite{Dyn93a} or Theorem~2.11 in \cite{Fit88}, there is a version of $X$ such that $t \mapsto \la f,X_t\ra$ is almost surely right-continuous for all continuous $f \in bp(D)$. We will always work with this version.
In most texts the mild equation \eqref{eq:mild} is written in a slightly different form: instead of $(S_t)_{t\ge 0}$, the semigroup of $L$ is used and $\psi_0$ is replaced by $\psi_{\beta}$. Using Feynman-Kac arguments (see Lemma~\ref{lem:semi}(i) in the appendix) and Gronwall's lemma one easily checks that (for bounded $\beta$) the two equations are equivalent.

The main goal of this article is to determine the large-time behaviour of
\begin{equation}\label{eq:LLNratio}
\frac{\la f,X_t\ra}{P_{\mu}[\la f,X_t\ra]}
\end{equation}
for suitable test functions $f$ and starting measures $\mu$. We say that $X$ satisfies the \emph{Strong Law of Large Numbers} (SLLN) if, for all test functions $f\in C_c^+(D)$, $f\not= \mathbf{0}$, the ratio in \eqref{eq:LLNratio} converges to a finite, non-deterministic random variable which is independent of $f$. Here $C_c^+(D)$ denotes the space of nonnegative, continuous functions of compact support, and $\mathbf{0}$ is the constant function with value $0$.

\subsection{Statement of assumptions}\label{sec:assumptions}

Historically, superprocesses have been investigated using tools from the field of analysis starting from the mild equation~\eqref{eq:mild} or the corresponding partial differential equation, and by exploiting a characterisation of the superprocess as a high-density limit of branching particle processes.
A more probabilistic view on supercritical superprocesses is offered by the \emph{skeleton decomposition}. This, 
by now classical, 
cf. \cite{EvaOCo94,EngPin99,DuqWin07,BerFonMar08,BerKypMur11,KypPerRen13pre}, decomposition has been studied under a variety of names. 
It provides a pathwise representation of the superprocess as an immigration process along a supercritical branching particle process, that we call the \emph{skeleton}.
The skeleton captures the global behaviour of the superprocess and its discrete nature makes it much more tractable than the superprocess itself. We exploit these facts to establish the SLLN for superdiffusions. Specifically, our fundamental aim it to show that the SLLN for superdiffusions follows as soon as an appropriate SLLN holds for its skeleton. Given the existing knowledge for branching particle processes, this will lead us to a large class of superprocesses for which the SLLN can be stated.

Classically, the skeleton was constructed using the event $\extinction=\{\exists t\ge 0\colon X_t(D)=0\}$ of extinction after finite time to guide the branching particle process into regions where extinction of the superprocess is unlikely. The key property of $\extinction$ exploited in the skeleton decomposition is that the function $x\mapsto w(x)=-\log P_{\delta_x}(\extinction)$ gives rise to the multiplicative martingale $((e^{-\pint{w,X_t}})_{t\ge 0};P_{\mu})$. In the more general setup of the present article, we assume only the existence of such a \emph{martingale function} $w$.

\begin{ass}[Skeleton Assumption]\label{as:skeleton}
There exists a function $w \in p(D)$ with $w(x)>0$ for all $x\in D$,
\begin{align}
&\sup_{x \in B} w(x) <\infty\qquad \qquad\quad\;\;\, \text{for all }  B \subset\subset D, \label{eq:wlocbd}\\
&P_{\mu}\big[e^{-\pint{w,X_t}}\big]= e^{-\pint{ w,\mu}} \qquad \text{for all } \mu \in \M_c(D).\label{eq:mgFct}
\end{align}
\end{ass}

The function $w$ allows us to define the skeleton as a branching particle diffusion $Z$, where the spatial movement of each particle is equal in distribution to $(\xi=(\xi_t)_{t\ge 0};(\P_x^w)_{x\in D})$ with
\begin{equation}\label{eq:Pwdef}
\frac{d\P_x^w}{d\P_x}\bigg|_{\sigma(\xi_s\colon s \in [0,t])}=\frac{w(\xi_t)}{w(x)}\exp\Big(-\int_0^t \frac{\psi_{\beta}(\xi_s,w(\xi_s))}{w(\xi_s)}\, ds\Big) \qquad \text{on }\,\{t <\tau_D\} \text{ for all }t\ge 0.
\end{equation}
We will see in Lemma~\ref{lem:PwWellDef} that $\P_x^w$ is well-defined. Each particle dies at spatially dependent rate $q \in p(D)$ and is replaced by a random number of offspring with distribution $(p_k(x))_{k \ge 2}$, where $x$ is the location of its death. The branching rate $q$ and the offspring distribution $(p_k)_{k \ge 2}$ are uniquely identified by
\begin{equation}\label{eq:Generator}
G(x,s):=q(x) \sum_{k=2}^{\infty} p_k(x) (s^k-s) = \frac{1}{w(x)}\Big(\psi_0\big(x,w(x)(1-s)\big) - (1-s)\psi_0\big(x,w(x)\big)\Big)
\end{equation}
for all $s\in [0,1]$ and $x\in D$. The fact that $q$ and $(p_k)_{k \ge 2}$ are well-defined by \eqref{eq:Generator} is contained in Theorem~\ref{thm:KPR} below. In Section~\ref{sec:skeleton}, we define $Z$ on a rich probability space with probability measures $\bbP_{\mu}$, $\mu \in \M_f(D)$, where the initial configuration of $Z$ under $\bbP_{\mu}$ is given by a Poisson random measure with intensity $w(x)\mu(dx)$. 

As noted earlier, we are interested in the situation where the skeleton itself satisfies a SLLN. There is a substantial body of literature available that analyses the long-term behaviour of branching particle diffusions. 
To delimit the regime we want to study, we make two regularity assumptions. A detailed discussion of all assumptions can be found in Section~\ref{sec:properties}.

The first condition ensures that the semigroup $(S_t)_{t \ge 0}$ of $L+\beta$ grows precisely exponentially on compactly supported, continuous functions.

\begin{ass}[Criticality Assumption] \label{as:criticality}
The second order differential operator $L+\beta$ has positive \emph{generalised principal eigenvalue}
\begin{equation}\label{eq:lambdacdef}
\lambda_c:=\lambda_c(L+\beta):=\inf\big\{\lambda \in \R\colon \exists u \in C^{2,\eta}(D), u>0, (L+\beta -\lambda)u =0\big\}>0.
\end{equation}
Moreover, we assume that the operator $L+\beta - \lambda_c$ is \emph{critical}, that is, it does not possess a Green's function but there exists $\phi \in C^{2,\eta}(D)$, $\phi>0$, such that $(L+\beta-\lambda_c)\phi=0$. In this case, $\phi$ is unique up to constant multiples and is called the \emph{ground state}. With $L+\beta-\lambda_c$ also its formal adjoint is critical (cf.\ Pinsky \cite{Pin95}) and the corresponding ground state is denoted by $\tphi$. We further assume that $L+\beta -\lambda_c$ is \emph{product $L^1$-critical}, i.e.\ $\pint{ \phi,\tphi}<\infty$, and we normalize to obtain $\la \phi,\tphi\ra=1$. 
\end{ass}

We show in Corollary~\ref{cor:mg} below that under Assumptions~\ref{as:skeleton} and \ref{as:criticality} the process 
\[
W_t^{\phi/w}(Z)=e^{-\lambda_c t} \la \phi/w,Z_t\ra, \quad t \ge 0,
\]
is a nonnegative $\bbP_{\mu}$-martingale for all $\mu \in \M_f^{\phi}(D):=\{ \mu \in M_f(D)\colon \la \phi,\mu\ra<\infty\}$ and $(W_t^{\phi/w}(Z))_{t\ge 0}$ has an almost sure limit. To have the notation everywhere, we define $\mglZ:=\liminf_{t \to \infty}W_t^{\phi/w}(Z)$.

Our second regularity assumption consists essentially of moment conditions.

\begin{ass}[Moment Assumption]\label{as:moment1}
There exists $p\in (1,2]$ such that
\begin{align}
& \sup_{x \in D} \phi(x) \alpha(x) <\infty\label{alpha_cond1}\\
&\sup_{x \in D} \phi(x) \int_{(0,1]} y^2\, \Pi(x, dy) <\infty\label{as:PiLowerTail1}\\
&\sup_{x \in D} \phi(x)^{p-1} \int_{(1,\infty)} y^p\, \Pi(x,dy)<\infty,\label{as:PiUpperTail1}\\
&\pint{\phi^{p-1}, \phi\tphi} <\infty,\label{phi_p_cond1}\\
&\Bigpint{\int_{(1,\infty)} y^2 e^{-w(\cdot) y} \, \Pi(\cdot,dy), \phi \tphi}<\infty \label{cond:lat_cont1}.
\end{align}
\end{ass}

The parameter $p$ remains fixed throughout the article. Assumption~\ref{as:moment1} is satisfied, for example, when $\phi$ is bounded and $\sup_{x \in D} \int_{(1,\infty)} y^2 \, \Pi(x,dy)<\infty$. 
These second moment conditions appeared in the literature (cf.\ Section~\ref{sec:moments}) and we will see several examples in Section~\ref{sec:examples}.
However, our results are valid under the weaker conditions of Assumption~\ref{as:moment1}. In Sections~\ref{sec:moments} and \ref{sec:quadratic}, we explain that in the case of a quadratic branching mechanism only \eqref{alpha_cond1} is needed.

The SLLN has been proved for a large class of branching particle diffusions. Where it has not been established, yet, we assume a SLLN for the skeleton $Z$. It will be sufficient to assume convergence along lattice times.

\begin{ass}[Strong Law Assumption]\label{as:SLLN}
For all $\mu \in \M_c(D)$, $\delta>0$ and continuous $f\in p(D)$ with $fw/\phi$ bounded,
\[
\lim_{n \to \infty} e^{-\lambda_cn\delta}\pint{ f ,Z_{n \delta}}=\pint{ f,w\tphi } W_{\infty}^{\phi/w}(Z) \qquad \bbP_{\mu}\text{-almost surely}.
\]
\end{ass}
At first, Assumption~\ref{as:SLLN} may look like a strong assumption. However, given Assumptions~\ref{as:skeleton}--\ref{as:moment1}, the SLLN for the skeleton has been proved under two additional conditions. The first condition controls the spread of the support of the skeleton when started from a single particle; the second condition is a uniformity assumption on the convergence of an associated ergodic motion (the ``spine'') to its stationary distribution. See Theorem~\ref{thm:SLLN_skeleton} for details. These conditions hold for a wide class of processes and we demonstrate this for several key examples in Section~\ref{sec:examples}.

\subsection{Statement of the main results}\label{sec:results}

Before we state the SLLN for superdiffusion $X$, we relate the limiting random variable of \eqref{eq:LLNratio} to the limit that appears in Assumption~\ref{as:SLLN}. In Corollary~\ref{cor:mg} below we show that under Assumption~\ref{as:criticality}, the process
\[
W_t^{\phi}(X)=e^{-\lambda_c t} \la \phi,X_t\ra, \quad t \ge 0,
\]
is a nonnegative $P_{\mu}$-martingale for all $\mu \in \M_f^{\phi}(D)=\{ \mu \in \M_f(D)\colon \la \phi,\mu\ra <\infty\}$ and $(W_t^{\phi}(X))_{t\ge 0}$ has an almost sure limit. To have the notation everywhere, we define $\mglX:=\liminf_{t \to \infty} W_t^{\phi}(X)$.

\begin{proposition}\label{pro:mglim} Suppose Assumptions~\ref{as:skeleton}, \ref{as:criticality}, \eqref{alpha_cond1}--\eqref{as:PiUpperTail1} hold. For all $\mu \in \M_f^{\phi}(D)$, the martingales $(W_t^{\phi}(X))_{t \ge 0}$ and $(W_t^{\phi/w}(Z))_{t \ge 0}$ are bounded in $L^p(\bbP_{\mu})$ and
\begin{equation}\label{eq:mglimAgree}
\mglX=\mglZ \qquad \bbP_{\mu} \text{-almost surely.}
\end{equation}
\end{proposition}

Our main theorem is the following.

\begin{theorem}\label{thm:SLLN}
Suppose Assumptions 1--4 hold. For every $\mu \in \M_f^{\phi}(D)$, there exists a measurable set $\Omega_0$ such that $P_{\mu}(\Omega_0)=1$ and, on $\Omega_0$, for all $\ell$-almost everywhere continuous functions $f \in p(D)$ with $f/\phi$ bounded,
\begin{equation}\label{eq:SLLN}
\lim_{t \to \infty} e^{-\lambda_ct}\la f,X_t\ra=\la f,\tphi\ra \mglX.
\end{equation}
The convergence in \eqref{eq:SLLN} also holds in $L^1(P_{\mu})$. In particular, $P_{\mu}[\mglX]=\la \phi,\mu \ra$.
\end{theorem}

Even though our main interest is almost sure convergence, Theorem~\ref{thm:SLLN} also implies new results for convergence in probability; see the examples in Section~\ref{sec:examples}.
We record the following corollary of Theorem~\ref{thm:SLLN} to present the result in possibly more familiar terms.

\begin{corollary}
Suppose Assumptions 1--4 hold. In the vague topology, $e^{-\lambda_c t} X_t \to \mglX \tphi \ell$ $P_{\mu}$-almost surely as $t \to \infty$. If, in addition, $\phi$ is bounded away from zero, then the convergence holds in the weak topology $P_{\mu}$-almost surely.
\end{corollary}

Finally, we present the SLLN as announced in \eqref{eq:LLNratio}. This makes the comparison between $\la f,X_t\ra$ and its mean explicit.

\begin{corollary}\label{cor:SLLN}
Suppose that Assumptions 1--4 hold. For all $\mu \in \M_f^{\phi}(D)$, $\mu \not \equiv 0$, $f \in C_c^+(D)$, $f\not=\mathbf{0}$,
\[
\lim_{t \to \infty} \frac{\la f,X_t\ra}{P_{\mu}[\la f,X_t\ra]} = \frac{1}{\la \phi ,\mu\ra} \mglX \qquad P_{\mu}\text{-almost surely and in }L^1(P_{\mu}).
\]
\end{corollary}

The Weak Law of Large Numbers (WLLN), and even the $L^1$-convergence in \eqref{eq:SLLN}, can be obtained without assuming the SLLN for the skeleton as the next theorem reveals.

\begin{theorem}\label{thm:L1}
 Suppose Assumptions~\ref{as:skeleton}, \ref{as:criticality}, \eqref{alpha_cond1}--\eqref{phi_p_cond1} hold. For all $\mu \in \M_f^{\phi}(D)$ and $f \in p(D)$ with $f/\phi$ bounded, the convergence in \eqref{eq:SLLN} holds in $L^1(P_{\mu})$.
\end{theorem}

\subsection{Literature review}\label{sec:literature}

Terminology in the literature is not always consistent, so let us clarify that we refer to branching particle processes and superprocesses as branching diffusions and superdiffusions, respectively, if the underlying motion is a diffusion. Similar wording is used for other classes of underlying motions.

The limit theory of supercritical branching processes has been studied since the 1960s when sharp statements were established for classical finite type processes \cite{KesSti66,Ath68}. The first result for branching diffusions was due to Watanabe \cite{Wat67} in 1967, who proved an almost sure convergence result for branching Brownian motion and certain one-dimensional motions. The key ingredient to the proof was Fourier analysis, a technique recently used by Wang \cite{Wan10} and Kouritzin and Ren \cite{KouRen14} to establish the SLLN for super-Brownian motion. Super-Brownian motion on $\R^d$ with a spatially independent branching mechanism does not fall into the framework of the current article since $L+\beta-\lambda_c$ is not product $L^1$-critical in that case. Rather, $\phi=\tphi=\mathbf{1}$, where $\mathbf{1}$ denotes the constant function with value $1$, and $e^{-\lambda_c t}P_{\mu}[\la f,X_t\ra]$ converges to zero for all $f\in C_c^+(D)$.  
The missing scaling factor is $t^{d/2}$ and $P_{\mu}[\la f,X_t\ra] \sim (2\pi t)^{-d/2} e^{\lambda_c t} \la f,\mathbf{1}\ra \mu(\R^d)$ for $\mu\in\M_c(\R^d)$. Wang's \cite{Wan10} SLLN for super-Brownian motion takes the form 
\[
\lim_{t \to \infty} \frac{\la f,X_t\ra}{P_{\mu}[\la f,X_t\ra]} =\frac{\mglX}{\mu(\R^d)} \quad P_{\mu}\text{-almost surely, with martingale limit }W_\infty^\phi(X)=\lim_{t\rightarrow\infty}e^{-\lambda_c t}\la\mathbf{1},X_t\ra,
\]
for all nontrivial nonnegative continuous functions with compact support, and for $\mu =\delta_x$, $x \in \R^d$. 
Watanabe's argument is thought to be incomplete because the regularity for his argument is not proven, see \cite{Wan10}. Biggins \cite{Big92} developed a method to show uniform convergence of martingales for branching random walks. Wang combined these arguments with the compact support property of super-Brownian motion started from $\mu \in \M_c(D)$.
Kouritzin and Ren \cite{KouRen14} proved the SLLN for super-stable processes of index $\alpha \in (0,2]$ with spatially independent quadratic branching mechanism. The correct scaling factor in this case is $t^{d/\alpha} e^{-\lambda_c t}$. The authors allow any finite start measure with finite mean and a class of continuous test functions that decrease sufficiently fast at infinity.
Fourier-analytic methods were also used by Grummt and Kolb \cite{GruKol13} to prove the SLLN for the two-dimensional super-Brownian motion with a single point source (see \cite{FleMue04} for the definition and a proof of existence of this process).
Earlier, Engl\"ander \cite{Eng09} established convergence in probability for a class of superdiffusions that do not necessarily satisfy Assumption~\ref{as:criticality} using a time-dependent $h$-transform developed in \cite{EngWin06}.\smallskip

In the product $L^1$-critical case, the dominant method to prove almost sure limit theorems is due to Asmussen and Hering \cite{AsmHer76} (Kaplan and Asmussen use a similar method in \cite{KapAsm76}).
The main idea is as follows. For $s,t \ge 0$, write $\F_t=\sigma(X_r\colon r \le t)$ and
\begin{align*}
e^{-\lambda_c(s+t)} \la f,X_{s+t}\ra &= e^{-\lambda_c t} P_{\mu}[e^{-\lambda_c s} \la f,X_{s+t}\ra|\F_t] + \Big(e^{-\lambda_c(s+t)} \la f,X_{s+t}\ra - e^{-\lambda_c t} P_{\mu}[e^{-\lambda_c s} \la f,X_{s+t}\ra|\F_t]\Big)\\
&= \text{CE}_f(s,t)+\text{D}_f(s,t).
\end{align*}
Here CE stands for "conditional expectation" and D for "difference". The first step is to show $\text{D}_f(s,t) \to 0$ as $t \to \infty$. This is usually done via a Borel-Cantelli argument and therefore requires a restriction to lattice times $t=n\delta$. 
The second step is to show that $\text{CE}_f(s,t)$ behaves like the desired limit for $s$ and $t$ large. This is the hardest part of the proof and usually causes most of the assumptions. The third and last step is to extend the result from lattice to continuous time. 

Asmussen and Hering control $\text{CE}_f(s,t)$ for branching particle processes by a uniform Perron-Frobenius condition on the semigroup $(S_t)_{t \ge 0}$. Passage to continuous time is obtained under additional continuity assumptions on process and test functions. Recently, their method was generalized by Engl\"ander et al.\ \cite{EngHarKyp10} to show the SLLN for a class of branching diffusions on arbitrary domains. The authors control $\text{CE}_{f}(s,t)$ by an assumption that restricts the speed at which particles spread in space and a condition on the rate at which a certain ergodic motion (the ``spine'') converges to its stationary distribution. 
\smallskip

While Asmussen and Hering's idea for the proof of SLLNs along lattice times is rather robust and (under certain assumptions) feasible also for superprocesses, the argument used for the transition from lattice to continuous time relies heavily on the finite number of particles in the branching diffusion.
\smallskip

A new approach to almost sure limit theorems for branching processes was introduced by Chen and Shiozawa \cite{CheShi07} in the setup of branching symmetric Hunt processes. Amongst other assumptions, a spectral gap condition was used to obtain a Poincar\'{e} inequality which constitutes the main ingredient in the proof along lattice times. For the transition to continuous times the argument from Asmussen and Hering was adapted. Chen et al.\ \cite{CheRenWan08} proved the first SLLN for superprocesses and relied on the same Poincar\'{e} inequality and functional analytic methods for the result along lattice times. For the transition to continuous time, Perkins' It\^o formula for superprocesses \cite{Per02} was used. Even though their proven SLLN holds on the full domain $\R^d$, the assumptions on motion and branching mechanism are restrictive in the following way: the motion has to be symmetric (and in the diffusive case must have a uniformly elliptic generator) and the coefficients of the branching mechanism have to satisfy a strict Kato class condition.

The idea to use stochastic analysis was brought much further by Liu et al.\ \cite{LiuRenSon13}. The authors gave a proof which is based entirely on the martingale problem for superprocesses and decomposed the process into three martingale measures. Moreover, they introduced a new technique for the transition from lattice to continuous times based on the resolvent operator and estimates for the hitting probabilities of diffusions. 
The proof by Liu et al.\ follows again the three steps of Asmussen and Hering. To control the conditional expectation $\text{CE}_{f}(s,t)$, they assume that the transition density of the underlying motion is intrinsically ultracontractive and that the domain $D$ is of finite Lebesgue measure. This assumption excludes most of the classical examples, see Section~\ref{sec:examples}. 

To complete our review, we mention that the first Law of Large Numbers for superdiffusions was 
proven by Engl\"ander and Turaev \cite{EngTur02} on the domain $D=\R^d$. 
The authors use analytic tools from the theory of dynamical systems, in particular properties of invariant curves, to show the convergence in distribution. Besides classical superdiffusions, the $1$-dimensional super-Brownian motion with a single point source is studied.

\subsection{Outline of the proof of Theorem~\ref{thm:SLLN}} \label{sec:pf_idea}

The key to our argument is the skeleton decomposition for the supercritical superprocess $X$. Intuitively, this representation result states that the superprocess is a cloud of subcritical superdiffusive mass immigrating off a supercritical branching diffusion, the skeleton, which governs the large-time behaviour of $X$. It is important to note that we use the skeleton to make a connection between the asymptotic behaviour of a branching diffusion and of the superdiffusion, and we do not use any classical  approximation of the superprocess by branching particle systems in a high density limit regime.
\smallskip

Broadly speaking, our proof of Theorem~\ref{thm:SLLN} follows the three steps of Asmussen and Hering outlined in Section~\ref{sec:literature}. However, instead of the full process $X$ we consider only the immigration occurring after time $t$ in the decomposition into conditional expectation $\text{CE}_f$ and difference $\text{D}_f$. This immigration is a subprocess of $X$ and we show that the stated convergence for the full process follows when the subprocess converges to the claimed limit.
\smallskip

Using the tree structure of the skeleton, we can split the immigration that occurs after time $t$ according to the different branches of the skeleton at time $t$. This fact constitutes the main connection between skeleton and immigration that allows us to prove results about the asymptotic behaviour of the immigration process.
To analyse the conditional expectation $\text{CE}_f$ for the immigration after time $t$, we use the SLLN for the skeleton. After exponential rescaling, the immigration along different branches up to a fixed time $s$ is of constant order and the SLLN for the skeleton describes the asymptotic behaviour for large $t$. Taking the observed time frame $s$ to infinity then adjusts only the constants.
To replace the limiting random variable $\mglZ$, coming from the SLLN for the skeleton, by $\mglX$, we can, as it turns out, reverse the order in which these limits are taken. Taking first the observed time horizon $s$ to infinity for test function $\phi$, we recover the martingale for the skeleton as a consequence of the same invariance property of $\phi$ that makes $(W_t^{\phi}(X))_{t \ge 0}$ a martingale.
\smallskip

The analysis of $\text{D}_f$ for the immigration after time $t$ is fairly standard and for the transition from lattice to continuous times we adapt the argument by Liu et al.\ \cite{LiuRenSon13} relying again on the skeleton decomposition. 
The moment estimates needed for our analysis are obtained using a spine decomposition for the superprocess.

\subsection{Overview}
The outline of the paper is as follows. 
We start in Section~\ref{sec:skeleton} with an analysis of the skeleton assumption (Assumption~\ref{as:skeleton}) and give a detailed description of the skeleton decomposition. 
In the remainder of Section~\ref{sec:properties}, we discuss further basic properties of superprocesses and our other three main assumptions and we compare them to conditions that appeared in the literature.
In Section~\ref{sec:spine}, we give a spine decomposition for the superprocess $X$ and prove that the martingale $(W_t^{\phi}(X))_{t \ge 0}$ is bounded in $L^p$. 

The proofs of the main results are collected in Section~\ref{sec:main}. First, in Section~\ref{sec:reduction}, we reduce the SLLN to a statement that focuses on the main technical difficulty. 
In Section~\ref{mg_limits}, we show that the martingale limits for superprocess and skeleton agree and, in Section~\ref{sec:L1Conv}, we prove the WLLN stated in Theorem~\ref{thm:L1}. 
The asymptotic behaviour of the immigration process is studied in Section~\ref{sec:cond_exp} and the SLLN along lattice times is established.
The transition from lattice to continuous times is performed in Section~\ref{lattice_to_cont_sec} and we conclude our main results. 

In Section~\ref{sec:examples}, we provide several examples to illustrate our results. 
Spatially independent branching mechanisms are discussed in Section~\ref{sec:nonspatial}; quadratic branching mechanisms are considered in Section~\ref{sec:quadratic}. 
In Section~\ref{sec:bd_domain}, we study the super-Wright-Fisher diffusion and prove a conjecture by Fleischmann and Swart \cite{FleSwa03}.

Some minor statements needed along the way are proven in the appendix: Section~\ref{sec:FeynKac} contains Feynman-Kac-type arguments and Section~\ref{sec:beta_unb} discusses a generalised version of the mild equation \eqref{eq:mild} for $\beta$ bounded only from above.

\section{Preliminaries}

This section is split into two parts. In the first part, we discuss our four main assumptions, and in the second part, we prove that the martingale $(W_t^{\phi}(X))_{t \ge 0}$ converges in $L^p$.

\subsection{Basic properties}\label{sec:properties}

\subsubsection{Skeleton decomposition}\label{sec:skeleton}

In this section, we work under Assumption~\ref{as:skeleton}. 
The skeleton decomposition for supercritical superprocesses offers a pathwise description of the superprocess in terms of a supercritical branching particle process dressed with an immigration process. 
Heuristically, one can think of the skeleton as the prolific individuals of the branching process, i.e.\ individuals belonging to infinite lines of descent. 
The function $w$ assigns a small value to regions that prolific individuals
should avoid.
If $w(x)=-\log P_{\delta_x}(\ext)$ for some event $\ext$, then the skeleton particles avoid the behaviour specified by $\ext$.
Classical examples are the event of extinction in finite time
$\extinction=\{\exists t\ge 0 \colon \pint{\mathbf{1},X_t}=0\}$, \cite{EvaOCo94,EngPin99}, and the event of weak extinction $\extinguishing=\{\lim_{t\to \infty} \pint{\mathbf{1},X_t}=0\}$, \cite{BerFonMar08,BerKypMur11}.

For all $f\in p(D)$, we let $\widetilde{f}(x,t)=f(x)$ for all $(x,t) \in D\times [0,\infty)$. 
Dynkin \cite{Dyn93} derives the superprocess $X$ from \emph{exit measures} that describe the evolution of mass not only in time but also in space. 
He showed that for any domain $B \subseteq D$ and $t \ge 0$, there exists a random, finite measure $\widetilde{X}_t^B$ on $D\times [0,\infty)$ such that for all  $\mu \in \M_f(D)$ and $f\in bp(D)$,
\begin{equation}\label{eq:logLapExit}
P_{\mu}\big[e^{-\pint{\widetilde{f},\widetilde{X}_t^B}}\big]=e^{-\pint{\widetilde{u}_f^B(\cdot,t),\mu}},
\end{equation}
where $\widetilde{u}_f^B$ is the unique, nonnegative solution to the integral equation
\begin{equation}\label{eq:mildExit}
u(x,t)=\P_x\big[f(\xi_{t \land \tau_B})\big] - \P_x\Big[\int_0^{t \land \tau_B} \psi_{\beta}(\xi_s, u(\xi_s,t-s)) \, ds\Big] \quad \text{for all }(x,t) \in D\times [0,\infty),
\end{equation}
and $\tau_B=\inf\{t \ge 0\colon \xi_t \not\in B\}$. For $f \in p(D)$, there exists a sequence of functions $f_k \in bp(D)$ such that $f_k \uparrow f$ pointwise. 
By \eqref{eq:logLapExit}, $\widetilde{u}_{f_k}^B(x,t)$ is monotonically increasing in $k$ and we denote the limit by $\widetilde{u}_{f}^B(x,t)\in [0,\infty]$. 
With this notation, the monotone convergence theorem implies that \eqref{eq:logLapExit} is valid for all $f\in p(D)$. The same argument shows that \eqref{eq:mild} holds for all $f\in p(D)$ and \eqref{eq:mgFct} implies $u_w=w$. Hence, \eqref{eq:mgFct} holds for all $\mu \in \M_f(D)$.
The superprocess $X_t$ is obtained as a projection of $\widetilde{X}_t^D$ restricted to $D \times \{t\}$. 
Writing $\wtw(x,t)=w(x)$ for $(x,t) \in D\times [0,\infty)$, the Markov property (cf.\ Theorem I.1.3 \cite{Dyn93}) and \eqref{eq:mgFct} yield for all $\mu \in \M_f(D)$,
\begin{equation}\label{eq:mgPropw}
P_{\mu}\big[e^{-\pint{\widetilde{w},\widetilde{X}_t^B}}\big]=P_{\mu}\big[e^{-\pint{w,X_t}}\big]=e^{-\pint{w,\mu}}
\end{equation}
and comparing to \eqref{eq:logLapExit}, we deduce that $\widetilde{u}_{w}^B=\wtw$. Now let $B \subset \subset D$. 
If the support of $\mu$, $\text{supp}(\mu)$, is a subset of $B$, then 
 $\widetilde{X}_t^B$ is supported on the boundary of $B \times [0,t)$; if $\text{supp}(\mu) \subseteq D\setminus B$, then $\widetilde{X}_t^B=\mu$ almost surely (cf.\ Theorem~I.1.2 in \cite{Dyn93}). 
In particular, \eqref{eq:wlocbd} implies that $\widetilde{w}$ in $\pint{\wtw,\widetilde{X}_t^B}$ can be interpreted as a bounded function and we combine \eqref{eq:mgPropw} and \eqref{eq:mildExit} to obtain
\begin{equation}\label{eq:mildw}
w(x)=\P_x\big[w(\xi_{t \land \tau_B})\big] - \P_x\Big[\int_0^{t \land \tau_B} \psi_{\beta}(\xi_s, w(\xi_s)) \, ds\Big] \qquad \text{ for all } (x,t) \in \overline{B}\times [0,\infty).
\end{equation}
Since $w$ is bounded on $\overline{B}$, the continuity of the diffusion $\xi$ yields that $w$ is continuous on $B$ (see the argument in the last paragraph of page 708 in \cite{EngPin99}). Because $B$ was arbitrary, we conclude:

\begin{lemma}\label{lem:wcont}
The martingale function $w$ is continuous on $D$.
\end{lemma}

Lemma~\ref{lem:semi}(i) in the appendix shows that \eqref{eq:mildw} can be transformed into
\[
w(x)=\P_x\Big[w(\xi_{t \land \tau_B}) \exp\Big(-\int_0^{t \land \tau_B} \frac{\psi_{\beta}(\xi_s,w(\xi_s))}{w(\xi_s)} \, ds\Big)\Big] \quad \text{for all }  (x,t) \in \overline{B}\times [0,\infty).
\]
Hence, for any domain $B \subset \subset D$, $x\in B$,
\begin{equation}\label{eq:PwBmg}
w(\xi_{t \land \tau_B}) \exp\Big(-\int_0^{t \land \tau_B} \frac{\psi_{\beta}(\xi_s,w(\xi_s))}{w(\xi_s)} \, ds\Big), \quad t \ge 0,\quad \text{is a }\P_x\text{-martingale.}
\end{equation}
Since every nonnegative local martingale is a supermartingale, we conclude that for all $x\in D$,
\[
\frac{w(\xi_t)}{w(x)}\exp\Big(-\int_0^t \frac{\psi_{\beta}(\xi_s,w(\xi_s))}{w(\xi_s)} \, ds\Big),\quad t\ge 0, \quad \text{is a }\P_x\text{-supermartingale.}
\]
In particular, we deduce the following lemma.

\begin{lemma}\label{lem:PwWellDef}
For every $x \in D$, $\P_x^w$ is a well-defined (sub-)probability measure and $(\xi=(\xi_t)_{t \ge 0}; (\P_x^w)_{x \in D})$ is a (possibly non-conservative) Markov process, which we consider as a Markov process in $D\cup\{\cem\}$.
\end{lemma}
\noindent If $w$ is bounded, the argument leading to \eqref{eq:PwBmg} is valid for $B=D$ and $(\xi;\P^w)$ is conservative.
\smallskip

To give a description of the skeleton decomposition, we use the martingale function $w$ to construct an auxiliary subcritical $\M_f(D)$-valued Markov process. Let for all $x\in D$, $z \ge 0$, and $f\in p(D)$, $\Pi^*(x,dy):=e^{-w(x) y} \,\Pi(x,dy)$,
\begin{align}
\beta^*(x)&:=\beta(x)-2\alpha(x) w(x)-\int_{(0,\infty)} (1-e^{-w(x)y}) y \, \Pi(x,dy),\notag \\
\psi_0^*(x,z)&:=\alpha(x)z^2+\int_{(0,\infty)} (e^{-yz}-1+yz)\, \Pi^*(x,dy),\notag \\
S_t^*f(x)&:=\P_x\big[e^{\int_0^t \beta^*(\xi_s) \, ds} f(\xi_t)\big].\label{eq:defSstar}
\end{align}
Since $\beta^*(x) \le \beta(x)$ for all $x\in D$, $\beta^*$ is bounded from above. 
However, it is not clear whether $\beta^*$ is bounded from below. Hence, the branching mechanism $\psi_{\beta^*}^*(x,z)=-\beta^*(x)z+\psi_0^*(x,z)$ does not satisfy the assumptions from Section~\ref{sec:model}. 
The following lemma shows that the mild equation corresponding to $(L,\psi_{\beta^*}^*;D)$ still has a unique solution. 

\begin{lemma}\label{lem:mildStar}
For all $f \in bp(D)$, there exists a unique solution $u_f^*\in p(D\times [0,\infty))$ to
\begin{equation}\label{eq:mild_star}
u(x,t)=S_t^*f(x)-\int_0^t S_s^*[\psi_0^*(\cdot, u(\cdot,t-s)) ](x) \,ds\quad \text{for all }(x,t) \in D \times [0,\infty).
\end{equation}
\end{lemma}

The proof of Lemma~\ref{lem:mildStar} is deferred to Appendix~\ref{sec:beta_unb} (Lemmas~\ref{lem:mild_uni} and \ref{lem:mild_exist}).
Since $w$ is locally bounded according to \eqref{eq:wlocbd}, $\beta^*$ is locally bounded and for every domain $B \subset \subset D$, $(L,\psi_{\beta^*}^*;B)$ satisfies the assumptions of Section~\ref{sec:model}, where the motion is killed at the boundary of $B$. 
An $(L,\psi_{\beta^*}^*;D)$-superprocess $X^*$ can be obtained as a distributional limit of $(L,\psi_{\beta^*}^*;B)$-superprocesses using an increasing sequence of compactly embedded domains to approximate $D$ (see the argument of Lemma~A2 and Theorem~A1 in \cite{EngPin99} or before Corollary~6.2 in \cite{KypPerRen13pre} and our Lemma~\ref{lem:mild_exit_mono}). If $w(x)=-\log P_{\delta_x}(\ext)$ for a tail event $\ext$ with $P_{\mu}(\ext)=e^{-\pint{w,\mu}}$ for all $\mu \in \M_f(D)$, then $X^*$ can be obtained from $X$ by conditioning on $\ext$, i.e., the distribution of $X_t^*$ is given by $P_{\mu}(X_t \in \cdot\, |\, \ext)$, cf.\ \cite{EvaOCo94,EngPin99,BerKypMur11,KypPerRen13pre}.
\medskip

The following theorem is a concise
version of the skeleton decomposition at the level of detail that is
useful to us. It is based on a result from Kyprianou et al.\ \cite{KypPerRen13pre}. We denote by $\M_a^{{\rm loc}}$ the set of locally finite integer-valued measures on $\B(D)$. 

\begin{theorem}[{\bf{Kyprianou et al.~\cite{KypPerRen13pre}}}]\label{thm:KPR}
There exists a probability space with probability measures $\bbP_{\mu,\nu}$, $\mu \in \M_f(D)$, $\nu \in \M_a^{{\rm loc}}(D)$, that carries the following processes:
\begin{itemize}
\item[{\rm (i)}] $(Z=(Z_t)_{t \ge 0}; \bbP_{\mu,\nu})$ is a branching diffusion with motion $(\xi;\P^w)$ defined in \eqref{eq:Pwdef}, and branching rate $q$ and offspring distribution $(p_k)_{k \ge 2}$ defined by \eqref{eq:Generator} and $\bbP_{\mu,\nu}(Z_0 =\nu)=1$. 
\item[{\rm (ii)}] $(X^*=(X_t^*)_{t \ge 0};\bbP_{\mu,\nu})$ is a $\M_f(D)$-valued time-homogeneous Markov process such that for every $\mu \in \M_f(D)$, $f\in bp(D)$ and $t\ge 0$,
\[
\bbP_{\mu,\nu}\big[e^{-\pint{f,X_t^*}}\big]=e^{-\pint{u_f^*(\cdot,t),\mu}},
\]
where $u_f^*$ is the unique solution to \eqref{eq:mild_star}. Moreover, $X^*$ is independent of $Z$ under $\bbP_{\mu,\nu}$. 
\item[{\rm (iii)}] $(I=(I_t)_{t\ge 0};\bbP_{\mu,\nu})$ is a $\M_f(D)$-valued process such that
\begin{itemize}
\item[{\rm (a)}] $\bbP_{\mu,\sum_i \delta_{x_i}}[e^{-\la f, I_t\ra}]\!=\!\prod_i \bbP_{\mu, \delta_{x_i}}\big[e^{-\la f,I_t\ra}\big]
$ for all $\mu \in \M_f(D), x_i \in D, f \in p(D)$.
Moreover, $\bbP_{\mu,\nu}(I \in \cdot)$ does not depend on $\mu$, $\bbP_{\mu,\nu}(I_0=0)=1$, and, under $\bbP_{\mu,\nu}$, $(Z,I)$ is independent of $X^*$. 
\item[{\rm (b)}] $((X,Z):=(X^*+I,Z);\bbP_{\mu,\nu})$ is a Markov process. 
\item[{\rm (c)}]  $(X=X^*+I;\bbP_{\mu})$ is equal in distribution to $(X;P_{\mu})$, where $\bbP_{\mu}$ denotes the measure $\bbP_{\mu,\nu}$ with $\nu$ replaced by a Poisson random measure with intensity $w(x)\mu(dx)$.
\item[{\rm (d)}] Under $\bbP_{\mu}$, conditionally given $X_t$, the measure $Z_t$ is a Poisson random measure with intensity $w(x) X_t(dx)$.
\end{itemize}
\end{itemize}
\end{theorem}

We call the probability space from Theorem~\ref{thm:KPR} the \emph{skeleton space}. The process $I$ is called \emph{immigration process} or simply \emph{immigration}. 
As the processes $(X;\bbP_{\mu})$ on the skeleton space and $(X;P_{\mu})$ on the generic space have the same distribution, we may, without loss of generality, work on the skeleton space whenever it is convenient. 
Since the distributions of $X^*$ and $I$ under $\bbP_{\mu,\nu}$ do not depend on $\nu$ and $\mu$, respectively, we sometimes write $\bbP_{\mu, \bullet}$ or $\bbP_{\bullet, \nu}$.

Kyprianou et al.\ \cite{KypPerRen13pre} identify the immigration process explicitly. 
We need only the properties listed in Theorem~\ref{thm:KPR} but, for definiteness, we now give a full characterisation of the immigration process.
\smallskip

Dynkin and Kuznetsov \cite{DynKuz04} showed that on the canonical space of measure-valued c\`{a}dl\`{a}g functions $\mathbb{D}([0,\infty),\M_f(D))$ for every $x \in D$ there is a unique measure $\N_x$ such that for all $f \in bp(D)$, $t \ge 0$, 
\begin{equation}\label{eq:Nmeasure_def}
-\log P_{\delta_x}[e^{-\pint{f,X_t}}]=\N_x[1-e^{-\la f, X_t\ra}].
\end{equation}
We denote the $\N_x$-measures corresponding to the superprocess $X^*$ by $\N_x^*$.

To describe the immigration processes, we use the classical Ulam-Harris notation to uniquely refer to individuals in the genealogical tree $\tree$ of $Z$ (see for example page 290 in \cite{HarHar09}).
For each individual $u \in \tree$, we write $b_u$ and $d_u$ for its birth and death times, respectively, and $\{z_u(r)\colon r\in [b_u,d_u]\}$ for its spatial trajectory. 
The skeleton space carries the following processes:

\begin{itemize}
\item[{\rm (iii.1)}] $(\alongbb;\bbP_{\mu,\nu})$ is a random measure, such that conditional on $Z$, $\alongbb$ is a Poisson random measure that issues, for every $u \in \tree$, $\M_f(D)$-valued processes $X^{\alongbb,u,r}=(X_t^{\alongbb,u,r})_{t \ge 0}$ along the space-time trajectory $\{(z_u(r),r)\colon r \in (b_u,d_u]\}$ with rate 
\[
dr \times \Big( 2 \alpha(z_u(r)) d\N_{z_u(r)}^*+ \int_{(0,\infty)} \Pi(z_u(r),dy) \, ye^{-w(z_u(r)) y} \times dP_{y \delta_{z_u(r)}}^*\Big), 
\]
where $P_{\mu}^*$ denotes the distribution of $X^*$ started in $\mu$.
Since at most countable many processes $X^{\alongbb,u,r}$ are not equal to the constant zero measure, immigration at time $t$ that occurred in the form of processes $X^{\alongbb,u,r}$ until time $t$ can be written as
\[
I_t^{\alongbb}=\sum_{u \in \tree}\sum_{b_u <r \le d_u \land t} X_{t-r}^{\alongbb,u,r}.
\]
The processes $(X^{\alongbb,u,r}\colon u \in \tree, b_u < r\le d_u)$ are independent given $Z$ and independent of $X^*$.
\item[{\rm (iii.2)}] $(\bbpb;\bbP_{\mu,\nu})$ is a random measure, such that conditional on $Z$, $\bbpb$ issues, for every $u \in \tree$, at space-time point $(z_u(d_u),d_u)$ process $X^{\bbpb,u}$ with law $P_{Y_u \delta_{z_u(d_u)}}^*$.
Given that $u$ is replaced by $k$ particles at its death time $d_u$, the independent random variable $Y_u$ is distributed according to the measure
\[
\frac{1}{q(x)w(x)p_k(x)} \Big(\alpha(x)w(x)^2 \delta_0(dy)\mathbbm{1}_{\{k=2\}}+w(x)^k \frac{y^k}{k!} e^{-w(x)y} \, \Pi(x,dy)\Big) \Big|_{x=z_u(d_u)}.
\]
The immigration at time $t$ that occurred in the form of processes $X^{\bbpb,u}$ until time $t$ is denoted by
\[
I_t^{\bbpb}=\sum_{u \in \tree}\mathbbm{1}_{\{d_u \le t\}} X_{t-d_u}^{\bbpb,u}.
\]
The processes $(X^{\bbpb,u}\colon u \in \tree)$ are independent of $X^*$ and, given $Z$, are mutually independent and independent of $\alongbb$. 
\end{itemize}
The full immigration process is given by $I=I^{\alongbb}+I^{\bbpb}$. 

\begin{proof}[Proof of Theorem~\ref{thm:KPR}]
Theorem~\ref{thm:KPR} generalises Corollary~6.2 in \cite{KypPerRen13pre} in three ways. 
First, the authors choose $w(x)=-\log P_{\delta_x}(\extinction)$ but after defining $Z$ and $X^*$ this choice is not used anymore and their argument goes through without any changes for a general martingale function $w$ satisfying Assumption~\ref{as:skeleton}. Second, the authors assume that $w$ is locally bounded away from zero. 
Since $w$ is continuous by Lemma~\ref{lem:wcont}, this condition is automatically satisfied. 
Finally, Kyprianou et al.\ enforce additional regularity conditions on the underlying motion to use a comparison principle from the literature in the proof of their Lemma~6.1 (see also their Footnote~1). 
The comparison principle allows them to conclude that the solution $\widetilde{u}_f^B$ to \eqref{eq:mildExit} is pointwise increasing in the domain $B$ when the support of $f$ is a subset of $B$. 
Lemmas~\ref{lem:semi}(i) and~\ref{lem:mild_exit_mono} show that this monotonicity  holds in our more general setup, too.
\end{proof}

We introduce notation to refer to the different parts of the skeleton decomposition.

\begin{notation}[Notation for $Z$]\label{not:Z}
For $t \ge 0$, we write $Z_t=\sum_{i=1}^{N_t} \delta_{\xi_i(t)}$, where $N_t$ denotes the number of skeleton particles at time $t$ and $(\xi_i(t)\colon i=1,\ldots,N_t)$ their (conveniently ordered) locations.  
Given $Z_0$, $(Z^{i,0}\colon i=1,\ldots, N_0)$ denote the independent subtrees of the skeleton obtained by splitting $Z$ according to the ancestor at time $0$ and the Markov property implies that $Z^{i,0}$ follows the same distribution as $(Z;\bbP_{\bullet, \delta_{\xi_i(0)}})$, $i=1,\ldots, N_0$.
Under $\bbP_{\mu}$ with $\mu \in \M_c(D)$, $N_0=\pint{\mathbf{1},Z_0}$ is a Poisson random variable with mean $\pint{w,\mu}$.
\end{notation}

For $t\ge 0$, let $\bbF_t$ denote the $\sigma$-algebra generated by the processes $X^*$, $Z$ and $I$ up to time $t$.
Using the characterisation of the immigration process from Theorem~\ref{thm:KPR}, we obtain for all $\mu \in \M_f(D)$, $\nu \in \M_a^{{\rm loc}}(D)$, $f\in p(D)$ and $s,t\ge 0$,
\begin{equation}\label{eq:bbDistr}
\begin{split}
\bbP_{\mu,\nu}\big[e^{-\la f, X_{s+t}\ra}\big|\bbF_t\big] &\overset{\text{(b)}}{=} \bbP_{X_t,Z_t}\big[e^{-\la f,X_s\ra}\big]= \bbP_{X_t,Z_t}\big[e^{-\la f,X_s^*+I_s\ra}\big]\\
&\overset{\rm{(a)}}{=}\bbP_{X_t,\bullet}\big[e^{-\la f,X_s^*\ra}\big]\prod_{i=1}^{N_t}\bbP_{\bullet,\delta_{\xi_i(t)}}\big[e^{-\la f,I_s\ra}\big] \qquad \bbP_{\mu,\nu}\text{-almost surely}.
\end{split}
\end{equation}
Since under $\bbP_{\mu}$ and given $X_t$, $Z_t$ is a Poisson random measure with intensity $w(x) X_t(dx)$ by (d), \eqref{eq:bbDistr} holds $\bbP_{\mu}$-almost surely when $\bbP_{\mu,\nu}$ on the left-hand side is replace by $\bbP_{\mu}$. 
To make use of this identity, we split the immigration process according to the immigration that occurred before time $t$ and the immigration that occurred along different branches of $Z$ after time $t$.

\begin{notation}[Notation for $I$]\label{not:I}
For $t\ge 0$, denote by $I_s^{*,t}$ the immigration at time $s+t$ that occurred along the skeleton before time $t$; $I^{*,t}=(I_s^{*,t})_{s\ge 0}$.
In addition, for $i\in \{1,\ldots, N_t\}$, let $I_s^{i,t}$ denote the immigration at time $s+t$ that occurred along the subtree of the skeleton rooted at the $i$-th particle at time $t$ with location $\xi_i(t)$; $I^{i,t}=(I_s^{i,t})_{s\ge 0}$. We have
\begin{equation}\label{eq:bb}
X_{s+t}=X_{s+t}^*+ I_s^{*,t}+ \sum_{i=1}^{N_t} I_s^{i,t} \qquad \text{for all }s,t\ge 0.
\end{equation}
According to \eqref{eq:bbDistr} and by the Markov property, given $\bbF_t$, $(X_{s+t}^*+I_s^{*,t})_{s\ge 0}$ follows the same distribution as $(X^*,\bbP_{X_t})$ and $I^{i,t}$ follows the same distribution as $(I;\bbP_{\bullet, \delta_{\xi_i(t)}})$, $i=1,\ldots, N_t$. Moreover, given $\bbF_t$, the processes $(I^{i,t}\colon i=1,\ldots, N_t)$ are independent and independent of $I^{*,t}$.
\end{notation}

We end this section with a note on terminology.
Several different phrases have been used in the literature to refer to the skeleton decomposition. Evans and O'Connell \cite{EvaOCo94} proved the first skeleton decomposition for supercritical superprocesses in the case of a conservative motion (not necessarily a diffusion) and a quadratic, spatially independent branching mechanism with $\alpha,\beta \in (0,\infty)$, and call the result ``representation theorem''. 
Their study was motivated by the ``immortal particle representation'' derived by Evans \cite{Eva93} for critical superprocesses conditioned on non-extinction. 
This representation is in terms of a single ``immortal particle'' that throws off pieces of mass. Evans' article is part of a cluster of papers that study conditioned superprocesses. Salisbury and Verzani \cite{SalVer99} condition the exit measure of a super-Brownian motion to hit $n$ fixed, distinct points on the boundary of a bounded smooth domain. 
The authors show that the resulting process can be described as the sum of a tree with $n$ leaves that throws off mass in a Poissonian way and of a copy of the unconditioned process, and call this decomposition ``backbone representation''. 
In a follow-up article \cite{SalVer00} they consider different conditionings and derive an ``immortal particle description'' where the guiding object is a tree with possibly infinitely many branches that they call ``backbone'' or ``branching backbone''. 
Salisbury and Sezer \cite{SalSez13} describe the super-Brownian motion conditioned on boundary statistics in terms of a ``branching backbone'' or ``branching backbone system''. 
Etheridge and Williams \cite{EthWil03} represent a critical superprocess with infinite variance conditioned to survive until a fixed time as immigration along a Poisson number of ``immortal trees''. 
An overview of decompositions of conditioned superprocesses was offered by Etheridge \cite{Eth00} using the names ``skeleton'' and ``immortal skeleton''. 
Back in our setup of supercritical superprocesses, Engl\"ander and Pinsky \cite{EngPin99} speak about a ``decomposition with immigration'' and Fleischmann and Swart \cite{FleSwa04} construct a ``trimmed tree''. 
For the analysis of continuous-state branching processes, Duquesne and Winkel \cite{DuqWin07} find a ``Galton-Watson forest''.
In the corresponding superprocess setup, Berestycki et al.\ \cite{BerKypMur11} identify the ``prolific backbone'' and call the representation itself a ``backbone decomposition''. 
The latter phrase has been used several times since \cite{KypRen12,KypPerRen13pre,Mil13pre,RenSonZha13pre}. 

We decided to use the term ``skeleton decomposition''  for the following reasons. Since the words ``backbone'' and ``spine'' are used interchangeably
in spoken English,
using these two words to mean different things might cause confusion. 
Furthermore, spine/backbone describes one key, supporting element of an object and does not branch. 
In contrast, a skeleton carries the entire structure and determines the main features of an object. This is the correct intuition for the spine decomposition and the skeleton decomposition of branching processes as well as the distinction between them.

\subsubsection{Product $L^1$-criticality}\label{sec:ProL1}

The first two moments of the superprocess can be expressed in terms of the underlying motion and the branching mechanism. That is, (see for instance Proposition 2.7 in \cite{Fit88}) for all $\mu \in \M_f(D)$ and $f \in bp(D)$,
\begin{align}
P_{\mu}[\la f,X_t\ra]&=\la S_tf,\mu\ra,\label{eq:opp}\\
\text{Var}_{\mu}\big(\la f,X_t\ra\big)&=\int_0^t \big\la S_s\Big[ \big(2 \alpha +\int_{(0,\infty)} y^2 \, \Pi(\cdot, dy)\big) (S_{t-s}f)^2\Big] , \mu \big\ra \,ds.\label{eq:variance}
\end{align}
Here $\text{Var}_{\mu}(\la f,X_t\ra)$ denotes the variance of $\la f,X_t\ra$ under $P_{\mu}$. By the monotone convergence theorem, the boundedness of $f$ in \eqref{eq:opp} is unnecessary and \eqref{eq:variance} holds for $f \in p(D)$ as soon as $\la S_tf,\mu\ra<\infty$.
Similarly, under Assumption~\ref{as:skeleton} and for $\mu \in \M_f(D)$, $f \in bp(D)$, the first two moments of $\pint{f,X_t^*}$ (see the discussion around Lemma~\ref{lem:mildStar} for the definitions) can be expressed as,
\begin{align}
&\bbP_{\mu}[\la f,X_t^*\ra]=\la S_t^*f,\mu\ra ,\label{eq:opp_star}\\
&{\rm{\bf{Var}}}_{\mu}\big(\la f,X_t^*\ra\big)=\int_0^t \big\la S_s^*\Big[ \big(2 \alpha +\int_{(0,\infty)} y^2 \, \Pi^*(\cdot, dy)\big) (S^*_{t-s}f)^2\Big] , \mu \big\ra \,ds.\label{eq:variance_star}
\end{align}

The main purpose of this section is to discuss Assumption~\ref{as:criticality}, that enforces conditions on the operator $L+\beta$ and consequently on its semigroup $(S_t)_{t\ge 0}$ which is the expectation semigroup of $X$ by \eqref{eq:opp}. 
Throughout the section, we suppose that Assumptions~\ref{as:skeleton} and \ref{as:criticality} hold. Key features of the local behaviour of the superdiffusion $X$ are determined by the generalised principal eigenvalue $\lambda_c=\lambda_c(L+\beta)$. 
If $\alpha$ and $\Pi$ are sufficiently smooth and $\lambda_c\le 0$, then the superdiffusion exhibits \emph{weak local extinction}, i.e.\ the total mass assigned to a compact set by the superprocess tends to zero.
For quadratic branching mechanisms this was shown by Pinsky \cite[Theorem 6]{Pin96}; for general branching mechanisms the proof of Theorem~3(i) in \cite{EngKyp04} gives the result. 
This is the reason to assume $\lambda_c>0$.
\smallskip

The assumption of product $L^1$-criticality restricts this article to the situation where the expectation semigroup $(S_t)_{t \ge 0}$ scales precisely exponentially on compactly supported, continuous functions. 
In general, writing $S_tf(x)=e^{\lambda_c t} \omega_{f,x}(t)$, the limit $\omega_{f,x}:=\lim_{t \to \infty}\omega_{f,x}(t)$ exists for all $f\in C_c^+(D)$, $x\in D$. 
Product $L^1$-criticality is equivalent to $\omega_{f,x}
>0$ for all $f\not=\mathbf{0}$. The alternative is $\omega_{f,x}=0$ for all $f$ and $x$ (cf.\ Theorem~7 in \cite{Pin96} and Appendix~A in \cite{EngWin06}). Some of the relevant literature for this regime was discussed in Section~\ref{sec:literature}. 
The notion of product $L^1$-criticality comes from the criticality theory of second order elliptic operators. 
See Appendix B of \cite{EngPin99} for a good summary and Chapter~4 in \cite{Pin95} for a comprehensive treatment.
\smallskip

By Theorem~4.8.6 in \cite{Pin95}, criticality implies that the ground state $\phi$ is an invariant function of $e^{-\lambda_c t}S_t$, that is $e^{-\lambda_c t}S_t\phi =\phi$, and we define a conservative diffusion $(\xi=(\xi_t)_{t\ge 0};(\P_x^{\phi})_{x \in D})$ by
\begin{equation}\label{eq:Sphi}
\frac{d\P_x^{\phi}}{d\P_x}\bigg|_{\sigma(\xi_s\colon s \le t)} =\frac{\phi(\xi_t)}{\phi(x)} e^{\int_0^t (\beta(\xi_s)-\lambda_c) \, ds} \quad \text{on }\{t<\tau_D\}, \quad \P_x^{\phi}[g(\xi_t)]=\phi(x)^{-1} e^{-\lambda_c t}S_t[\phi g](x), 
\end{equation}
for all $x \in D$, $t \ge 0$. Product $L^1$-criticality is equivalent to $(\xi=(\xi_t)_{t\ge 0}; (\P_x^{\phi})_{x \in D})$ being a positive recurrent diffusion with stationary distribution $\phi(x) \tphi(x) \, dx$ and we call it the \emph{ergodic motion} or the \emph{spine} (as we explain in Section~\ref{sec:spine}). In particular, see Theorems~4.3.3 and~4.8.6 in \cite{Pin95},
\begin{equation}\label{eq:statDistrSphi}
\la \P_{\cdot}^{\phi}[g(\xi_t)],\phi \tphi\ra=\la g,\phi\tphi\ra \qquad \text{ for all }g \in p(D),
\end{equation}
and for every probability measure $\pi$ on $D$ and $g \in bp(D)$,
\begin{equation}\label{eq:ergodic}
\pint{ \P_{\cdot}^{\phi}[g(\xi_t)], \pi} \to \la g,\phi \tphi\ra \qquad \text{as } t \to \infty.
\end{equation}
If, in addition, the initial distribution $\pi$ is of compact support, then \eqref{eq:ergodic} holds for all $g \in p(D)$ with $\pint{g,\phi \tphi}<\infty$. 
For $g$ bounded, \eqref{eq:ergodic} follows from Theorem 4.9.9 in \cite{Pin95} and the dominated convergence theorem. 
If the support of $\pi$, $\text{supp}(\pi)$, is compactly embedded in $D$, choose a domain $B\subset \subset D$ with $\text{supp}(\mu) \subseteq B$. 
There exists a constant $C>0$ such that 
\begin{equation}\label{eq:density_bd}
p^{\phi}(x,y,t) \le C \phi(y) \tphi(y) \qquad \text{for all }x\in B, y \in D, t>1, 
\end{equation}
where $p^{\phi}(x,y,t)$ denotes the transition density of $(\xi,\P^{\phi})$ and $\lim_{t \to \infty} p^{\phi}(x,y,t)=\phi(y) \tphi(y)$ for every $x,y \in D$ (cf.\ Pinchover \cite[(2.12) and Theorem~1.3(ii)]{Pin13}). 
Hence, \eqref{eq:ergodic} for $\pi \in \M_c(D)$ and $g \in p(D)$ with $\pint{g,\phi \tphi}<\infty$ follows from the dominated convergence theorem.

\begin{lemma}[Many-to-One Lemma for $X$ and $Z$]\label{lem:mto}
For all $\mu \in \M_f(D)$, $\nu \in \M_a^{{\rm loc}}(D)$ and $g\in p(D)$,
\begin{align}
&e^{-\lambda_c t} P_{\mu}[\la \phi g,X_t\ra]= \la \P_{\cdot}^{\phi}[g(\xi_t)] ,\phi\mu\ra, \label{eq:mtoX}\\
&e^{-\lambda_c t} \bbP_{\bullet,\nu}\Big[\Big\la \frac{\phi}{w}g,Z_t\Big\ra\Big]=\Bigpint{ \P_{\cdot}^{\phi}[g(\xi_t)], \frac{\phi}{w} \nu }\label{eq:oppZ}\\
&e^{-\lambda_c t} \bbP_{\mu}\Big[\Big\la \frac{\phi}{w}g,Z_t\Big\ra\Big]=\big\la \P_{\cdot}^{\phi}[g(\xi_t)],\phi \mu \big\ra.\label{eq:mtoZ}
\end{align} 
\end{lemma}

\begin{proof}
Identity \eqref{eq:mtoX} follows immediately from \eqref{eq:opp} and \eqref{eq:Sphi}. For \eqref{eq:oppZ}, notice that by \eqref{eq:Generator} the local growth rate of $Z$ is given by
\[
\beta^Z(x):=q(x) \Big(\sum_{k=2}^{\infty} kp_k(x)-1\Big)=\partial_s \ZGen (x,s)\big|_{s=1}= \frac{\psi_0(x,w(x))}{w(x)} \qquad \text{for all }x\in D
\]
and, using the definition of $\P_x^w$ in \eqref{eq:Pwdef}, we obtain for all $x\in D$,
\begin{align*}
\P_x^w\Big[e^{\int_0^t \beta^Z(\xi_s) \, ds}\frac{\phi(\xi_t)}{w(\xi_t)}g(\xi_t)\Big]&=w(x)^{-1}\P_x\Big[\exp\Big(\int_0^t \big(\beta^Z(\xi_s)-\frac{\psi_{\beta}(\xi_s,w(\xi_s))}{w(\xi_s)}\big)\, ds\Big) \phi(\xi_t)g(\xi_t) \Big]\\
&=w(x)^{-1}S_t[\phi g](x)
\overset{\eqref{eq:Sphi}}{=}\frac{\phi(x)}{w(x)} e^{\lambda_c t}\P_x^{\phi}[g(\xi_t)].
\end{align*}
Hence, the first moment formula for branching diffusions (see for example Theorem 8.5 in \cite{HarHar09}) yields
\[
\bbP_{\bullet,\nu}\Big[e^{-\lambda_c t} \Big\la \frac{\phi}{w} g,Z_t\Big\ra \Big]=e^{-\lambda_c t} \Big\la \P_{\cdot}^w\Big[ e^{\int_0^t \beta^Z(\xi_s) \, ds} \frac{\phi(\xi_t)}{w(\xi_t)} g(\xi_t)\Big], \nu\Big\ra=\Bigpint{ \P_{\cdot}^{\phi}[g(\xi_t)], \frac{\phi}{w} \nu }.
\]
Since under $\bbP_{\mu}$, initial configuration of $Z$ is given by a Poisson random measure with intensity $w(x)\mu(dx)$, \eqref{eq:mtoZ} follows from \eqref{eq:oppZ}.
\end{proof}

We record the following consequence of Lemma~\ref{lem:mto}.
\begin{corollary}\label{cor:mg}
For all $\mu \in \M_f^{\phi}(D)$, $((W_t^{\phi}(X))_{t \ge 0};P_{\mu})$ and $((W_t^{\phi/w}(Z))_{t \ge 0};\bbP_{\mu})$ are martingales with 
\[
P_{\mu}\big[W_t^{\phi}(X)\big]=\bbP_{\mu}\big[W_t^{\phi/w}(Z)\big]=\pint{\phi,\mu} \quad \text{ for all }t\ge 0.
\]
\end{corollary}

\begin{proof}
Since $(\xi,\P_x^{\phi})$ is conservative, the formula for the expectations follows immediately from \eqref{eq:mtoX} and \eqref{eq:mtoZ}. The Markov property of $X$ combined with \eqref{eq:mtoX} gives the claim for $X$. 
The Markov property of $Z$ and \eqref{eq:oppZ} imply that $(W_t^{\phi/w}(Z))_{t\ge 0}$ is a $\bbP_{\bullet,\nu}$-martingale for all $\nu \in \M_a^{{\rm loc}}(D)$ with $\pint{\phi/w,\nu}<\infty$.
Replacing $\nu$ by a Poisson random measure with intensity $w(x)\mu(dx)$ concludes the proof.
\end{proof}

Let $\mu \in \M_f^{\phi}(D)$, $\mu \not\equiv 0$. After dividing the right-hand side of \eqref{eq:mtoX} by $\la \phi, \mu\ra$, the expression can be interpreted as the expectation of $g(\xi_t)$, where $\xi$ is the ergodic motion with start point randomised according to $\frac{\phi \mu}{\la \phi,\mu\ra}$. With this motivation, we define for all measurable sets $A$,
\begin{equation}\label{eq:Pphimu}
\P_{\phi\mu}^{\phi}(A):=\frac{1}{\la \phi,\mu\ra} \big\la \P_{\cdot}^{\phi}(A), \phi\mu \big\ra.
\end{equation}

We end this section with a remark for the case that the superprocess is deterministic.

\begin{remark}\label{rem:detCase}
If $\ell(\{x \in D\colon \alpha(x)+\Pi(x,(0,\infty))>0\})=0$, then \eqref{eq:variance} implies that $\pint{f,X_t}=\pint{S_tf,\mu}$ for all $t\ge 0$, $P_{\mu}$-almost surely, for all continuous $f\in bp(D)$. Hence, Assumption~\ref{as:skeleton} cannot be satisfied. 
However, under Assumption~\ref{as:criticality}, the conclusion of Theorem~\ref{thm:SLLN} still holds since for all $\mu \in \M_f^{\phi}(D)$ and $\ell$-almost everywhere continuous $f \in p(D)$ with $f/\phi$ bounded, a standard approximation, \eqref{eq:Sphi} and \eqref{eq:ergodic} imply that $P_{\mu}$-almost surely, as $t \to \infty$,
\[
e^{-\lambda_c t} \la f,X_t\ra=e^{-\lambda_c t}\pint{S_tf,\mu}=\la \P_{\cdot}^{\phi}[f(\xi_t)/\phi(\xi_t)],\phi\mu\ra \to \la f/\phi,\phi\tphi\ra \la \phi,\mu\ra=\pint{f,\tphi}\mglX.
\] 
\end{remark}

\subsubsection{Moment conditions}\label{sec:moments}

In this section, we discuss Assumption~\ref{as:moment1} and compare it to the conditions used in the literature. 
We work under Assumptions~\ref{as:skeleton} and \ref{as:criticality}. 
While Assumption~\ref{as:moment1} seems to be the most useful set of conditions, we prove our results under the following weaker moment assumption.

\begin{customassump}{3'}\label{as:moment}
There exists $p\in (1,2]$, $\varphi_1, \varphi_2\in p(D)$, $\sigma_1,\sigma_2,\sigma_3 \in [p,2]$ and $j_1,j_2 \in \{0,1\}$ such that,
\begin{align}
& \sup_{x \in D} \phi(x)^{\sigma_1-1} \alpha(x) <\infty \label{alpha_cond}\\
&\sup_{x \in D} \phi(x)^{\sigma_2-1} \int_{(0,\varphi_1(x)]} y^{\sigma_2}\, \Pi(x, dy) <\infty,\label{as:PiLowerTail}\\
&\sup_{x \in D} \phi(x)^{\sigma_3-1} \int_{(\varphi_1(x),\infty)} y^{\sigma_3}\, \Pi(x,dy)<\infty,\label{as:PiUpperTail}\\
&\pint{\phi^{p-1},\phi \tphi}<\infty,\label{phi_p_cond}\\
&\Bigpint{ \phi^{j_1} \int_{(0,\varphi_2(\cdot)]} y^2 e^{-w(\cdot)y} \, \Pi(\cdot,dy),\phi \tphi}<\infty\label{as:lat_con_Lower},\\
&\Bigpint{ \phi^{j_2} \int_{(\varphi_2(\cdot),\infty)} y^2 e^{-w(\cdot)y} \, \Pi(\cdot,dy),\phi \tphi}<\infty\label{as:lat_con_Upper}.
\end{align}
\end{customassump}

Assumption~\ref{as:moment1} is the special case $\varphi_1=\varphi_2=\mathbf{1}$, $\sigma_1=\sigma_2=2$, $\sigma_3=p$ and $j_1=j_2=0$ of Assumption~\ref{as:moment}. 
Notice that with this choice, Condition~\eqref{as:lat_con_Lower} trivially holds since $\pint{\phi,\tphi}<\infty$ and $x \mapsto \int_{(0,1]} y^2 \, \Pi(x,dy)$ is a bounded function by the model assumptions in Section~\ref{sec:model}. 
Therefore, the following theorem generalises Proposition~\ref{pro:mglim} and Theorems~\ref{thm:SLLN} and \ref{thm:L1}.

\begin{theorem}\label{thm:SLLN_gen}
Suppose Assumptions~\ref{as:skeleton}, \ref{as:criticality}, \eqref{alpha_cond}--\eqref{as:PiUpperTail} hold and $\mu \in \M_f^{\phi}(D)$. 
\begin{itemize}
\item[{\rm (i)}] The martingales $(W_t^{\phi}(X))_{t\ge 0}$ and $(W_t^{\phi/w}(Z))_{t\ge 0}$ are bounded in $L^p(\bbP_{\mu})$ and $\mglX=\mglZ$ $\bbP_{\mu}$-almost surely.
\item[{\rm (ii)}]
Suppose that, in addition, \eqref{phi_p_cond} holds. For all $f\in p(D)$ with $f/\phi$ bounded, we have in $L^1(P_{\mu})$
\begin{equation}\label{eq:SLLN2}
\lim_{t\to \infty} e^{-\lambda_c t} \pint{f,X_t} = \pint{f,\tphi} \mglX.
\end{equation}
\item[{\rm (iii)}]
If, in addition, Assumptions~\ref{as:moment} and \ref{as:SLLN} hold, then there exists a measurable set $\Omega_0$ with $P_{\mu}(\Omega_0)=1$ and, on $\Omega_0$, for all $\ell$-almost everywhere continuous functions $f \in p(D)$ with $f/\phi$ bounded, the convergence in \eqref{eq:SLLN2} holds.
\end{itemize}
\end{theorem}

The first three moment conditions, \eqref{alpha_cond1}--\eqref{as:PiUpperTail1} or \eqref{alpha_cond}--\eqref{as:PiUpperTail},
are used to guarantee that the martingale $(W_t^{\phi}(X))_{t \ge 0}$ is bounded in $L^p$ (see Theorem~\ref{thm:Lp} below). 
To the best of our knowledge, even though these conditions may not be optimal, they are the best conditions obtained so far to guarantee $L^p$-boundedness, $p \in (1,2)$, for general superprocesses. 
For the case of a super-Brownian motion, similar conditions were found in \cite{KypMur13}. Condition~\eqref{alpha_cond1} appeared as the main moment assumption in \cite{EngTur02} and \cite{EngWin06} to establish the convergence \eqref{eq:SLLN2} in distribution and in probability, respectively. 
The two articles that study almost sure convergence in the product $L^1$-critical regime (i.e.\ under Assumption~\ref{as:criticality}) are by Chen et al.\ \cite{CheRenWan08} and Liu et al. \cite{LiuRenSon13}. 
In both papers, $\alpha$ and $\phi$ are bounded; hence, \eqref{alpha_cond1} holds.

The article \cite{CheRenWan08} is restricted to quadratic branching mechanisms, i.e.\ $\Pi\equiv 0$, and \eqref{as:PiLowerTail1}--\eqref{as:PiUpperTail1} are trivially satisfied. 
Liu et al.\ \cite{LiuRenSon13} do not require $\Pi$ to have a $p$-th moment. 
The authors show that under their assumptions ($D$ of finite Lebesgue measure and $(S_t)_{t \ge 0}$ intrinsically ultracontractive), the martingale limit $\mglX$ is nontrivial if and only if $\pint{\int_{(1,\infty)} y \log y \, \Pi(\cdot, y/\phi), \tphi}<\infty$, and they establish their result under this condition. In the alternative case, the martingale limit is zero almost surely and the stated convergence \eqref{eq:SLLN} holds trivially.

The fourth assumption, \eqref{phi_p_cond1} or \eqref{phi_p_cond}, is a technical condition. It is only used in Proposition~\ref{pro:BC_immi} to compare the immigration after a large time $t$, $\sum_{i=1}^{N_t} \pint{f,I_s^{i,t}}$, to its expectation $\sum_{i=1}^{N_t} \bbP_{\mu}[\pint{f,I_s^{i,t}}|\bbF_t]$. 
In the previous articles on the SLLN \cite{CheRenWan08,LiuRenSon13}, Assumption \eqref{phi_p_cond1} holds since $\phi$ is bounded.

The technical condition can be avoided using an $h$-transform.
The $h$-transform for measure-valued diffusions was introduced by Engl\"ander and Pinsky in \cite{EngPin99}.
For $h \in C^{2,\eta}(D)$, $h>0$, let
\begin{equation}\label{eq:hdef}
L_0^h= L+a \frac{\nabla h}{h} \cdot \nabla, \quad \beta^h(x)=\frac{(L+\beta)h(x)}{h(x)},\quad \psi_0^h(x,z)=\frac{\psi_0(x,h(x)z)}{h(x)}.
\end{equation}
If $\beta^h$, $\alpha h$ and $x \mapsto \int_{(0,\infty)} (y \land h(x) y^2) \, \Pi(x,dy)$ belong to $b(D)$, then $\psi_{\beta^h}^h(x,z):=-\beta^h(x) z+\psi_0^h(x,z)$ satisfies the assumptions from Section~\ref{sec:model}. We denote the space of such functions $h$ by $\mathbb{H}(\psi_{\beta})$. 
An $(L_0^h, \psi_{\beta^h}^h;D)$-superprocess $X^h$ started in $h(x) \mu(dx)$ can be obtained from an $(L,\psi;D)$-superprocess $X$ started in $\mu$ by setting $X_t^h(dx):=h(x)X_t(dx)$. 
The result follows immediately from a comparison of the Laplace transforms using the mild equation \eqref{eq:mild} and Corollary~4.1.2 in \cite{Pin95}; see \cite{EngPin99} for the computation in the quadratic case. In the following, we superscript all quantities derived from $X^h$ with an $h$.
Clearly, the $(L,\psi;D)$-superprocess can be recovered from the $(L_0^h,\psi_{\beta^h}^h;D)$-superprocess by a transform with $1/h$.

\begin{lemma}\label{lem:SLLN_h}
Let $h \in \mathbb{H}(\psi_{\beta})$ and $\mu \in \M_f^{\phi}(D)$.
\begin{itemize}
\item[{\rm (i)}] The operator $L_0^h +\beta^h$ satisfies Assumption~\ref{as:criticality} with $\phi^h=\phi/h$, $\tphi^h=\tphi h$ and $\lambda_c^h=\lambda_c$, and the process $(W_t^{\phi^h}(X^h)=e^{-\lambda_c^h t} \pint{\phi^h,X_t^h}\colon t\ge 0;P_{h \mu}^h)$ is a martingale with almost sure limit $W_{\infty}^{\phi^h}(X^h)$.

\item[{\rm (ii)}] Suppose \eqref{eq:SLLN2} holds $P_{\mu}$-almost surely for some $f \in p(D)$. Then
\begin{equation}\label{eq:SLLNXh}
\lim_{t \to \infty} e^{-\lambda_c^h t} \la f/h,X_t^h\ra =\la f/h,\tphi^h\ra W_{\infty}^{\phi^h}(X^h) \qquad P_{h\mu}^h\text{-almost surely.}
\end{equation}
If \eqref{eq:SLLN2} holds in $L^1(P_{\mu})$ instead, then \eqref{eq:SLLNXh} holds in $L^1(P_{h \mu}^h)$.
\end{itemize}
\end{lemma}

\begin{proof}
The first part of the claim was proved by Pinsky \cite[Chapter 4]{Pin95}. Setting $X^h:=h X$, we immediately obtain $W_t^{\phi^h}(X^h)=e^{-\lambda_c^ht}\la \phi^h,X_t^h\ra= W_t^{\phi}(X)$ and, $P_{\mu}$-almost surely (in $L^1(P_{\mu})$, respectively),
\[
e^{-\lambda_c^h t}\la f/h, X_t^h\ra =e^{-\lambda_c t}\la f, X_t\ra \to \la f,\tphi\ra W_{\infty}^{\phi}(X)=\la f/h, \tphi^h\ra W_{\infty}^{\phi^h}(X^h) \qquad\text{as }t\to\infty.\qedhere
\]
\end{proof}

Lemma~\ref{lem:SLLN_h} states that Assumption~\ref{as:criticality} and our results are invariant under $h$-transforms. The same is true for Assumption~\ref{as:skeleton} and \ref{as:SLLN}.

\begin{lemma}\label{lem:invar}
Let $h \in \mathbb{H}(\psi_{\beta})$. The $(L_0^h,\psi_{\beta^h}^h;D)$-superprocess $X^h$ satisfies Assumption~\ref{as:skeleton} with martingale function $\smash{w^h=w/h}$ and the distribution of the skeleton $\smash{Z^h}$ under $\smash{\bbP_{h \mu}^h}$ agrees with the distribution of $Z$ under $\bbP_{\mu}$ for all $\mu \in \M_c(D)$. 
In particular, if $X$ satisfies Assumption~\ref{as:SLLN}, then $X^h$ satisfies Assumption~\ref{as:SLLN}. 
\end{lemma}

\begin{proof}
The claim follows immediately from the definitions.
\end{proof}

Exploiting the invariance under $h$-transforms, we can prove our main results under the following moment assumption.

\begin{customassump}{3''}\label{as:moment_h}
There exists $p\in (1,2]$ such that Conditions~\eqref{alpha_cond1}--\eqref{as:PiUpperTail1} and \eqref{as:lat_con_Upper} for $j_2=1$, $\varphi_2=\mathbf{1}$ hold and
\begin{align}
\sup_{ x\in D} \int_{(1/\phi(x),\infty)} y \, \Pi(x,dy) <\infty. \label{as:h_trans} 
\end{align}
\end{customassump}

Crucially, Assumption~\ref{as:moment_h} does not require $\pint{\phi^p,\tphi}<\infty$. In the case of a quadratic branching mechanism, only boundedness of $\phi\alpha$ is needed. 
Even though our main assumptions and results are invariant under $h$-transforms, our setup is not, and Condition~\eqref{as:h_trans} is needed to guarantee $\phi \in \mathbb{H}(\psi_{\beta})$.

\begin{theorem}\label{thm:SLLN_h}
Suppose Assumptions~\ref{as:skeleton}, \ref{as:criticality}, \eqref{alpha_cond1}--\eqref{as:PiUpperTail1} and \eqref{as:h_trans} hold and $\mu \in \M_f^{\phi}(D)$. 
\begin{itemize}
\item[{\rm (i)}]
For all $f \in p(D)$ with $f/\phi$ bounded, the convergence in \eqref{eq:SLLN2} holds in $L^1(P_{\mu})$.
\item[{\rm (ii)}]
If, in addition, Assumptions~\ref{as:moment_h} and \ref{as:SLLN} hold, then there exists a measurable set $\Omega_0$ with $P_{\mu}(\Omega_0)=1$ and, on $\Omega_0$, for all $\ell$-almost everywhere continuous functions $f \in p(D)$ with $f/\phi$ bounded, the convergence in \eqref{eq:SLLN2} holds.
\end{itemize}
\end{theorem}

\begin{proof}
Part~(i): since $\beta^{\phi}=\lambda_c$, $\alpha^{\phi}=\phi \alpha$ and $\Pi^{\phi}(x,dy)=\frac{1}{\phi(x)}\Pi(x,dy/\phi(x))$, Conditions~\eqref{alpha_cond1}, \eqref{as:PiLowerTail1}, \eqref{as:h_trans} and the model assumptions in Section~\ref{sec:model} guarantee that $\phi \in \mathbb{H}(\psi_{\beta})$. 
By Lemma~\ref{lem:SLLN_h}(i), $X^{\phi}$ satisfies Assumption~\ref{as:criticality} and $\phi^{\phi}=\mathbf{1}$. 
Thus, \eqref{alpha_cond1}--\eqref{as:PiUpperTail1} imply that $X^{\phi}$ satisfies~\eqref{alpha_cond}--\eqref{phi_p_cond} with $\varphi_1=\phi$, $\sigma_2=2$, $\sigma_3=p$ and $\sigma_1 \in [p,2]$ arbitrary. 
Using Lemma~\ref{lem:invar}, we deduce that Theorem~\ref{thm:SLLN_gen}(i) applies to $X^{\phi}$ and the claim follows from Lemma~\ref{lem:SLLN_h}(ii).

Part~(ii): $X^{\phi}$ satisfies \eqref{as:lat_con_Lower}--\eqref{as:lat_con_Upper} with $\varphi_2=\phi$ and arbitrary $j_1,j_2 \in \{0,1\}$, and Assumption~\ref{as:SLLN} by Lemma~\ref{lem:invar}. 
Hence, Theorem~\ref{thm:SLLN_gen}(iii) applies to $X^{\phi}$ and Lemma~\ref{lem:SLLN_h}(ii) completes the proof for fixed functions $f$. 
The existence of a common set $\Omega_0$ will be proved in Lemma~\ref{lem:omega0} below.
\end{proof}

Engl\"ander and Winter \cite{EngWin06} proved the convergence \eqref{eq:SLLN2} in probability under the assumption of a quadratic branching mechanisms and \eqref{alpha_cond1}. 
Their argument can easily be extended to general branching mechanisms. 
Since the proof relies on an $h$-transform with $h=\phi$ and second moment estimates, the additional conditions needed for this generalisation are \eqref{as:PiLowerTail1}, \eqref{as:PiUpperTail1} and \eqref{as:h_trans} with $p=2$. 

The freedom to choose $p\in (1,2]$ allows us to analyse processes where $(W_t^{\phi}(X))_{t \ge 0}$ is bounded in $L^p$ for $p \in (1,2)$ but not in $L^2$. 
Examples of such processes are given in Section~\ref{sec:examples}. 
In these cases, not only our almost sure convergence result is new but also the implied convergence in probability result.
The main tool to deal with non-integer moments is a spine decomposition presented in Section~\ref{sec:spine} and we are not aware of any other way to obtain these conditions.

The final conditions \eqref{as:lat_con_Lower}--\eqref{as:lat_con_Upper} simplify to \eqref{cond:lat_cont1} in the case $j_1=j_2=0$, $\varphi_2=\mathbf{1}$. 
These assumptions guarantee that the process $X^*$ from the skeleton decomposition has finite second moments \eqref{eq:variance_star}, a fact which is only used in the transition from lattice to continuous times. 
In particular, the SLLN along lattice times in Theorem~\ref{SLLN_lattice} holds without it.
If $w$ is bounded away from zero, for instance when the branching mechanism is spatially independent and the motion is conservative (see Section~\ref{sec:nonspatial}), then \eqref{cond:lat_cont1} holds automatically.
Since Chen et al.~\cite{CheRenWan08} consider a quadratic branching mechanism, the conditions automatically hold in their article. In contrast, Liu et al.~\cite{LiuRenSon13} have no conditions of this type.
\smallskip

In summary, our moment conditions are weaker than those used in \cite{CheRenWan08} but compared to \cite{LiuRenSon13}, we imposed stricter assumptions on the L\'{e}vy measure $\Pi$, yet allow a much larger class of underlying motions $\xi$ and domains $D$.

\subsubsection{The Strong Law of Large Numbers for the skeleton}\label{sec:SLLN_as}

Throughout this section, we suppose that Assumptions~\ref{as:skeleton},~\ref{as:criticality} and~\ref{as:moment} hold.
Assumption~\ref{as:SLLN} may look like a strong assumption on first glance. However, we argue that this is not so. 
The skeleton decomposition shows that the large-time behaviour of the superprocess is guided by the skeleton. 
This suggests that the total mass the superprocess assigns to a compact ball, will be asymptotically well-behaved if and only if the skeleton carrying the superprocess has asymptotically a well-behaved number of particles in that ball. 
We write $\B_0(D):=\{B \in \B(D)\colon \ell(\partial B)=0\}$. 
To show that Assumption~\ref{as:SLLN}  holds, it suffices to prove that for all $\mu \in \M_c(D)$, $B \in \B_0(D)$,
\[
\liminf_{n \to \infty} e^{-\lambda_c n\delta} \Bigpint{\frac{\phi}{w}\mathbbm{1}_{B},Z_{n\delta}} \ge \pint{\phi \mathbbm{1}_B,\tphi} \mglZ \qquad \bbP_{\mu}\text{-almost surely}
\]
as we will see in Lemma~\ref{lem:reduc}(ii) below. Often it is a much easier task to prove the convergence along lattice times than along continuous times.
\smallskip

There are good results in the literature proving SLLNs for branching diffusions. Some of the relevant literature was reviewed in Section~\ref{sec:literature}. 
A nice argument to obtain almost sure asymptotics for spatial branching particle processes from related asymptotic behaviour of the spine was found recently by Harris and Roberts \cite{HarRob13pre}. 
However, they assume a convergence for the spine which usually does not hold in our setup. 
The theorem we use to verify several examples in Section~\ref{sec:examples} is based on a result from Engl\"ander et al.\ \cite{EngHarKyp10}. 
The authors prove the convergence for strictly dyadic branching diffusions along continuous times. We require only convergence along lattice times but a more general branching generator. 
The following theorem is a version of their result as our proof reveals. 

\begin{theorem}\label{thm:SLLN_skeleton}
Let $\mu \in \M_c(D)$ and assume that for every $x$ in the support of $\mu$ the following conditions hold:
\begin{itemize}
\item[{\rm (i)}] There is a family of sets $D_t \in \B(D)$, $t \ge 0$, such that for all $\delta >0$,
\[
\bbP_{\bullet, \delta_x}(\exists n_0 \in \N\colon{\rm{supp}}(Z_{n \delta}) \subseteq D_{n \delta} \text{ for all }n\ge n_0 )=1.
\]
\item[{\rm (ii)}] For every $B \subset \subset D$, there exists a constant $K>0$ such that
\begin{equation}\label{eq:ErgodicSpeed}
\sup_{y \in D_t}\big|\P_y^{\phi}[\mathbbm{1}_B(\xi_{Kt})]-\pint{\phi\mathbbm{1}_B,\tphi}\big| \to 0 \qquad \text{as }t \to \infty.
\end{equation}
\end{itemize}
Then, for all $\delta>0$, $f \in p(D)$ with $fw/\phi$ bounded,
\[
\lim_{n\to \infty} e^{-\lambda_c n\delta} \pint{ f ,Z_{n \delta}} = \pint{f,w\tphi} \mglZ \qquad \bbP_{\mu}\text{-almost surely.}
\]
\end{theorem}

\begin{proof}
Using Notation~\ref{not:Z}, we have $Z=\sum_{i=1}^{N_0} Z^{i,0}$, where given $\bbF_0$, $(Z^{i,0}\colon i =1,\ldots,N_0)$ are independent and $(Z^{i,0};\bbP_{\mu}(\cdot|\bbF_0))$ is equal in distribution to $(Z; \bbP_{\bullet, \delta_{\xi_i(0)}})$. 
In particular, $\mglZ=\sum_{i=1}^{N_0} W_{\infty}^{\phi/w}(Z^{i,0})$ and
\begin{align*}
\bbP_{\mu}\Big(\lim_{n  \to \infty} e^{-\lambda_c n \delta} \la f,Z_{n \delta} \ra =\pint{f,w\tphi} &\mglZ\Big)\ge \bbP_{\mu}\Big(\bigcap_{i=1}^{N_0} \Big\{ \lim_{n \to \infty} e^{-\lambda_c n\delta } \la f,Z_{n \delta}^{i,0} \ra =\pint{f,w\tphi} W_{\infty}^{\phi/w}(Z^{i,0}) \Big\}\Big)\\
&\phantom{space}= \bbP_{\mu}\Big[\prod_{i=1}^{N_0} \bbP_{\bullet, \delta_{\xi_i(0)}}\Big(\lim_{n \to \infty } e^{-\lambda_c n\delta} \la f,Z_{n \delta} \ra = \pint{f,w\tphi}\mglZ \Big)\Big].
\end{align*}
It remains to argue that under the stated assumptions, $\bbP_{\bullet,\delta_x}(\lim_{n \to \infty} e^{-\lambda_c n\delta} \la f,Z_{n\delta} \ra = \pint{f,w\tphi}\mglZ )$ equals $1$ for every $x$ in the support of $\mu$. 
Engl\"ander et al.\ \cite{EngHarKyp10} give a proof of this result for strictly dyadic branching diffusions in two steps. 
The argument can be generalised as follows. 
The first step is to show that with $(s_n)_{n \ge 0}$ nonnegative and non-decreasing, and $U_n=e^{-\lambda_c (s_n +\delta n) } \pint{f,Z_{s_n +\delta n}}$, the sequence $\text{D}_f(s_n,\delta n)=|U_n-\bbP_{\bullet, \delta_x}[U_n|\sigma(Z_r\colon r\le n\delta)]|$ converges to zero. 
The key to this result is an upper bound on the $p$-th moment of $W_t^{\phi/w}(Z)$ and is obtained via a spine decomposition of the branching diffusion. 
This would be possible even in our more general setup but is not needed since the required bound follows easily from Theorem~\ref{thm:SLLN_gen}(i).
The second step is to show the convergence of $\text{CE}_f(s_n,\delta n)=\bbP_{\bullet ,\delta_x}[U_n|\sigma(Z_r\colon r \le n\delta)]$ for $s_n=K \delta n$ to $\pint{f,w\tphi} \mglZ$. 
Their assumptions for this convergence are Condition~(iii) and (iv) in their Definition~4. Condition~(iii) is our Condition~(i) in Theorem~\ref{thm:SLLN_skeleton}. 
From the proof in \cite{EngHarKyp10} it is easy to see that their Condition~(iv) in Definition~4 can be relaxed to our Condition~(ii), a fact that has also been used in the verification of some examples in \cite{EngHarKyp10}.
\end{proof}

The following lemma is useful in the verification of the conditions of Theorem~\ref{thm:SLLN_skeleton} and has been proved by Engl\"ander et al. \cite{EngHarKyp10}. We give the main argument for completeness. We denote by $\twonorm{\cdot}$ the $\ell^2$-norm on $\R^d$.

\begin{lemma}
Suppose for $x\in D$ there are a continuous function $a\colon [0,\infty) \to [0,\infty)$ and some $\eps>0$ such that 
\begin{equation}
\P_x^{\phi}\big[\mathbbm{1}_{\{\twonorm{\xi_t}\ge a(t)\}} w(\xi_t)/\phi(\xi_t)] \le e^{-(\lambda_c +\eps)t} \qquad \text{for all $t$ sufficiently large.}\label{eq:SupVeri}
\end{equation}
Then Condition~(i) in Theorem~\ref{thm:SLLN_skeleton} holds with $D_t=\{y\in D\colon \twonorm{y}< a(t)\}$. If, in addition, for every $B\subset \subset D$, there is $K>0$ such that
\begin{equation}
\sup_{\twonorm{y_1}< a(t), y_2 \in B} \Big|\frac{p^{\phi}(y_1,y_2,Kt)}{\phi(y_2)\tphi(y_2)}-1\Big| \to 0 \qquad \text{as }t \to \infty,\label{eq:ErgoVeri}
\end{equation}
where $p^{\phi}$ denotes the transition density of $(\xi;\P^{\phi})$, then Condition~(ii) in Theorem~\ref{thm:SLLN_skeleton} is satisfied.
\end{lemma}

\begin{proof}
Markov's inequality and \eqref{eq:oppZ} yield for all $t \ge 0$,
\begin{align*}
\bbP_{\bullet, \delta_x}(\mbox{\rm supp}\big(Z_t) \not\subseteq D_t\big) \le \bbP_{\bullet,\delta_x}\big[\la \mathbbm{1}_{D_t^c},Z_t\ra\big] = e^{\lambda_c t} \frac{\phi(x)}{w(x)}\P_x^{\phi}\big[\mathbbm{1}_{\{\twonorm{\xi_t}\ge a(t)\}} w(\xi_t)/\phi(\xi_t)].
\end{align*}
The Borel-Cantelli Lemma yields the first part of the lemma. The second part follows immediately from $\pint{\phi,\tphi}<\infty$.
\end{proof}

We will see in Section~\ref{sec:examples} that for many of the main examples of superdiffusions the SLLN for the skeleton already follows from Theorem~\ref{thm:SLLN_skeleton}. 
For those processes where Assumption~\ref{as:SLLN} has not been proved yet, we believe that the particle nature of the skeleton will make it easier to obtain the SLLN for the skeleton than to derive further convergence statements in the superprocess setup. 
This article will then allow to carry results for the branching diffusion over to the superdiffusion. 
We emphasize that the SLLN for the skeleton is only needed along lattice times and for compactly supported starting measures. 

\subsection{Spine decomposition}\label{sec:spine}

In this section, we use a spine decomposition of $X$ to identify $(W_t^{\phi}(X))_{t\ge 0}$ as an $L^p$-bounded martingale,
where $p \in (1,2]$ is determined by Assumption~\ref{as:moment}. 
A similar decomposition has been used for other purposes by
Engl\"ander and Kyprianou \cite{EngKyp04} on bounded subdomains for quadratic branching mechanisms and by Liu et al.\ \cite{LiuRenSon09} in the case $\alpha=\mathbf{0}$. 
For the one-dimensional super-Brownian motion the spine decomposition was used by Kyprianou et al.\ \cite{KypLiuMurRen12,KypMur13} to establish $L^p$-boundedness of martingales closely related to $(W_t^{\phi}(X))_{t\ge 0}$. 
Similar arguments have been used in the setup of branching diffusions in \cite{HarHar09,EngHarKyp10}. 
See \cite{EngKyp04} for an overview of the history of spine decompositions for branching processes.
Throughout this section, 
we suppose that Assumption~\ref{as:criticality} holds. Further conditions used are stated explicitly. 
Recall that $\M_f^{\phi}(D)=\{\mu \in \M_f(D)\colon\la \phi,\mu\ra<\infty\}$.

\begin{theorem}\label{thm:Lp}
Suppose Assumptions \eqref{alpha_cond}--\eqref{as:PiUpperTail} hold. For all $\mu \in \M_f^{\phi}(D)$, $((W_t^{\phi}(X))_{t \ge 0};P_{\mu})$ is an $L^p$-bounded martingale. In particular, $((W_t^{\phi}(X))_{t \ge 0};P_{\mu})$ converges in $L^p(P_{\mu})$.
\end{theorem}

Let $\mu \in \M_f^{\phi}(D)$, $\mu \not\equiv 0$. We already showed in Corollary~\ref{cor:mg} that $(W_t^{\phi}(X))_{t\ge 0}$ is a martingale. 
Hence, it suffices to prove $L^p$-boundedness and we can define a new probability measure $Q_{\mu}$ by
\begin{equation*}
\frac{dQ_{\mu}}{dP_{\mu}}\bigg|_{\sigma(X_s\colon s \in [0, t])}= \frac{W_t^{\phi}(X)}{\la \phi, \mu\ra}\quad \text{for all }t\ge 0.
\end{equation*}
Recall from \eqref{eq:Pphimu} that $(\xi=(\xi_t)_{t \ge 0}; \P_{\phi \mu}^{\phi})$ is the ergodic motion with randomised start point and use \eqref{eq:Sphi} to obtain
\begin{equation}\label{eq:mto_Pphi}
\P_{\phi \mu}^{\phi}(A)=\frac{e^{-\lambda_c t}}{\la \phi,\mu\ra} \Big\la \P_{\cdot}\big[e^{\int_0^t \beta(\xi_s) \, ds } \phi(\xi_t) \mathbbm{1}_A\big],\mu\Big\ra \qquad \text{for all } A \in \sigma(\xi_s \colon 0 \le s \le t).
\end{equation}

\begin{lemma}\label{lem:LaplaceQ}
For all $\mu \in \M_f^{\phi}(D)$, $\mu \not\equiv 0$, $f,g \in bp(D)$, $t \ge 0$,
\begin{equation}\label{eq:LaplaceQ}
Q_{\mu}\Big[e^{-\la f,X_t\ra}\frac{\la \phi g ,X_t\ra}{\la \phi,X_t\ra}\Big]=P_{\mu}\big[e^{-\la f,X_t\ra}\big] \P_{\phi \mu}^{\phi}\Big[g(\xi_t)\exp\Big(-\int_0^t \partial_z\psi_0(\xi_s, u_f(\xi_s,t-s)) \, ds\Big)\Big].
\end{equation}
\end{lemma}

Notice that by definition, $\la \phi,X_t\ra>0$, $Q_{\mu}$-almost surely.

\begin{proof}[Proof of Lemma~\ref{lem:LaplaceQ}]
We prove \eqref{eq:LaplaceQ} only for $g$ compactly supported since the general case then follows from the monotone convergence theorem.
The continuity of $\phi$ implies that $f+\theta \phi g\in bp(D)$ for all $\theta\ge 0$. 
We use the definition of $Q_{\mu}$ and interchange differentiation and integration using the dominated convergence theorem to obtain
\begin{align*}
Q_{\mu}\Big[e^{-\la f,X_t\ra}\frac{\la \phi g,X_t\ra}{\la \phi,X_t\ra}\Big]=-\frac{e^{-\lambda_c t}}{\la \phi,\mu \ra}P_{\mu}\big[\partial_{\theta}\big|_{\theta=0} e^{-\la f+\theta \phi g,X_t\ra}\big]=\frac{e^{-\lambda_c t}}{\pint{ \phi,\mu }}e^{-\la u_{f}(\cdot, t),\mu\ra} \partial_{\theta}\big|_{\theta=0}\pint{ u_{f+\theta \phi g}(\cdot, t),\mu}.
\end{align*}
By \eqref{eq:logLaplace}, the definition of $\psi_{\beta}$, and \eqref{eq:mto_Pphi} the claim follows when we have shown that
\begin{equation}\label{ts:dutheta}
h_{f,g}(x,t):=\partial_{\theta}\big|_{\theta=0}u_{f+\theta \phi g}(x, t)=\P_x\Big[\phi(\xi_t)g(\xi_t)\exp\Big(-\int_0^t \partial_z\psi_{\beta}(\xi_s, u_f(\xi_s,t-s)) \, ds\Big)\Big]
\end{equation}
since integration with respect to $\mu$ and differentiation can be interchanged using the dominated convergence theorem. By \eqref{eq:mild}, for any $\theta>0$,
\[
\frac{u_{f+\theta \phi g}(x,t)-u_f(x,t)}{\theta}=S_t[\phi g](x)-\int_0^t S_s\Big[ \frac{\psi_0(\cdot, u_{f+\theta \phi g}(\cdot, t-s))-\psi_0(\cdot, u_f(\cdot, t-s))}{\theta}\Big](x) \, ds.
\]
The Laplace exponent $\theta \mapsto v(\theta):=u_{f+\theta \phi g}(x,t)=-\log P_{\delta_x}[e^{-\la f+\theta \phi g,X_t\ra}]$ is increasing, concave and nonnegative. In particular, $\frac{v(\theta)-v(0)}{\theta}$ is decreasing in $\theta$. 
Moreover, $z \mapsto\psi_0(x,z)$ is increasing, convex, and nonnegative. 
Hence, for all $(x,t) \in D\times [0,\infty)$,
\[
0 \le \frac{v(\theta)-v(0)}{\theta}=\frac{u_{f+\theta \phi g}(x,t)-u_f(x,t)}{\theta} \le S_t[\phi g](x) \le \supnorm{\phi g} e^{\bbeta t},
\]
where $\bbeta=\sup_{x \in D} \beta(x)$ and $\supnorm{\cdot}$ denotes the supremum norm.
A Taylor expansion of $\psi_0$ yields for every $(x,t) \in D \times [0,\infty)$ some $\tilde{\theta} \in (0,\theta)$ such that
\[
\frac{\psi_0(x, v(\theta))-\psi_0(x, v(0))}{\theta}=\partial_z\psi_0(x, v(0)) \frac{v(\theta)-v(0)}{\theta} +\Big(\partial_z\psi_0(x, v(\tilde{\theta}))-\partial_z\psi_0(x, v(0))\Big) \frac{v(\theta)-v(0)}{\theta}.
\]
The first summand on the right-hand side is nonnegative and increases as $\theta \downarrow 0$, the second term is dominated and tends to zero. Hence,
\begin{equation}\label{eq:uDeriv}
h_{f,g}(x,t)=S_t[\phi g](x)-\int_0^t S_s\big[\partial_z\psi_0(\cdot, u_f(\cdot,t-s))h_{f,g}(\cdot,t-s)\big](x)\, ds.
\end{equation}
Lemma~\ref{lem:semi}(ii) below applied to the functions $g_1(x,t)=-\partial_z\psi_{\beta}(x,u_f(x,t))$, $g_2(x,t)=\partial_z \psi_0(x,u_f(x,t))$, $f_1=\phi g$ and $f_2(x,t)=-\partial_z\psi_0(x,u_f(x,t))h_{f,g}(x,t)$ shows that the unique solution to \eqref{eq:uDeriv} is given by
the right-hand side of \eqref{ts:dutheta}. 
\end{proof}

Recall the definition of Dynkin and Kuznetsov's $\N_x$-measures from \eqref{eq:Nmeasure_def} and let $\mu \in \M_f^{\phi}(D)$, $\mu \not\equiv 0$. 
On a suitable probability space with measure $P_{\mu,\phi}$, we define the following processes:

\begin{itemize}
\item[(i)] $(\xi=(\xi_t)_{t \ge 0}; P_{\mu,\phi})$ is equal in distribution to $(\xi_t\colon t \ge 0; \P_{\phi \mu}^{\phi})$, that is an ergodic diffusion. 
We refer to this process as the \emph{spine}.

\item[(ii)] {\bf{Continuous immigration:}} $(\n;P_{\mu,\phi})$ a random measure such that, given $\xi$, $\n$ is a Poisson random measure which issues $\M_f(D)$-valued processes $X^{\n,t}=(X_s^{\n,t})_{s \ge 0}$ at space-time point $(\xi_t,t)$ with rate $2 \alpha(\xi_t)\, dt \times d\N_{\xi_t}$. 
The almost surely countable set of immigration times is denoted by $\D^{\n}$; $\D_t^{\n}:=\D^{\n}\cap(0,t]$. 
Given $\xi$, the processes $(X^{\n,t}\colon t \in \D^{\n})$ are independent.

\item[(iii)] {\bf{Discontinuous immigration:}} $(\m;P_{\mu,\phi})$ a random measure such that, given $\xi$, $\m$ is a Poisson random measure which issues $\M_f(D)$-valued processes $X^{\m,t}$ at space-time point $(\xi_t,t)$ with rate $dt \times \int_{(0,\infty)}\Pi(\xi_t, dy)\,y \times d P_{y \delta_{\xi_t}}$. 
The almost surely countable set of immigration times is denoted by $\D^{\m}$; $\D_t^{\m}=\D^{\m}\cap (0,t]$. 
Given $\xi$, the processes $(X^{\m,t}\colon t \in \D^{\m})$ are independent and independent of $\n$ and $(X^{\n,t}\colon t \in \D^{\n})$.

\item[(iv)] $(X=(X_t)_{t \ge 0};P_{\mu,\phi})$ is equal in distribution to $(X=(X_t)_{t \ge 0};P_{\mu})$, a copy of the original process. 
Moreover, $X$ is independent of $\xi, \n,\m$ and all immigration processes.
\end{itemize}
We denote by 
\[
X_t^{\n}=\sum_{s \in \D_t^{\n}}X_{t-s}^{\n,s}, \qquad \text{and} \qquad X_t^{\m}=\sum_{s \in \D_t^{\m}}X_{t-s}^{\m,s}, 
\]
the continuous and discontinuous immigration processes, respectively. 
We write $\Gamma_t:=X_t+X_t^{\n}+X_t^{\m}$ for all $t \ge 0$ and $\overset{d}{=}$ denotes distributional equality.

\begin{proposition}[{\bf{Spine decomposition}}]\label{pro:spine} For all $\mu \in \M_f^{\phi}(D)$, $\mu \not\equiv 0$,
\[
(X_t\colon t \ge 0; Q_{\mu})\overset{d}{=}(\Gamma_t=X_t+X_t^{\n}+X_t^{\m}\colon t \ge 0;P_{\mu,\phi}).
\]
\end{proposition}

The proof of Proposition~\ref{pro:spine} is very similar to the proof of Theorem 5.2 in \cite{KypLiuMurRen12} and we omit long computations.

\begin{proof}[Proof of Proposition~\ref{pro:spine}]
Using the definitions and Campbell's formula for Poisson random measures, one easily checks that the marginal distributions agree. 
By definition, $( (\Gamma_t,\xi_t)_{t \ge 0};P_{\mu,\phi})$ is a time-homogeneous Markov process and when we show that
\[
P_{\mu,\phi}(\xi_t \in \, dx\,|\,\Gamma_t)=\frac{1}{\la \phi,\Gamma_t\ra} \phi(x) \Gamma_t(dx) \quad \text{ for all } t \ge 0,
\]
then $((\Gamma_t)_{t \ge 0};P_{\mu,\phi})$ is a time-homogeneous Markov process (by the argument given on page 21 of \cite{KypLiuMurRen12}). 
Using the definition, $(\Gamma_t;P_{\mu,\phi})\overset{d}{=}(X_t;Q_{\mu})$ and Lemma~\ref{lem:LaplaceQ}, we find that for all $f,g \in bp(D)$,
\[
P_{\mu,\phi}\big[e^{-\la f,\Gamma_t\ra} P_{\mu,\phi}[g(\xi_t)\,|\,\Gamma_t]\big]=P_{\mu,\phi}\Big[e^{-\la f,\Gamma_t\ra} \frac{\la \phi g,\Gamma_t\ra}{\la \phi,\Gamma_t\ra}\Big],
\]
and the claim follows.
\end{proof}

For all $t \ge 0$, let $\mathcal{G}_t$ be the $\sigma$-algebra generated by $\xi$ up to time $t$ and by $\n$ and $\m$ restricted in the time component to $[0,t]$.

\begin{lemma}\label{lem:spine} For all $\mu \in \M_f^{\phi}(D)$, $\mu\not\equiv 0$, and $t \ge 0$,
\[
P_{\mu,\phi}\big[e^{-\lambda_c t} \la \phi,\Gamma_t\ra|\mathcal{G}_t\big]=\la \phi, \mu \ra +\sum_{s \in \D_t^{\n}} e^{-\lambda_c s} \phi(\xi_s)+\sum_{s \in \D_t^{\m}} e^{-\lambda_c s} \mathcal{I}_s^{\m} \phi(\xi_s) \qquad P_{\mu,\phi}\text{-almost surely},
\]
where $(\mathcal{I}_t^{\m}:=\pint{ \mathbf{1},X_0^{\m,t}}\colon t\ge 0;P_{\mu,\phi})$ is, given $\xi$, a Poisson point process with intensity measure $dt \times \int_{(0,\infty)} \Pi(\xi_t,dy)\, y$.
\end{lemma}

\begin{proof} 
Proposition 1.1 of \cite{DynKuz04} states that, for $f \in p(D)$ with $P_{\delta_x}[\la f,X_t\ra]<\infty$,
\begin{equation} \label{eq:Nmean}
\N_x[\la f,X_t\ra]=P_{\delta_x}[\la f,X_t\ra].
\end{equation}
Using first the definition of $\Gamma_t$, and then \eqref{eq:Nmean} and the branching property of $X$ we obtain,
\begin{align*}
P_{\mu,\phi}[e^{-\lambda_c t}\la \phi,\Gamma_t\ra|\mathcal{G}_t] &= P_{\mu}\big[W_t^{\phi}(X)\big]+\sum_{s \in \D_t^{\n}} e^{-\lambda_c t} \N_{\xi_s}[\la \phi, X_{t-s}\ra] +\sum_{s \in \D_t^{\m}} e^{-\lambda_c t} P_{\mathcal{I}_s^{\m} \delta_{\xi_s}}[\la \phi,X_{t-s}\ra]\\
&= P_{\mu}\big[W_t^{\phi}(X)\big]+\sum_{s \in \D_t^{\n}} e^{-\lambda_c t} P_{\delta_{\xi_s}}[\la \phi, X_{t-s}\ra] +\sum_{s \in \D_t^{\m}} e^{-\lambda_c t} \mathcal{I}_s^{\m} P_{\delta_{\xi_s}}[\la \phi,X_{t-s}\ra].
\end{align*}
Since $W_t^{\phi}(X)=e^{-\lambda_c t}\la \phi,X_t\ra$, $t\ge 0$, is a $P_{\mu}$- and $P_{\delta_x}$-martingale for all $x\in D$, the claim follows.
\end{proof}

Finally, everything is prepared for the proof of Theorem~\ref{thm:Lp}. 
Throughout the article, we use the letters $c$ and $C$ for generic constants in $(0,\infty)$ and their value can change from line to line. 
Important constants are marked by an index indicating the order in which they occur.

\begin{proof}[Proof of Theorem~\ref{thm:Lp}] The martingale property was proved in Corollary~\ref{cor:mg}. 
We have to show the $L^p$-boundedness. If $\mu \equiv 0$, then $X_t(D)=0$ for all $t \ge 0$ and the statement is trivially true. 
Let $\mu \in \M_f^{\phi}(D)$, $\mu \not \equiv 0$. 
We write $W_t^{\phi}(\Gamma)=e^{-\lambda_c t} \pint{\phi,\Gamma_t}$ and $\bar{p}=p-1 \in (0,1]$. 
Then, $x \mapsto x^{\bar{p}}$ is concave and $(x+y)^{\bar{p}} \le x^{\bar{p}}+y^{\bar{p}}$ for $x,y\ge 0$. 
Hence, the definition of $Q_{\mu}$, Proposition~\ref{pro:spine}, Jensen's inequality and Lemma~\ref{lem:spine}, yield
\begin{align*}
\frac{P_{\mu}[W_t^{\phi}(X)^p]}{\la \phi, \mu \ra}&=Q_{\mu}\big[W_t^{\phi}(X)^{\bar{p}}\big]=P_{\mu,\phi}\big[P_{\mu,\phi}[W_t^{\phi}(\Gamma)^{\bar{p}}|\mathcal{G}_t]\big]\le P_{\mu,\phi}[P_{\mu,\phi}[W_t^{\phi}(\Gamma)|\mathcal{G}_t]^{\bar{p}}]\\
&\le P_{\mu,\phi}\Big[\la \phi, \mu \ra^{\bar{p}} +\Big(\sum_{s \in \D_t^{\n}} e^{-\lambda_c s} \phi(\xi_s)\Big)^{\bar{p}}\Big]\\
&\phantom{space}+P_{\mu,\phi}\Big[\Big(\sum_{s \in \D_t^{\m}, \mathcal{I}_s^{\m} \le \varphi_1(\xi_s)} e^{-\lambda_c s} \mathcal{I}_s^{\m} \phi(\xi_s)\Big)^{\bar{p}}+\Big(\sum_{s \in \D_t^{\m}, \mathcal{I}_s^{\m} > \varphi_1(\xi_s)} e^{-\lambda_c s} \mathcal{I}_s^{\m} \phi(\xi_s)\Big)^{\bar{p}}\Big],
\end{align*}
where $\varphi_1$ is determined by Assumption~\ref{as:moment}.
The first summand is deterministic. For the remaining three summands we first use that $x^{\bar{p}} \le 1+x^{\bar{\sigma}}$ for all $\bar{\sigma} \ge \bar{p}$, then $(x+y)^{\bar{\sigma}} \le x^{\bar{\sigma}}+y^{\bar{\sigma}}$ for all $\bar{\sigma} \in [0,1]$, and finally apply Campbell's formula to obtain
\begin{align*}
&P_{\mu,\phi}\Big[\Big(\sum_{s \in \D_t^{\n}} e^{-\lambda_c s} \phi(\xi_s)\Big)^{\bar{p}}\Big]\le 1+ \int_0^t 2e^{-\lambda_c \bar{\sigma}_1 s} \P_{\phi \mu}^{\phi}[\phi(\xi_s)^{\bar{\sigma}_1} \alpha(\xi_s)]\, ds,\\
&P_{\mu,\phi}\Big[\Big(\sum_{s \in \D_t^{\m}, \mathcal{I}_s^{\m} \le \varphi_1(\xi_s)} e^{-\lambda_c s} \mathcal{I}_s^{\m} \phi(\xi_s)\Big)^{\bar{p}}\Big]
\le 1+ \int_0^t e^{-\lambda_c \bar{\sigma}_2 s} \P_{\phi \mu}^{\phi}\Big[\int_{(0,\varphi_1(\xi_s)]} \phi(\xi_s)^{\bar{\sigma}_2} y^{\bar{\sigma}_2+1}\, \Pi(\xi_s, dy)\Big]\, ds,\\
&P_{\mu,\phi}\Big[\Big(\sum_{s \in \D_t^{\m}, \mathcal{I}_s^{\m} >\varphi_1(\xi_s)} e^{-\lambda_c s} \mathcal{I}_s^{\m} \phi(\xi_s)\Big)^{\bar{p}}\Big]
\le 1+ \int_0^t e^{-\lambda_c \bar{\sigma}_3 s} \P_{\phi \mu}^{\phi}\Big[\int_{(\varphi_1(\xi_s),\infty)} \phi(\xi_s)^{\bar{\sigma}_3} y^{\bar{\sigma}_3+1}\, \Pi(\xi_s, dy)\Big]\, ds,
\end{align*}
where $\bar{\sigma}_i=\sigma_i-1 \in [\bar{p},1]$ with $\sigma_i$ defined in Assumption~\ref{as:moment}, $i \in \{1,2,3\}$.
We conclude that if Assumptions~\eqref{alpha_cond}--\eqref{as:PiUpperTail} hold, then there exists a constant $C_1 \in (0,\infty)$ independent of $\mu$ and $t$ such that
\begin{equation}\label{eq:EWtoq}
\frac{P_{\mu}[W_t^{\phi}(X)^p]}{\la \phi,\mu\ra} \le \la \phi,\mu\ra^{\bar{p}} + C_1 \qquad \text{for all }t \ge 0,
\end{equation}
and $((W_t^{\phi}(X))_{t \ge 0};P_{\mu})$ is an $L^p$-bounded martingale. Doob's inequality yields the stated $L^p$-convergence.
\end{proof}

In Section~\ref{sec:L1Conv} the following lemma will be used in the comparison between immigration process and its conditional expectation.

\begin{lemma}\label{lem:mg_moments}
Suppose Assumptions \eqref{alpha_cond}--\eqref{phi_p_cond} hold. For every $\mu \in \M_c(D)$, $\mu \not \equiv 0$, there exists a time $T>0$ and a constant $C_2\in (0,\infty)$ such that
\[
\P_{\phi\mu}^{\phi}\big[\phi(\xi_t)^{-1} P_{\delta_{\xi_t}}[W_s^{\phi}(X)^p]\big] \le C_2 \qquad \text{for all }s\ge 0, t\ge T.
\]
\end{lemma}

\begin{proof}
According to \eqref{eq:EWtoq}, for all $s,t\ge 0$,
\[
\P_{\phi\mu}^{\phi}\big[\phi(\xi_t)^{-1} P_{\delta_{\xi_t}}[W_s^{\phi}(X)^p]\big]  \le \P_{\phi\mu}^{\phi}[\phi(\xi_t)^{p-1}] + C_1.
\]
Since $\mu \in \M_c(D)$ and $\pint{\phi^{p-1},\phi \tphi}<\infty$ by Assumption~\eqref{phi_p_cond}, $\P_{\phi \mu}^{\phi}[\phi(\xi_t)^{p-1}]$ converges to $\pint{\phi^{p-1},\phi \tphi}$ by \eqref{eq:ergodic} and we obtain the required time $T$ and constant $C_2$.\end{proof}

\section{Proofs of the main results}\label{sec:main}

\subsection{Reduction to a core statement}\label{sec:reduction}

In this section, we work under Assumption~\ref{as:criticality}.
We first show that it suffices to consider test functions $f=\phi\mathbbm{1}_B=:\phi|_B$ for Borel sets $B \in \B_0(D)=\{B \in \B(D)\colon \ell(\partial B)=0\}$ and 
that we only have to prove that $\liminf_{t \to \infty} e^{-\lambda_c t} \la f, X_t\ra \ge \la f,\tphi \ra \mglX$ instead of the full convergence. 
The proof is based on standard approximation theory combined with an idea that appeared in Lemma 9 of \cite{AsmHer76}. 
We denote by $C_{\ell}^+(D)$ the space of nonnegative, measurable, $\ell$-almost everywhere continuous functions on $D$.

\begin{lemma}\label{lem:reduc}
Let $\mu \in \M_f^{\phi}(D)$ and either $\T=[0,\infty)$ or $\T=\delta \N$ for some $\delta>0$. 
In addition, let either $\mathcal{A}=\B_0(D)$ and $\mathcal{A}^{\phi}=\{f \in C_{\ell}^+(D) \colon f/\phi \in b(D)\}$, or $\mathcal{A}=\B(D)$ and $\mathcal{A}^{\phi}=\{f \in p(D) \colon f/\phi \in b(D)\}$. 
We define $\mathcal{A}^{\phi/w}$ like $\mathcal{A}^{\phi}$ where $\phi$ is replaced by $\phi/w$.
\begin{itemize}
\item[{\rm (i)}] 
If for all $B \in \mathcal{A}$,
\begin{equation}\label{liminf_general}
\liminf_{\T \ni t \to \infty} e^{-\lambda_ct}\la \phi|_B ,X_t\ra\ge \la \phi|_B,\tphi\ra \mglX \qquad P_{\mu}\text{-almost surely},
\end{equation}
then $\lim_{\T \ni t \to \infty} e^{-\lambda_ct}\la f ,X_t\ra=\la f,\tphi\ra \mglX$  $P_{\mu}$-almost surely for all $f \in \mathcal{A}^{\phi}$.
\item[{\rm (ii)}] If for all $B \in \mathcal{A}$, $\liminf_{\T\ni t \to \infty} e^{-\lambda_ct}\la \frac{\phi}{w} \mathbbm{1}_B ,Z_t\ra\ge \la \phi|_B,\tphi\ra \mglZ$ $\bbP_{\mu}$-almost surely,
then, for all $f \in \mathcal{A}^{\phi/w}$, $\lim_{\T \ni t \to \infty} e^{-\lambda_ct}\la f ,Z_t\ra=\la f,w\tphi\ra \mglZ$ $\bbP_{\mu}$-almost surely.
\end{itemize}
\end{lemma}

\begin{proof}
We show only Part~(i); the proof of Part~(ii) is similar. Let $\simplefct=\{\sum_{i=1}^k c_i \phi|_{B_i}\colon k \in \N, c_i \in [0,\infty), B_i \in \mathcal{A}\}$ and $f \in \mathcal{A}^{\phi}$. 
There exists a sequence of functions $f_k\in \simplefct$ such that $0\le f_k \le f$ and $f_k\uparrow f$ pointwise. 
Using \eqref{liminf_general} and the monotone convergence theorem, we deduce that $P_{\mu}$-almost surely,
\[
\liminf_{\T \ni t \to \infty} e^{-\lambda_c t} \la f,X_t\ra \ge \sup_{k \in \N}\liminf_{\T \ni t \to \infty} e^{-\lambda_c t} \la f_k,X_t\ra\ge \sup_{k \in \N} \la f_k, \tphi \ra \mglX=\la f,\tphi\ra \mglX.
\]
Let $c=\sup_{x \in D} f(x)/\phi(x)$. Since $0 \le c \phi-f \le c\phi$, the same argument can be applied to $c\phi-f$ and we conclude that $P_{\mu}$-almost surely
\begin{align*}
\limsup_{\T \ni t \to \infty} e^{-\lambda_c t} \la f,X_t\ra&=\limsup_{\T \ni t \to \infty}\Big(cW_t^{\phi}(X)- e^{-\lambda_c t} \la c\phi-f,X_t\ra\Big)\le c\mglX- \liminf_{\T \ni t \to \infty}e^{-\lambda_c t} \la c\phi-f,X_t\ra\\
&\le c \la \phi,\tphi\ra \mglX- \la c\phi-f, \tphi \ra \mglX=\la f,\tphi\ra \mglX.\qedhere
\end{align*}
\end{proof}

In the next step, we use the branching property of the superprocess to restrict ourselves to compactly supported starting measures.

\begin{lemma}\label{lem:muComp}
Let $\T=[0,\infty)$ or $\T=\delta \N_0$ for some $\delta>0$, and let $\mathcal{A}^{\phi}=\{f \in C_{\ell}^+(D)\colon f/\phi \in b(D)\}$ or $\mathcal{A}^{\phi}=\{f \in p(D)\colon f/\phi \in b(D)\}$.
\begin{itemize}
\item[\rm{(i)}]
If for all $\mu \in \M_c(D)$ and $f \in \mathcal{A}^{\phi}$,
\begin{equation}\label{eq:muComp}
\lim_{\T \ni t \to \infty} e^{-\lambda_c t} \la f,X_t\ra = \la f,\tphi\ra \mglX \qquad P_{\mu}\text{-almost surely,}
\end{equation}
then \eqref{eq:muComp} holds for all $\mu \in \M_f^{\phi}(D)$. 
\item[\rm{(ii)}] If convergence \eqref{eq:muComp} holds in $L^1(P_{\mu})$ for all $\mu \in \M_c(D)$, then it holds for all $\mu \in \M_f^{\phi}(D)$.
\end{itemize}
\end{lemma}

\begin{proof}
Let $\mu \in \M_f^{\phi}(D)$ and take a sequence of domains $B_k \subset\subset D$, $B_k \subseteq B_{k+1}$, with $D=\bigcup_{k=1}^{\infty} B_k$; $\hat{B}_k:=B_k\setminus B_{k-1}$, where $B_0:=\emptyset$. 
On a suitable probability space, let $\smash{X^{\hat{B}_k}}$, $k \in \N$, be independent $(L,\psi_{\beta};D)$-superprocesses, where $X^{\hat{B}_k}$ is started in $\mathbbm{1}_{\hat{B}_k} \mu$. 
By the branching property, $\smash{X^{B_k}:=\sum_{l=1}^k X^{\hat{B}_l}}$, $\smash{X^{D\setminus B_k}:=\sum_{l=k+1}^{\infty} X^{\hat{B}_l}}$ and 
$X:=X^{B_k}+X^{D \setminus B_k}$ are $(L,\psi_{\beta};D)$-superprocesses with starting measures $\mathbbm{1}_{B_k}\mu$, $\mathbbm{1}_{D \setminus B_k}\mu$ and $\mu$, respectively. In particular,
\[
W_t^{\phi}(X)=e^{-\lambda_c t} \la \phi,X_t^{B_k}+X_t^{D\setminus B_k}\ra = W_t^{\phi}(X^{B_k})+W_t^{\phi}(X^{D\setminus B_k}),
\]
and the martingale limits $W_{\infty}^{\phi}(X^{D \setminus B_k}):=\liminf_{t \to \infty}W_t^{\phi}(X^{D\setminus B_k})$, $k \in \N$, are decreasing in $k$. Fatou's Lemma yields
\[
P_{\mu}\big[W_{\infty}^{\phi}(X^{D \setminus B_k})\big]=P_{\mu}\big[\lim_{t \to \infty}W_{t}^{\phi}(X^{D \setminus B_k})\big]\le \liminf_{t \to \infty}P_{\mu}\big[W_{t}^{\phi}(X^{D \setminus B_k})\big]=\la \phi, \mathbbm{1}_{D \setminus B_k} \mu\ra.
\]
In particular, $\la \phi,\mu\ra <\infty$ implies that $(W_{\infty}^{\phi}(X^{D \setminus B_k})\colon k \in \N)$ converges to zero in $L^1(P_{\mu})$ as $k \to \infty$ and since the sequence is monotonically decreasing, almost sure convergence follows.
We conclude that $\lim_{k \to \infty} W_{\infty}^{\phi}(X^{B_k}) =\mglX$ almost surely and in $L^1(P_{\mu})$. 

By Lemma~\ref{lem:reduc}(i), for Part~(i) it suffices to show that for all $f \in \mathcal{A}^{\phi}$,
\[
\liminf_{\T \ni t \to \infty} e^{-\lambda_c t} \la f,X_t \ra \ge \la f,\tphi \ra \mglX \qquad P_{\mu}\text{-almost surely.}
\]
Since $\mathbbm{1}_{B_k} \mu \in \M_c(D)$, the assumption implies that for all $k \in \N$,
\begin{equation*}
\liminf_{t \to \infty} e^{-\lambda_c t} \la f,X_t \ra \ge \liminf_{t \to \infty} e^{-\lambda_c t} \la f,X_t^{B_k} \ra\ge \la f,\tphi\ra W_{\infty}^{\phi}(X^{B_k}) \qquad P_{\mu}\text{-almost surely,}
\end{equation*}
and taking $k\to \infty$ yields the claim.

To show Part~(ii), let $c=\sup_{x \in D}f(x)/\phi(x)$ and estimate for fixed $k \in \N$,
\begin{align*}
P_{\mu}\big[\big|e^{-\lambda_c t}& \pint{f,X_t}- \pint{f,\tphi} \mglX\big|\big] \\
&\le c P_{\mu}\big[e^{-\lambda_c t} \pint{\phi,X_t^{D\setminus B_k}}\big] + P_{\mu}\big[\big|e^{-\lambda_c t} \pint{f,X_t^{B_k}}- \pint{f,\tphi} W_{\infty}^{\phi}(X^{B_k})\big|\big]+ \pint{f,\tphi} P_{\mu}\big[W_{\infty}^{\phi}(X^{D \setminus B_k})\big].
\end{align*}
The second term on the right-hand tends to zero as $t \to \infty$  by assumption. The first term is equal to $c \pint{\phi,\mathbbm{1}_{D\setminus D_k} \mu}$ and, therefore, tends to zero as $k \to \infty$, and so does the third term.
\end{proof}
Let $\M(D)$ be the set of all $\sigma$-finite measures on $D$.

\begin{remark}\label{rem:sigmaFinMu}
The superprocess $X$ can be defined for starting measures $\mu \in \M(D)$ via the branching property (see also Section~1.4.4.1 in \cite{Dyn02}).
 The proof of Lemma~\ref{lem:muComp} then shows that \eqref{eq:muComp} for all $\mu \in \M_c(D)$ implies \eqref{eq:muComp} for all $\mu \in \M(D)$ with $\pint{\phi,\mu}<\infty$.
\end{remark}

Finally, we show that it suffices to consider fixed test functions. The argument
is borrowed from Chen and Shiozawa \cite[Theorem 3.7]{CheShi07}.

\begin{lemma}[{\bf Chen and Shiozawa \cite{CheShi07}}]\label{lem:omega0}
Let $\mu \in \M_f^{\phi}(D)$. If for every $B\in \B_0(D)$, $P_{\mu}$-almost surely, $\lim_{t\to \infty}e^{-\lambda_c t}\pint{\phi|_B,X_t}=\pint{\phi|_B,\tphi}\mglX$, then there exists a  measurable set $\Omega_0$ such that $P_{\mu}(\Omega_0)=1$ and, on $\Omega_0$, the convergence in \eqref{eq:SLLN2} holds for all $f \in C_{\ell}^+(D)$ with $f/\phi$ bounded.
\end{lemma}

\begin{proof}
Take a countable base $(B_k)_{k \in \N}$ of $\B_0(D)$ which is closed under finite unions and let
\[
\Omega_0=\Big\{\lim_{t \to \infty} e^{-\lambda_c t} \la \phi|_{B_k} ,X_t \ra= \la \phi|_{B_k},\tphi\ra \mglX \; \text{for all } k \in \N\Big\}.
\] 
Then $P_{\mu}(\Omega_0)=1$ by assumption. On $\{\mglX=0\}$, convergence \eqref{eq:SLLN} trivially holds for all $f \in p(D)$ with $f/\phi$ bounded. On $\{\mglX>0\}\cap\Omega_0$, we define
\[
\chi_t(A):= e^{-\lambda_c t} \frac{\la \phi|_A,X_t\ra}{\mglX} \qquad\text{and} \qquad \chi(A)=\la\phi|_A,\tphi\ra, \qquad \text{for all } A \in \B(D). 
\]
Then $\chi_t$ converges vaguely to $\chi$ as $t \to \infty$ and, since $\lim_{t\to \infty}\chi_t(D)=\chi(D)=1$, the convergence holds also in the weak sense (cf.\ Theorem 13.35 in \cite{Kl08}). 
For every $f\in C_{\ell}^+(D)$ with $f/\phi$ bounded, $g:=f/\phi \in bp(D)$ is $\ell$-almost everywhere continuous and, since $\chi $ is absolutely continuous with respect to $\ell$, $\lim_{t \to \infty} \pint{g,\chi_t}=\pint{g,\chi}$, which is equivalent to
\[
\lim_{t \to \infty} e^{-\lambda_c t} \la f,X_t\ra=\mglX \la f,\tphi\ra \qquad \text{on } \Omega_0 \cap \{\mglX>0\}.\qedhere
\]
\end{proof}

\subsection{Martingale limits}\label{mg_limits}

In this section, we prove Proposition~\ref{pro:mglim}, that is, we show that the martingale limits for the superprocess and its skeleton agree almost surely.
We assume only Assumptions~\ref{as:skeleton} and~\ref{as:criticality} throughout this section.
Recall from \eqref{eq:opp_star} that $(S_t^*)_{t\ge 0}$ denotes the expectation semigroup of $X^*$

\begin{lemma}\label{lem:ext_lim}
Let $f \in p(D)$ with $f/\phi$ bounded. For all $x \in D$,
\[
\theta_t^*(x):=e^{-\lambda_c t} S_t^*f(x)/\phi(x) \to 0 \qquad \text{as }t \to \infty,
\]
and for $t>0$, the function $x\mapsto \theta_t^*(x)$ is continuous. Moreover, $\theta_t^*(x)$ is uniformly bounded in $t$ and $x$, and, if $f=\phi$, $\theta_t^*(x)$ is non-increasing in $t$.
\end{lemma}

\begin{proof}
Let $c=\sup_{x\in D}f(x)/\phi(x)$. By \eqref{eq:defSstar} and \eqref{eq:Sphi}, for all $(x,t) \in D\times [0, \infty)$,
\[
0 \le \theta_t^*(x) = \P_x^{\phi}\Big[ e^{\int_0^t [\beta^*(\xi_s)-\beta(\xi_s)]\, ds} f(\xi_t)/\phi(\xi_t)\Big]\le c e^{-\lambda_c t}S_t^*\phi(x)/\phi(x).
\]
Since $\beta^*-\beta \le \mathbf{0}$ pointwise, Theorem~7.2.4 in \cite{StrVar79} (see also Theorem~4.9.7 in \cite{Pin95}) implies that $\theta_t^*$ is continuous for $t>0$, and $e^{-\lambda_c t}S_t^*\phi(x)/\phi(x)$ is non-increasing in $t$. 
The dominated convergence theorem yields
\[
\lim_{t \to \infty}e^{-\lambda_c t}S_t^*\phi(x)/\phi(x)= \P_x^{\phi}\Big[\exp\Big(\int_0^{\infty} [\beta^*(\xi_s)-\beta(\xi_s)] \, ds\Big)\Big].
\]
By Assumption~\ref{as:criticality}, the diffusion $(\xi=(\xi_t)_{t \ge 0};\P_x^{\phi})$ is positive recurrent and Theorem 4.9.5(ii) in \cite{Pin95} yields
\[
\lim_{t \to \infty} \frac{1}{t} \int_0^t \min\{\beta(\xi_s)-\beta^*(\xi_s),1\} \, ds= \bigpint{\hspace{-0.08cm}\min\{\beta-\beta^*,1\}, \phi\tphi} >0\qquad \P_x^{\phi}\text{-almost surely,}
\]
where the limit is positive since $\ell(\{x \in D\colon \alpha(x)+\Pi(x,(0,\infty))>0\})>0$ by Remark~\ref{rem:detCase}. 
Hence, $\int_0^{\infty} [\beta^*(\xi_s)-\beta(\xi_s)] \, ds =-\infty$ $\P_x^{\phi}$-almost surely  and the claim is established.
\end{proof}

\begin{lemma}\label{lem:bb_meas}
For ${\rm p}\ge 1$, $f \in p(D)$, $x \in D$ and $t \ge 0$,
\[
\bbP_{\bullet,\delta_x}\big[\la f,I_t\ra^{\rm p}\big]\le w(x)^{-1}\bbP_{\delta_x}\big[\la f,I_t\ra^{\rm p}\big],
\]
where the inequality is an equality in the case ${\rm p}=1$.
\end{lemma}

\begin{proof}
Using Notation~\ref{not:I}, $I_t=\sum_{i=1}^{N_0} I_t^{i,0}$, where under $\bbP_{\delta_x}$, $N_0$ is Poisson distributed with mean $w(x)$, $\bbF_0$-measurable, and $(I_t^{i,0};\bbP_{\delta_x}(\cdot|\bbF_0))$ is equal in distribution to $(I_t;\bbP_{\bullet, \delta_x})$. Using the monotonicity of the $\ell^{\rm p}$-norm, we derive
\begin{align*}
\bbP_{\delta_x}[\pint{f,I_t}^{\rm p}] = \bbP_{\delta_x}\Big[\Big(\sum_{i=1}^{N_0} \pint{f,I_t^{i,0}}\Big)^{\rm p}\Big] \ge \bbP_{\delta_x}\Big[\sum_{i=1}^{N_0} \pint{f,I_t^{i,0}}^{\rm p}\Big] =\bbP_{\delta_x}\big[N_0 \bbP_{\bullet, \delta_x}[\pint{f,I_t}^{\rm p}]\big]= w(x) \bbP_{\bullet,\delta_x}[\pint{f,I_t}^{\rm p}].
\end{align*}
For ${\rm p}=1$ the inequality is an equality. Rearranging terms concludes the proof.
\end{proof}

We now come to the main part of this section. First, we employ the skeleton decomposition to compute the conditional expectation of $\pint{ f,X_{s+t}}$. 

\begin{proposition}\label{pro:bb_mean}
For all $\mu \in \M_f^{\phi}(D)$, $f \in p(D)$ with $f/\phi$ bounded, and $s,t\ge 0$,
\[
\bbP_{\mu}\big[\la f,X_{s+t}\ra|\bbF_t\big]=\la S_s^* f,X_t\ra+\Big\la \frac{S_sf}{w},Z_t\Big\ra -\Big\la \frac{S_s^*f}{w},Z_t \Big\ra \qquad \bbP_{\mu}\text{-almost surely.}
\]
\end{proposition}

\begin{proof}
Using Notation~\ref{not:I} and \eqref{eq:opp_star}, we have $\pint{f,X_{s+t}}=\pint{f,X_{s+t}^*+I_s^{*,t}}+\sum_{i=1}^{N_t} \pint{f,I_s^{i,t}}$, where $(X_{s+t}^*+I_s^{*,t};\bbP_{\mu}(\cdot|\bbF_t))$ is equal in distribution to $(X_s^*;\bbP_{X_t})$ and $(I_s^{i,t};\bbP_{\mu}(\cdot|\bbF_t))$ to $(I_s;\bbP_{\bullet,\delta_{\xi_i(t)}})$, $i=1,\ldots,N_t$. Hence, $\bbP_{\mu}$-almost surely,
\begin{equation}\label{eq:bb_cond}
\bbP_{\mu}[\la f,X_{s+t}\ra|\bbF_t]=\bbP_{\mu}\Big[\pint{f,X_{s+t}^*+I_s^{*,t}}+\sum_{i=1}^{N_t} \pint{f,I_s^{i,t}} \Big|\bbF_t\Big]=\bbP_{X_t}[\la f,X_s^*\ra]+\sum_{i=1}^{N_t}\bbP_{\bullet,\delta_{\xi_i(t)}}[\la f,I_{s}\ra].
\end{equation}
The first summand on the right can be rewritten using \eqref{eq:opp_star} to obtain $\bbP_{X_t}[\la f,X_s^*\ra]=\la S_s^*f, X_t\ra$. 
For the second summand, we use Lemma~\ref{lem:bb_meas} and Theorem~\ref{thm:KPR}~(d) to derive
\[
\sum_{i=1}^{N_t} \bbP_{\bullet,\delta_{\xi_i(t)}}[\la f,I_s\ra]=\sum_{i=1}^{N_t} w(\xi_i(t))^{-1} \bbP_{\delta_{\xi_i(t)}}[\la f,X_s-X_s^*\ra].
\]
Since $f/\phi$ is bounded and $\mu \in \M_f^{\phi}(D)$, $\bbP_{\delta_{\xi_i(t)}}[\la f,X_s\ra]$ is finite, and \eqref{eq:opp} and \eqref{eq:opp_star} yield
\begin{equation}\label{eq:Eimmi}
\bbP_{\mu}\Big[\sum_{i=1}^{N_t} \pint{f,I_s^{i,t}} \Big|\bbF_t\Big]=\sum_{i=1}^{N_t} \bbP_{\bullet,\delta_{\xi_i(t)}}[\la f,I_s\ra] = \Big\la \frac{S_sf}{w},Z_t\Big\ra- \Big\la \frac{S_s^*f}{w},Z_t\Big\ra \quad \bbP_{\mu}\text{-almost surely}
\end{equation}
as required.\end{proof}


\begin{proof}[Proof of Proposition~\ref{pro:mglim} and Theorem~\ref{thm:SLLN_gen}(i)] 
Proposition~\ref{pro:bb_mean} yields, $\bbP_{\mu}$-almost surely,
\begin{equation}\label{eq:CondExpMgX}
\bbP_{\mu}\big[W_{s+t}^{\phi}(X)|\bbF_t\big]=e^{-\lambda_c t}\la e^{-\lambda_c s} S_s^{*}\phi,X_t\ra+e^{-\lambda_c t}\Big\la e^{-\lambda_c s}\frac{S_s\phi}{w},Z_t\Big\ra -e^{-\lambda_ct}\Big\la e^{-\lambda_c s} \frac{S_s^*\phi}{w},Z_t \Big\ra,
\end{equation}
and we are interested in the limit as $s \to \infty$. The first and last summand tend to zero $\bbP_{\mu}$-almost surely by Lemma~\ref{lem:ext_lim} and the dominated convergence theorem. 
The second summand is independent of $s$ since $e^{-\lambda_c s} S_s\phi=\phi$. Hence, the right-hand side of \eqref{eq:CondExpMgX} converges to $W_t^{\phi/w}(Z)$. 
By Theorem~\ref{thm:Lp}, $((W_t^{\phi}(X))_{t \ge 0}; {\bf{P_{\mu}}})$ is an $L^p$-bounded martingale and we can interchange on the left-hand side of \eqref{eq:CondExpMgX} the limit $s \to \infty$ with the integration to obtain
\begin{equation}\label{eq:ZasCondExp}
\bbP_\mu\big[\mglX|\bbF_t\big]=\lim_{s \to \infty} \bbP_\mu\big[W_{s+t}^{\phi}(X)|\bbF_t\big]=W_t^{\phi/w}(Z) \qquad \bbP_{\mu}\text{-almost surely.}
\end{equation}
Since $\mglX$ is measurable with respect to $\bbF_{\infty}=\sigma(\bigcup_{t \ge 0} \bbF_t)$, \eqref{eq:mglimAgree} follows by taking $t \to \infty$ in \eqref{eq:ZasCondExp}. 
Moreover, \eqref{eq:ZasCondExp} shows that $(W_t^{\phi/w}(Z))_{t\ge 0}$ is a uniformly integrable martingale and sine $\mglX=\mglZ$ is in $L^p(\bbP_{\mu})$, $L^p$-boundedness of $(W_t^{\phi/w}(Z))_{t\ge 0}$ follows by Jensen's inequality.
\end{proof}

\subsection{Convergence in $L^1(P_{\mu})$}\label{sec:L1Conv}

In this section, we prove the WLLN in the form of Theorem~\ref{thm:L1}. We suppose that Assumptions~\ref{as:skeleton}, \ref{as:criticality} and \eqref{alpha_cond}--\eqref{phi_p_cond} hold and begin with an $L^p$-estimate for the immigration that occurred after a large time $t$. Recall Notations~\ref{not:Z} and \ref{not:I}.

\begin{proposition}\label{pro:Lpsum}
For every $\mu \in \M_c(D)$ and $f \in p(D)$ with $f/\phi$ bounded, there exists a time $T>0$ and a constant $C_3\in (0,\infty)$ such that for all $s \ge 0$, $t\ge T$,
\[
e^{-\lambda_c  p(s+t) }\bbP_{\mu}\Big[\Big|\sum_{i=1}^{N_{t}} \Big(\la f,I_{s}^{i,t}\ra-\bbP_{\bullet,\delta_{\xi_i(t)}}[\la f,I_{s}\ra]\Big)\Big|^p\Big] \le C_3 e^{-\lambda_c  (p-1)t}.
\]
\end{proposition}

\begin{proof}
For $\mu\equiv 0$ the claim is trivial. Let $\mu \not\equiv 0$. It was noted in \cite[Lemma 1]{Big92} that for $p \in [1,2]$, $n \in \N$ and $(Y_i\colon i\in \{1,\ldots,n\})$ independent, centered random variables (or martingale differences)
\[
P\Big[\Big|\sum_{i=1}^n Y_i\Big|^p\Big]\le 2^p \sum_{i=1}^n P\big[|Y_i|^p\big].
\]
For $s,t \ge 0$, we first apply this inequality to $\bbP_{\mu}[\cdot|\bbF_t]$, $n=N_t$ and $Y_i= \la f,I_{s}^{i,t} \ra-\bbP_{\bullet,\delta_{\xi_i(t)}}[\la f,I_{s} \ra]$, 
and then use $|x-y|^p \le x^p+y^p$ for $x,y \ge 0$, $(I_s^{i,t};\bbP_{\mu}(\cdot|\bbF_t))\overset{d}{=}(I_s;\bbP_{\bullet,\delta_{\xi_i(t)}})$ and Jensen's inequality to obtain
\begin{align*}
\bbP_{\mu}\Big[\Big|\sum_{i=1}^{N_{t}} \Big(\la f,I_{s}^{i,t}\ra-\bbP_{\bullet,\delta_{\xi_i(t)}}[\la f,I_{s}\ra]\Big)\Big|^p\Big|\bbF_t\Big]&\le 2^p\sum_{i=1}^{N_{t}}  \bbP_{\mu}\big[\big|\la f,I_{s}^{i,t}\ra-\bbP_{\bullet,\delta_{\xi_i(t)}}[\la f,I_{s}\ra]\big|^p\big|\bbF_t\big]\\
&\le 2^p\sum_{i=1}^{N_{t}} \Big( \bbP_{\bullet,\delta_{\xi_i(t)}}[\la f,I_{s}\ra^p]+\bbP_{\bullet,\delta_{\xi_i(t)}}[\la f,I_{s}\ra]^p\Big)\\
&\le 2^{p+1} \sum_{i=1}^{N_{t}}  \bbP_{\bullet,\delta_{\xi_i(t)}}[\la f,I_{s}\ra^p].
\end{align*}
Lemma~\ref{lem:bb_meas}, the identity $X_s=X_s^*+I_s$ under $\bbP_{\delta_{\xi_i(t)}}$ and the monotonicity of $x \mapsto x^p$ on $[0,\infty)$ yield
\begin{align*}
\bbP_{\mu}\Big[\Big|\sum_{i=1}^{N_{t}} \Big(\la f,I_{s}^{i,t}\ra-\bbP_{\bullet,\delta_{\xi_i(t)}}[\la f,I_{s}\ra]\Big)\Big|^p\Big|\bbF_t\Big] &\le 2^{p+1} \sum_{i=1}^{N_t} \frac{\bbP_{\delta_{\xi_i(t)}}[\la f,X_s \ra^p]}{w(\xi_i(t))}=2^{p+1} \Big\la \frac{\bbP_{\delta_{\cdot}}[\la f,X_{s} \ra^p]}{w},Z_t\Big\ra.
\end{align*}
Writing $c=\sup_{x \in D}f(x)/\phi(x)<\infty$, the Many-to-One Lemma for $Z$, i.e.\ \eqref{eq:mtoZ}, yields
\begin{align*}
e^{-\lambda_c  p(s+t) }\bbP_{\mu}\Big[\Big|\sum_{i=1}^{N_{t}} \Big(\la f,I_{s}^{i,t}\ra-\bbP_{\bullet,\delta_{\xi_i(t)}}[\la f,I_{s}\ra]\Big)\Big|^p\Big] 
& \le 2^{p+1} c^p e^{-\lambda_c pt}\bbP_{\mu}\Big[\Big\la \frac{\bbP_{\delta_{\cdot}}[W_s^{\phi}(X)^p]}{w},Z_t\Big\ra\Big]\\
&=2^{p+1} c^p e^{-\lambda_c  (p-1)t} \la \phi, \mu \ra \P_{\phi \mu}^{\phi}\big[\phi(\xi_t)^{-1}\bbP_{\delta_{\xi_t}}[W_s^{\phi}(X)^p]\big].
\end{align*}
Since $\mu \in \M_c(D)$, Lemma~\ref{lem:mg_moments} yields a time $T>0$ and a constant $C_2\in (0,\infty)$ such that the right-hand side is bounded by 
$2^{p+1}c^p\pint{\phi,\mu} C_2 e^{-\lambda_c (p-1)t}$ for all $t \ge T$ and the proof is complete.
\end{proof}

We are now in the position to prove Theorems~\ref{thm:L1} and \ref{thm:SLLN_gen}(ii).

\begin{proof}[Proof of Theorems~\ref{thm:L1} and \ref{thm:SLLN_gen}(ii)]
By Lemma~\ref{lem:muComp}(ii) it suffices to consider $\mu \in \M_c(D)$ and, without loss of generality, we work on the skeleton space. Using the skeleton decomposition in the form of \eqref{eq:bb}, we write for $s,t \ge 0$,
\begin{align*}
e^{-\lambda_c (s+t)} \pint{f,&X_{s+t}} - \pint{f,\tphi} \mglX\\
&=e^{-\lambda_c (s+t)} \pint{f,X_{s+t}^*+I_s^{*,t}}+ e^{-\lambda_c (s+t)} \sum_{i=1}^{N_t} \Big( \pint{f,I_s^{i,t}} - \bbP_{\bullet, \delta_{\xi_i(t)}}[\pint{f,I_s}] \Big)\\
&+ \Big( e^{-\lambda_c (s+t)} \sum_{i=1}^{N_t} \bbP_{\bullet, \delta_{\xi_i(t)}}[\pint{f,I_s}]- \pint{f,\tphi} W_t^{\phi/w}(Z)\Big)
+ \Big(\pint{f,\tphi} W_t^{\phi/w}(Z) - \pint{f,\tphi} \mglX\Big)\\
&=: \Xi_1(s,t)+\Xi_2(s,t)+ \Xi_3(s,t) + \Xi_4(s,t).
\end{align*}
It suffices to show that
\begin{equation}\label{eq:L1thm4partsTS}
\limsup_{s\to \infty}\limsup_{t\to \infty}\bbP_{\mu}[|\Xi_i(s,t)|]=0 \qquad \text{for all } i\in \{1,\ldots,4\}. 
\end{equation}
We begin with $\Xi_1$: Since $(X_{s+t}^*+I_s^{*,t};\bbP_{\mu}(\cdot|\bbF_t))\overset{d}{=}(X_s^*;\bbP_{X_t})$, the first moment formulas \eqref{eq:opp_star}, \eqref{eq:opp} and \eqref{eq:Sphi} yield
\[
\bbP_{\mu}[|\Xi_1(s,t)|]= e^{-\lambda_c (s+t)} \bbP_{\mu}[\bbP_{X_t}[\pint{f,X_s^*}]] = e^{-\lambda_c (s+t)} \pint{S_t S_s^*f, \mu} = \pint{\P_{\cdot}^{\phi}[\theta_s^*(\xi_t)], \phi \mu},
\]
where $\theta_s^*(x)=e^{-\lambda_c s} S_s^*f(x)/\phi(x)$. By Lemma~\ref{lem:ext_lim}, $\theta_s^*(x)$ is uniformly bounded in $s$ and $x$ and converges to zero as $s \to \infty$. 
Using the ergodicity of $(\xi;\P^{\phi})$, cf.\ \eqref{eq:ergodic}, and the dominated convergence theorem, we conclude
\[
\lim_{s \to \infty} \lim_{t \to \infty} \bbP_{\mu}[|\Xi_1(s,t)|]= \lim_{s\to \infty} \pint{\theta_s^*,\phi\tphi} \pint{\phi,\mu}=0.
\]
Proposition~\ref{pro:Lpsum} implies that $\Xi_2(s,t)$ converges to zero in $L^p(\bbP_{\mu})$ as $t\to \infty$ for every fixed $s>0$. 
By monotonicity of norms, \eqref{eq:L1thm4partsTS} for $i=2$ follows.

For $\Xi_3$, we use \eqref{eq:Eimmi} and \eqref{eq:Sphi} to rewrite
\[
e^{-\lambda_c (s+t)} \sum_{i=1}^{N_t} \bbP_{\bullet, \delta_{\xi_i(t)}}[\pint{f,I_s}] = e^{-\lambda_c (s+t)} \Bigpint{\frac{S_sf-S_s^*f}{w},Z_t} =e^{-\lambda_c t}  \Bigpint{ \P_{\cdot}^{\phi}\big[f(\xi_s)/\phi(\xi_s)\big]- \theta_s^*,\frac{\phi}{w} Z_t}.
\]
Let $\Upsilon_s(x):=\P_x^{\phi}[f(\xi_s)/\phi(\xi_s)]-\theta_s^*(x)$.  Since $f/\phi$ is bounded, $\Upsilon$ is uniformly bounded in $s$ and $x$, 
and by \eqref{eq:ergodic} and Lemma~\ref{lem:ext_lim}, $\lim_{s\to \infty}\Upsilon_s(x)= \pint{f,\tphi}$. 
Moreover, $\Xi_3(s,t)= e^{-\lambda_c t} \pint{\Upsilon_s-\pint{f,\tphi},\frac{\phi}{w} Z_t}$ by the definition of $W_t^{\phi/w}(Z)$
and the Many-to-One Lemma for $Z$, i.e.\ \eqref{eq:mtoZ}, yields
\[
\bbP_{\mu}[|\Xi_3(s,t)|]\le e^{-\lambda_c t} \bbP_{\mu}\Big[ \Bigpint{\big|\Upsilon_s-\pint{f,\tphi}\big| ,\frac{\phi}{w}  Z_t}\Big]= \bigpint{\P_{\cdot}^{\phi}\big[\big|\Upsilon_s(\xi_t)-\pint{f,\tphi}\big|\big],\phi \mu}.
\]
Since $\Upsilon_s$ is bounded and $\phi(x) \mu(dx)$ is a finite measure, \eqref{eq:ergodic} implies that the right-hand side converges to $\pint{|\Upsilon_s-\pint{f,\tphi}|,\phi \tphi} \pint{\phi,\mu}$ as $t\to \infty$ 
and this expression converges to zero by the dominated convergence theorem as $s \to \infty$.

Finally, since $(W_t^{\phi/w}(Z))_{t \ge 0}$ is an $L^p(\bbP_{\mu})$-bounded martingale by Theorem~\ref{thm:SLLN_gen}(i), it converges to $\mglZ=\mglX$ in $L^1(\bbP_{\mu})$. Hence, \eqref{eq:L1thm4partsTS} for $i=4$ holds and the proof is complete.
\end{proof}

\subsection{Asymptotics for the immigration process and the SLLN along lattice times}\label{sec:cond_exp}

In this section, we analyse the asymptotic behaviour of the immigration process $I$. 
According to Lemma~\ref{lem:reduc}~(i), the process $X$ in $e^{-\lambda_c t} \la f,X_t\ra$ can be replaced by the immigration process when $e^{-\lambda_c t} \la f,I_t\ra$ converges to the correct limit. 
To show this, we begin with the conditional expectation of the immigration after a large time $t$ studied in \eqref{eq:Eimmi}. 
We now work under Assumptions~\ref{as:skeleton}, \ref{as:criticality}, \ref{as:moment} and \ref{as:SLLN}.

\begin{lemma}\label{expectU_limitt}
For all $\mu \in \M_c(D)$, $s,\delta>0$ and all $f \in p(D)$ with $f/\phi$ bounded,
\[
\lim_{\delta \N \ni t \to \infty} e^{-\lambda_c t} \sum_{i=1}^{N_t} \bbP_{\bullet,\delta_{\xi_i(t)}}[\la f,I_s\ra] = e^{\lambda_c s} \Big(\la f,\tphi\ra -\la e^{-\lambda_c s} S_s^*f,\tphi\ra\Big) \mglX \qquad \bbP_{\mu}\text{-almost surely.}
\]
\end{lemma}

\begin{proof}
We apply the SLLN for the skeleton (Assumption~\ref{as:SLLN}) to the functions $f_1,f_2$ given by
\begin{align*}
f_1(x):=\frac{S_sf(x)}{w(x)}= e^{\lambda_c s} \P_x^{\phi}\big[f(\xi_s)/\phi(\xi_s)\big]\frac{\phi(x)}{w(x)}, \qquad 
f_2(x):=\frac{S_s^*f(x)}{w(x)}= e^{\lambda_c s} \theta_s^*(x)\frac{\phi(x)}{w(x)}.
\end{align*}
By Lemmas~\ref{lem:wcont} and~\ref{lem:ext_lim} and Theorem~4.9.7 in \cite{Pin95}, $f_1$ and $f_2$ are continuous. Hence, \eqref{eq:Eimmi} and Assumption~\ref{as:SLLN} yield
\[
\lim_{\delta \N \ni t \to \infty} e^{-\lambda_c t} \sum_{i=1}^{N_t} \bbP_{\bullet,\delta_{\xi_i(t)}}[\la f,I_s\ra] = \Big\la \frac{S_sf}{w},w \tphi\Big\ra W_{\infty}^{\phi/w}(Z)-\Big \la \frac{S_s^*f}{w},w\tphi\Big \ra W_{\infty}^{\phi/w}(Z) \qquad \bbP_{\mu}\text{-almost surely.}
\]
By \eqref{eq:Sphi} and \eqref{eq:statDistrSphi}, $\pint{S_sf,\tphi}=e^{\lambda_c s} \pint{f,\tphi}$ and Theorem~\ref{thm:SLLN_gen}(i) completes the proof.
\end{proof}

\begin{proposition}\label{pro:lim_Eimmi}
For all $\mu \in \M_c(D)$, $\delta>0$ and all $f \in p(D)$ with $f/\phi$ bounded,
\[
\lim_{\delta \N \ni s\to \infty} \lim_{\delta \N \ni t \to \infty} e^{-\lambda_c (s+t)} \sum_{i=1}^{N_t} \bbP_{\bullet,\delta_{\xi_i(t)}}[\la f,I_s\ra] = \la f,\tphi\ra W_{\infty}^{\phi}(X)\qquad \bbP_{\mu}\text{-almost surely.}
\]
\end{proposition}

\begin{proof} The claim follows immediately from Lemmas~\ref{expectU_limitt} and \ref{lem:ext_lim} and the dominated convergence theorem.
\end{proof}

Recall from Notation~\ref{not:I} that, given $\bbF_t$, $I^{i,t}$ denotes the immigration occurring along the skeleton descending from particle $i$ at time $t$.

\begin{proposition}\label{pro:BC_immi}
For all $\mu \in \M_c(D)$, $s,\delta>0$ and all $f \in p(D)$ with $f/\phi$ bounded,
\[
\lim_{\delta \N \ni t \to \infty} e^{-\lambda_c (s+t)} \Big|\sum_{i=1}^{N_t} \Big(\la f,I_s^{i,t}\ra-\bbP_{\bullet,\delta_{\xi_i(t)}}[\la f,I_s\ra]\Big)\Big|=0 \qquad \bbP_{\mu}\text{-almost surely.}
\]
\end{proposition}

\begin{proof}
By the Borel-Cantelli Lemma, it is sufficient to show that for all $\eps>0$, there is a $n_0 \in \N$ such that
\begin{equation}\label{eq:BC_immi1}
\sum_{n=n_0}^{\infty} \bbP_{\mu}\Big(e^{-\lambda_c(s+n\delta)}\Big|\sum_{i=1}^{N_{n\delta}} \Big(\la f,I_s^{i,n\delta}\ra-\bbP_{\bullet,\delta_{\xi_i(n\delta)}}[\la f,I_s\ra]\Big)\Big| > \eps\Big)<\infty.
\end{equation}
Proposition~\ref{pro:Lpsum} yields a time $T>0$ and a constant $C_3 \in (0,\infty)$ such that for $n_0 \ge T$, the left-hand side in \eqref{eq:BC_immi1} is bounded by
\[
\eps^{-p} \sum_{n =n_0}^{\infty} e^{-\lambda_c p(s+n\delta)} \bbP_{\mu}\Big[\Big|\sum_{i=1}^{N_{n\delta}} \Big(\la f,I_s^{i,n\delta}\ra-\bbP_{\bullet,\delta_{\xi_i(n\delta)}}[\la f,I_s\ra]\Big)\Big|^p\Big] \le \eps^{-p} C_3 \sum_{n=n_0}^{\infty} e^{-\lambda_c (p-1) n \delta}<\infty,
\]
where we used Markov's inequality in the first estimate.
\end{proof}

Combining Propositions~\ref{pro:lim_Eimmi} and \ref{pro:BC_immi}, we conclude that the immigration process alone has the asymptotic behaviour which we expect from the superprocess.  
\begin{corollary}\label{cor:immiSLLN}
For all $\mu \in \M_c(D)$, $\delta>0$ and all $f \in p(D)$ with $f/\phi$ bounded,
\[
\lim_{\delta \N \ni s \to \infty}\lim_{\delta \N \ni t \to \infty} e^{-\lambda_c (s+t)} \sum_{i=1}^{N_t} \la f,I_s^{i,t}\ra =\la f,\tphi\ra \mglX \qquad \bbP_{\mu}\text{-almost surely.}
\]
\end{corollary}

Now we are in the position to prove the SLLN along lattice times.

\begin{theorem}\label{SLLN_lattice}
For all $\mu \in \M_f^{\phi}(D)$, $\delta>0$ and all $f \in p(D)$ with $f/\phi$ bounded,
\[
\lim_{n \to \infty} e^{-\lambda_c n\delta} \la f,X_{n\delta}\ra =\la f,\tphi\ra \mglX \qquad P_{\mu}\text{-almost surely.}
\]
\end{theorem}

\begin{proof}
By Lemma~\ref{lem:muComp}(i), it suffices to consider $\mu \in\M_c(D)$. Moreover, without loss of generality, we work on the skeleton space from Theorem~\ref{thm:KPR}. Corollary~\ref{cor:immiSLLN} yields $\bbP_{\mu}$-almost everywhere,
\begin{align*}
\liminf_{\delta \N \ni t \to \infty} e^{-\lambda_c t} \la f,X_t\ra&= \liminf_{\delta \N \ni s \to \infty} \liminf_{\delta \N \ni t \to \infty} e^{-\lambda_c (s+t)} \la f,X_{s+t}\ra\\
&\ge \liminf_{\delta \N \ni s \to \infty} \liminf_{\delta \N \ni t \to \infty} e^{-\lambda_c (s+t)} \sum_{i=1}^{N_t} \la f,I_s^{i,t}\ra= \la f,\tphi\ra \mglX .
\end{align*}
Lemma~\ref{lem:reduc}(i) yields the claim.
\end{proof}

\subsection{Transition from lattice to continuous times}\label{lattice_to_cont_sec}

In this section, we extend the convergence along lattice times in Theorem~\ref{SLLN_lattice} to convergence along continuous times and conclude our main results. We work under Assumptions~\ref{as:skeleton}, \ref{as:criticality},~\ref{as:moment} and~\ref{as:SLLN}. For $\res>0$, let $U^{\res}$ be the resolvent operator in integral form, that is,
\[
U^{\res}f(x):=\int_{0}^{\infty} e^{-\res t} \P_x^{\phi}[f(\xi_t)] \, dt \qquad \text{for all }f \in bp(D), x \in D.
\]
The argument for the transition from lattice to continuous times proceeds in two steps. First we use the resolvent operator to bring the semigroup of $(\xi;(\P_x^{\phi})_{x \in D})$ into the argument. The semigroup property gives us a martingale which, combined with Doob's $L^p$-inequality, enables us to control the behaviour between time $n\delta$ and $(n+1)\delta$.
Second, we remove the resolvent operator by taking $\res \to \infty$ in $\res U^{\res}f(x)$. It is an analysis of hitting times for diffusion processes which allows us to control the convergences in this step. 
The main idea for the proof is borrowed from \cite{LiuRenSon13} but we employ the skeleton decomposition to replace the stochastic analysis and the martingale measures used there.

\begin{proposition}\label{SLLN_resolvent}
For all $\mu \in \M_c(D)$, $\res>0$ and $f \in bp(D)$,
\[
\lim_{t \to \infty} e^{-\lambda_c t} \la \phi \res U^{\res}f,X_{t}\ra =\la\phi f, \tphi \ra \mglX \qquad P_{\mu}\text{-almost surely.}
\]
\end{proposition}

\begin{proof}
Without loss of generality, we assume that $\mu\not\equiv 0$ and work on the skeleton space. Since $\res U^{\res}$ is linear with $\res U^{\res}\mathbf{1}= \mathbf{1}$ the same argument that led to Lemma~\ref{lem:reduc} shows that it suffices to prove that for all $f \in bp(D)$,
\begin{equation}\label{eq:resLB}
\liminf_{t \to \infty} e^{-\lambda_c t} \la \phi \res U^{\res} f,X_t\ra \ge\la \phi f,\tphi \ra \mglX  \qquad \bbP_{\mu}\text{-almost surely.}
\end{equation}
Let $f,g \in bp(D)$ with $\res U^{\res}f \ge g$, $\delta, t >0$ and let $n$ be such that $n\delta \le t <(n+1)\delta$. Then
\begin{equation}\label{eq:resDecomp}
\begin{split}
e^{-\lambda_c t} \la \phi \res U^{\res} f,X_t\ra&\ge  \Big(e^{-\lambda_c t} \la \phi \res U^{\res}f,X_t\ra - e^{-\lambda_c t} \la \phi \P_{\cdot}^{\phi}\big[\res U^{\res}f(\xi_{(n+1)\delta -t})\big] ,X_t\ra\Big)\\
&\phantom{\ge a} +\Big( e^{-\lambda_c t} \la \phi \P_{\cdot}^{\phi}[ g(\xi_{(n+1)\delta -t})],X_t\ra- e^{-\lambda_c n \delta} \la \phi \P_{\cdot}^{\phi}[g(\xi_{\delta})], X_{n \delta}\ra\Big)\\
&\phantom{\ge a} +\Big( e^{-\lambda_c n \delta} \la \phi \P_{\cdot}^{\phi}[g(\xi_{\delta})], X_{n \delta}\ra - \la \phi g,\tphi\ra \mglX\Big) + \la \phi g,\tphi\ra \mglX\\
&=: \Theta_{1, \res U^{\res}f}(n,\delta,t)+\Theta_{2, g}(n,\delta,t)+\Theta_{3, g}(n,\delta)+\la \phi g,\tphi\ra \mglX.
\end{split}
\end{equation}
If we show, for all $f,g \in bp(D)$, $g$ of compact support, that $\bbP_{\mu}$-almost surely,
\begin{align}
& \limsup_{\delta \to 0} \limsup_{n \to \infty} \sup_{t \in [n \delta, (n+1)\delta]} |\Theta_{1, \res U^{\res}f}(n,\delta,t)| =0, \label{eq:Theta1}\\
& \limsup_{n \to \infty} \sup_{t \in [n \delta, (n+1)\delta]} |\Theta_{2, g}(n,\delta,t)| =0 \qquad \text{for all }\delta>0, \label{eq:Theta2}\\
& \limsup_{n \to \infty} |\Theta_{3, g}(n,\delta)| =0 \qquad \qquad \qquad\qquad \text{for all }\delta>0, \label{eq:Theta3}
\end{align}
then we can choose $g=\mathbbm{1}_B \res U^{\res}f$ for $B\subset \subset D$ in \eqref{eq:resDecomp} to obtain $\liminf_{t \to \infty} e^{-\lambda_c t} \la \phi \res U^{\res} f,X_t\ra \ge \la \phi \mathbbm{1}_B \res U^{\res}f,\tphi\ra \mglX$ $\bbP_{\mu}$-almost surely. Choosing a sequence $B_j \subset \subset D$, $B_j \subseteq B_{j+1}$, $D=\bigcup_{j=1}^{\infty} B_j$, the factor $\pint{ \phi \mathbbm{1}_{B_j} \res U^{\res}f,\tphi}$ increases, as $j \to \infty$, to
\[
\la \phi \res U^{\res}f, \tphi\ra =\int_0^{\infty} \res e^{-\res t} \la \phi \P_{\cdot}^{\phi}[f(\xi_t)], \tphi\ra \, dt\overset{\eqref{eq:statDistrSphi}}{=}\int_0^{\infty} \res e^{-\res t} \la \phi f,\tphi\ra \, dt=\la \phi f,\tphi \ra,
\]
and \eqref{eq:resLB} follows.
We begin with the proof of \eqref{eq:Theta1}. Fubini's theorem and the Markov property of $(\xi;\P^{\phi})$ yield for all $x \in D$ and $s>0$,
\begin{align*}
\big|\res U^{\res}f(x)-\P_{x}^{\phi}[\res U^{\res}f(\xi_s)]\big|& =\Big|\int_{0}^{\infty} \res e^{-\res t} \P_x^{\phi}[f(\xi_t)] \, dt-\int_{0}^{\infty} \res e^{-\res t} \P_x^{\phi}[f(\xi_{t+s})]\, dt\Big|
\le 2 (1-e^{-\res s})\supnorm{f}.
\end{align*}
Using the linearity of integration and the definition of $W_t^{\phi}(X)$, we obtain
\[
\sup_{t \in [n\delta,(n+1)\delta]} \Theta_{1, \res U^{\res}f}(n,\delta,t) \le 2(1-e^{-\res\delta})\supnorm{f}  \sup_{t \in [n\delta,(n+1)\delta]} W_t^{\phi}(X).
\]
Since the martingale $(W_t^{\phi}(X))_{ t\ge 0}$ has a finite limit, \eqref{eq:Theta1} is established. For the proof of \eqref{eq:Theta2}, let $g \in bp(D)$ be compactly supported. By \eqref{eq:Sphi},
\[
\Theta_{2,g}(n,\delta,t)= e^{-\lambda_c (n+1)\delta} \Big(\la S_{(n+1)\delta -t}[\phi g],X_t\ra- \la S_{\delta}[\phi g],X_{n\delta}\ra \Big) \qquad \text{for all }t\in [0,(n+1)\delta],
\]
and, the Markov property of $X$ and \eqref{eq:opp} imply that $(\Theta_{2,g}(n,\delta,t)\colon t \in [n\delta,(n+1)\delta]; \bbP_{\mu})$ is a martingale. Hence, \eqref{eq:Theta2} follows from the Borel-Cantelli Lemma, Doob's $L^p$-inequality (cf.\ Theorem II.1.7 in \cite{RY99}) and $\bbP_{\mu}[\la \phi g, X_{(n+1)\delta}\ra|\bbF_{n\delta}]=\la S_{\delta}[\phi g],X_{n \delta}\ra$ when we prove that for sufficiently large $n_0$
\begin{equation}\label{BC_to_show}
\sum_{n=n_0}^{\infty} e^{-\lambda_c p(n+1)\delta} \bbP_{\mu}\big[ \big|\la \phi g,X_{(n+1)\delta}\ra-\bbP_{\mu}[\la \phi g,X_{(n+1)\delta} \ra|\bbF_{n\delta}]\big|^p\big] <\infty.
\end{equation}
By \eqref{eq:bb} and \eqref{eq:bb_cond}, we can use the skeleton decomposition to obtain, for all $s,t>0$, $\bbP_{\mu}$-almost surely,
\begin{equation}\label{eq:latContBB}
\begin{split}
\la \phi g,X_{s+t}\ra-&\bbP_{\mu}[\la \phi g,X_{s+t} \ra|\bbF_t]\\
&=\pint{\phi g,X_{s+t}^*+I_s^{*,t}}-\bbP_{X_t}[\pint{\phi g,X_s^*}] + \sum_{i=1}^{N_t} \Big(\pint{\phi g, I_s^{i,t}}-\bbP_{\bullet, \delta_{\xi_i(t)}}[\pint{\phi g,I_s}]\Big).
\end{split}
\end{equation}
The monotonicity of $L^p$-norms and $(X_{s+t}^*+I_s^{*,t};\bbP_{\mu}(\cdot|\bbF_t))\overset{d}{=} (X_s^*;\bbP_{X_t})$ imply 
\begin{equation}\label{eq:pvsVar}
\bbP_{\mu}\big[\big| \la  \phi g,X_{s+t}^*+I_s^{*,t}\ra-\bbP_{X_t}[\la  \phi g ,X_s^*\ra] \big|^p\big]\le \bbP_{\mu}\big[{ \rm{\bf{Var}}}_{X_t}(\la  \phi g, X_s^*\ra) \big]^{p/2}.
\end{equation} 
Denote $c_1^*(x)= 2 \alpha(x)$, $c_2^*=\int_{(0,\varphi_2(x)]} y^2\, \Pi^*(x,dy)$, $c_3^*(x)= \int_{(\varphi_2(x),\infty)} y^2\, \Pi^*(x,dy)$, $c^*(x)=\sum_{i=1}^3c_i^*(x)$ for all $x\in D$, where $\varphi_2$ is determined by Assumption~\ref{as:moment}. We notice that $\beta^* \le \beta$ pointwise implies $S_t^*f \le S_tf$ pointwise for all $f \in p(D)$. Using \eqref{eq:variance_star}, \eqref{eq:opp}, and the semigroup property of $S$, we obtain
\[
\bbP_{\mu}\big[{ \rm{\bf{Var}}}_{X_t}(\la  \phi g, X_s^*\ra) \big] \le \int_0^s \pint{S_tS_r^*\big[c^* (S_{s-r}^*[\phi g])^2\big],\mu}\, dr \le \int_0^s \pint{S_{t+r}\big[c^* (S_{s-r}[\phi g])^2\big],\mu}\, dr.
\]
Recall the definition of $\P_{\phi \mu}^{\phi}$ from \eqref{eq:Pphimu} and use \eqref{eq:Sphi} to deduce
\begin{equation}\label{eq:VarXst2ndbd}
\bbP_{\mu}\big[{ \rm{\bf{Var}}}_{X_t}(\la  \phi g, X_s^*\ra) \big] \le \pint{\phi,\mu} \supnorm{g} \int_0^s e^{\lambda_c (s+t)} \P_{\phi \mu}^{\phi}\big[c^*(\xi_{t+r}) S_{s-r}[\phi g](\xi_{t+r}) \big]\, dr .
\end{equation}
Writing $\bar{\beta}=\sup_{x\in D} \beta(x)$ and $M=\max\{e^{\bbeta s},1\}$, we notice that
$S_{s-r}[\phi g](x)\le M \supnorm{\phi g}$. Moreover, \eqref{eq:Sphi} implies $S_{s-r}[\phi g](x)\le e^{\lambda_c s} \supnorm{g} \phi(x)$. Hence, for all $i \in \{1,2,3\}$,
\begin{equation}\label{eq:cibd}
\P_{\phi \mu}^{\phi}\big[c_i^*(\xi_{t+r}) S_{s-r}[\phi g](\xi_{t+r})\big] \le \min\Big\{ M \supnorm{\phi g} \P_{\phi \mu}^{\phi}[c_i^*(\xi_{t+r})], e^{\lambda_c s} \supnorm{g} \P_{\phi \mu}^{\phi}\big[c_i^*(\xi_{t+r}) \phi(\xi_{t+r})\big]\Big\}.
\end{equation}
Since $\alpha$ is bounded, $c_1^*$ is bounded. For $c_2^*$ and $c_3^*$, the right-hand side of \eqref{eq:cibd} is bounded for large $t$ according to \eqref{eq:ergodic} and 
Conditions~\eqref{as:lat_con_Lower} and \eqref{as:lat_con_Upper}, respectively.
Combining \eqref{eq:pvsVar}, \eqref{eq:VarXst2ndbd} and \eqref{eq:cibd}, we obtain a time $T>0$ and a 
constant $C \in (0,\infty)$, which may depend on $s, g$ and $\mu$, such that
\begin{equation}\label{eq:ExtPartSummable}
e^{-\lambda_c p(s+t)} \bbP_{\mu}\big[\big| \la  \phi g,X_{s+t}^*+I_s^{*,t}\ra-\bbP_{X_t}[\la  \phi g ,X_s^*\ra] \big|^p\big] \le C e^{-\lambda_c p t/2} \qquad \text{for all }t\ge T.
\end{equation}
Since $|x+y|^p \le 2^p (|x|^p+|y|^p)$ for all $x,y \in \R$, \eqref{eq:latContBB}, \eqref{eq:ExtPartSummable} and Proposition~\ref{pro:Lpsum} yield \eqref{BC_to_show}.

It remains to prove \eqref{eq:Theta3}. To this end, note that for $g \in bp(D)$ and $\delta>0$, $x \mapsto \phi(x) \P_x^{\phi}[g(\xi_{\delta})]$ is continuous by Theorem~4.9.7 in \cite{Pin95}. 
Moreover, $\pint{\phi \P_{\cdot}^{\phi}[g(\xi_{\delta})],\tphi } =\la \phi g,\tphi \ra$ by \eqref{eq:statDistrSphi} and \eqref{eq:Theta3} follows from Theorem~\ref{SLLN_lattice}.
\end{proof}

In the second step we remove the resolvent operator from Proposition~\ref{SLLN_resolvent}. 
The proof is essentially borrowed from the lower bound in Theorem 2.2 in \cite{LiuRenSon13}. 
We present the argument here for completeness. Recall that $\B_0(D)=\{B \in \B(D)\colon \ell(\partial B)=0\}$ and $\phi|_B=\phi \mathbbm{1}_B$.

\begin{proposition}\label{pro:SLLN_indic}
For all $\mu \in \M_c(D)$ and $B \in \B_0(D)$,
\begin{equation}\label{eq:inf_SLLN_indic}
\liminf_{t \to \infty} e^{-\lambda_c t} \la \phi|_B ,X_t\ra \ge \la \phi|_B,\tphi\ra W_{\infty}^{\phi}(X)\qquad P_{\mu}\text{-almost surely.}
\end{equation}
\end{proposition}

\begin{proof}
The claim is trivial when $\ell(B)=0$ and when \eqref{eq:inf_SLLN_indic} is proved $B \in \B_0(D)$ with $B\subset \subset D$, then, for arbitrary $B\in \B_0(D)$, we choose a sequence of sets $B_k \in \B_0(D)$, with $B_k \subset \subset D$, $B_k \subseteq B_{k+1}$ and $B=\bigcup_{k \in \N}B_k$ and the monotone convergence theorem yields
\[
\liminf_{t \to \infty} e^{-\lambda_c t} \la \phi|_{B},X_t\ra \ge \sup_{k \in \N} \liminf_{t \to \infty} e^{-\lambda_c t} \la \phi|_{B_k},X_t\ra \ge  \la \phi|_B,\tphi\ra\mglX \qquad P_{\mu}\text{-almost surely}.
\]
Hence, let $B \in \B_0(D)$, $B \subset \subset D$, contain a non-empty, open ball. For small $\eps>0$, let $B_{\eps}=\{x \in B\colon \text{dist}(x,\partial B)\ge \eps\}\not=\emptyset$ and denote by $\sigma_{B_{\eps}}=\inf\{t>0\colon \xi_t \in B_{\eps}\}$ the hitting time of $B_{\eps}$. 
We write $U^{\res}(x,A)=U^{\res}\mathbbm{1}_A(x)$ for all $A\in \B(D)$. 
Since $\{\xi_t \in B_{\eps}\} \subseteq \{\sigma_{B_{\eps}} \le t\}$, for all $x \in D$,
\begin{align*}
\res U^{\res}(x,B_{\eps})\le \int_0^{\infty} \res e^{-\res t} \P_x^{\phi}(\sigma_{B_{\eps}} \le t) \, dt=\P_{x}^{\phi}[e^{-\res \sigma_{B_{\eps}}}] \le \mathbbm{1}_B(x) + \mathbbm{1}_{B^c}(x) \P_{x}^{\phi}[e^{-\res \sigma_{B_{\eps}}}],
\end{align*}
where $B^c:=D\setminus B$. In particular,
\[
e^{-\lambda_c t} \la \phi|_B,X_t \ra \ge e^{-\lambda_c t} \la \phi \res U^{\res}\mathbbm{1}_{B_{\eps}},X_t\ra-e^{-\lambda_c t} \la \phi|_{B^c}\P_{\cdot}^{\phi}[e^{-\res \sigma_{B_{\eps}}}],X_t\ra,
\]
and Proposition~\ref{SLLN_resolvent} yields, $P_{\mu}$-almost surely,
\begin{equation}\label{eq:1stBound}
\liminf_{t \to \infty} e^{-\lambda_c t} \la \phi|_B,X_t \ra \ge \la \phi|_{B_{\eps}},\tphi\ra\mglX -\limsup_{t \to \infty} e^{-\lambda_c t} \la \phi|_{B^c}\P^{\phi}_{\cdot}[e^{-\res \sigma_{B_{\eps}}}],X_t\ra.
\end{equation}
The first term on the right converges to $\la \phi|_B,\tphi \ra \mglX$ as $\eps \to 0$. Thus, we have to show that the second summand vanishes as $\eps \to 0$. 
Heuristically, if the SLLN holds, then the $\limsup$ is a limit with value
\[
\la \phi|_{B^c}\P^{\phi}_{\cdot}[e^{-\res \sigma_{B_{\eps}}}],\tphi\ra \mglX.
\]
Since $B_{\eps}$ has positive distance to $B^c$, this value converges to zero for $\res \to \infty$. 
Hence, we first take $\res \to \infty$ and then $\eps \to 0$. Of course, we do not know the SLLN, yet. 
Thus, we artificially reintroduce the resolvent operator in order to apply Proposition~\ref{SLLN_resolvent}.

Continuing rigorously, let $B_{\eps}':=\{ x \in B\colon \text{dist}(x,\partial B_{\eps}) \le \eps/2\}$. The situation is sketched in Figure 1. 

\begin{center}
\begin{tikzpicture}
[inner sep=0pt]
\colorlet{BlueSLLNa}{blue!20}
\colorlet{BlueSLLNb}{blue!65}
\draw[thick] (0,0) circle (1.5cm); 
\node at (-1.65,1.2) {$B$};
\draw[->, >=stealth] (-1.45,1.2) -- (-1.05,1.05); 
\node at (0,0) [circle,draw=black,fill=BlueSLLNa,minimum size=2.1cm]{};
\node at (0,0) [circle,draw=black,fill=white,minimum size=1.1cm]{};
\draw (0,0) circle (0.8cm); 
\node at (0,0) [circle,draw=black,minimum size=1.6cm,pattern=north west lines]{};
\node at (-1.6,-0.9) {$B_{\eps}$};
\draw[->, >=stealth] (-1.4,-0.85) -- (-0.63,-0.50); 
\node at (1.55,1.2) {\textcolor{BlueSLLNb}{$B_{\eps}'$}};
\draw[->, >=stealth] (1.30,1.2) -- (0.73,0.73); 
\pgfmathsetseed{2};
\draw[thin](4,0.5)
\foreach \x in {1,...,400}
{-- ++(-0.01,rand*0.1+0.0)};
\node at (4,0.5) [circle,draw=black,fill=black,minimum size=0.07cm]{}; 
\node at (4.3,0.5) {$x$};
\node at (0.7071067*0.8,-0.7071067*0.8) [circle,draw=black,fill=black,minimum size=0.07cm]{}; 
\node at (0.8,-0.80) {$y$};

\end{tikzpicture}\\
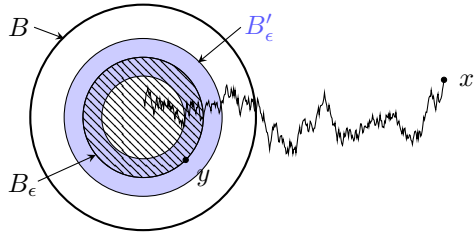
{\small{{\bf{Figure 1:}} The big ball with thick boundary is $B$, the small, hatched ball is $B_{\eps}$ and the shaded area denotes $B_{\eps}'$. The diffusion is started in $x \in B^c$.}}
\end{center}
When $\xi$ starts outside $B$, then $\xi_{\sigma_{B_\eps}} \in \partial B_{\eps}$ and we obtain for all $x \in B^c$,
\begin{equation}\label{hitting_transform}
\P_x^{\phi}[e^{-\res \sigma_{B_{\eps}}}]=\P_x^{\phi}\Big[e^{-\res \sigma_{B_{\eps}}} \frac{U^{\res}(\xi_{\sigma_{B_{\eps}}},B_{\eps}')}{U^{\res}(\xi_{\sigma_{B_{\eps}}},B_{\eps}')}\Big]\le \frac{1}{\inf_{y \in \partial B_{\eps}}U^{\res}(y,B_{\eps}')} \P_x^{\phi}\big[e^{-\res \sigma_{B_{\eps}}} U^{\res}(\xi_{\sigma_{B_{\eps}}},B_{\eps}')\big].
\end{equation}
For $t \ge 0$, let $\mathcal{H}_t:=\sigma(\xi_s: 0 \le s \le t)$. By the Markov property of $\xi$, the second factor on the right-hand side of \eqref{hitting_transform} can be estimated by
\begin{equation}\label{hitting_transform2} 
\begin{split}
\P_x^{\phi}\big[e^{-\res \sigma_{B_{\eps}}} U^{\res}(\xi_{\sigma_{B_{\eps}}},B_{\eps}')\big]&=\P_x^{\phi}\Big[e^{-\res \sigma_{B_{\eps}}} \P_x^{\phi}\Big[\int_0^{\infty} e^{-\res t} \mathbbm{1}\{\xi_{t+\sigma_{B_{\eps}}} \in B_{\eps}'\} \, dt\Big|\mathcal{H}_{\sigma_{B_{\eps}}}\Big]\Big]\\
&=\P_x^{\phi}\Big[\int_{\sigma_{B_{\eps}}}^{\infty} e^{-\res t} \mathbbm{1}\{\xi_{t} \in B_{\eps}'\} \, dt\Big]\le U^{\res}(x,B_{\eps}').
\end{split}
\end{equation}
Writing $\Phi(\res,\eps):=\inf_{y \in \partial B_{\eps}} \res U^{\res}(y,B_{\eps}')$, \eqref{hitting_transform} and \eqref{hitting_transform2} yield $\mathbbm{1}_{B^c}(x)\P_x^{\phi}[e^{-\res \sigma_{B_{\eps}}}] \le \res U^{\res}(x,B_{\eps}')/\Phi(\res,\eps)$ and Proposition~\ref{SLLN_resolvent} entails, $P_{\mu}$-almost surely,
\begin{equation}\label{eq:2ndBound}
\limsup_{t \to \infty} e^{-\lambda_c t} \la \phi|_{B^c}\P^{\phi}_{\cdot}[e^{-\res \sigma_{B_{\eps}}}],X_t\ra \le \frac{1}{\Phi(\res,\eps)} \la \phi|_{B_{\eps}'},\tphi\ra\mglX.
\end{equation}
Clearly, $\la \phi|_{B_{\eps}'},\tphi\ra\mglX$ converges to zero as $\eps\to 0$. 
Thus, it remains to bound $\res U^{\res}(y,B_{\eps}')$ for $y \in \partial B_{\eps}$, and therefore $\Phi(\res,\eps)$, away from zero. 
We write $b_0(x)$ for the vector whose $j$-th component is given by $b_j(x)+\sum_{i=1}^d \partial_{x_i}a_{i,j}(x)$, $j \in \{1,\ldots,d\}$, $x \in D$. Since $B \subset \subset D$, 
\[
c(B,\phi):=\frac{\inf_{x \in B} \phi(x)}{\sup_{x \in B} \phi(x)}, \quad \tilde{\beta}:=\sup_{x \in B} |\beta(x)|, \quad \tilde{b}_0:=\sup_{x \in B}|b_0(x)|, \quad \tilde{a}:=\sup_{x \in B}\sup_{|v|=1} \la v,a(x)v\ra,
\]
satisfy $c(B,\phi),\tilde{a} \in (0,\infty)$ and $\tilde{\beta},\tilde{b}_0 \in [0,\infty)$.
For all $T>0$, we have
\begin{equation}\label{lbound_resolvent}
\res U^{\res}(y,B_{\eps}')= \int_0^{\infty} \res e^{-\res t} \P_{y}^{\phi}(\xi_t \in B_{\eps}')\, dt = \int_0^{\infty} e^{-t} \P_y^{\phi}(\xi_{t/\res} \in B_{\eps}')\, dt\ge\int_0^{T} e^{-t} \P_y^{\phi}(\xi_{t/\res} \in B_{\eps}')\, dt,
\end{equation}
and for $y \in \partial B_{\eps}$, we use the definition of $B_{\eps}'$ and \eqref{eq:Sphi} to estimate
\begin{equation}\label{useBeps_prob}
\P_y^{\phi}(\xi_{t/\res} \in B_{\eps}') \ge \P_y^{\phi}\Big(\sup_{ 0\le s \le t/\res}|\xi_{s}-y| \le \eps/2\Big)\ge c(B,\phi) e^{-(\lambda_c+\tilde{\beta}) t/\res} \P_y\Big(\sup_{0 \le s \le t/\res} |\xi_s-y|\le \eps/2\Big).
\end{equation}
To estimate the probability on the right-hand side, we use Theorem 2.2.2 in \cite{Pin95}. 
Since this theorem is stated for a diffusion generator in non-divergence form, we introduced the function $b_0$. In particular, for $\res$ so large that 
$t\tilde{b}_0/\res \le \eps/4$ and for all $y \in \partial B_{\eps}$, we deduce
\begin{equation}\label{hitting_estimate}
\P_{y}\Big(\sup_{0 \le s\le t/\res}|\xi_s-y| \le \eps/2\Big)\ge 1- 2d \exp\Big(-\frac{\eps^2\res}{32 \tilde{a} td}\Big).
\end{equation}
Combining \eqref{lbound_resolvent}--\eqref{hitting_estimate}, we obtain for all $\eps >0$,
\[
\liminf_{\res \to \infty} \Phi(\res,\eps) \ge \int_0^{T} e^{-t} c(B,\phi)\, dt>0.
\]
Since the right-hand side does not depend on $\eps$, taking first $\res \to \infty$ and then $\eps \to 0$ in \eqref{eq:2ndBound} and \eqref{eq:1stBound} concludes the proof.
\end{proof}

We are now in the position to conclude our main results.

\begin{proof}[Proof of Theorem~\ref{thm:SLLN_gen}(iii)]
The $P_{\mu}$-almost sure convergence in \eqref{eq:SLLN2} for every given $\ell$-almost everywhere continuous $f \in p(D)$ with $f/\phi$ bounded 
follows from Proposition~\ref{pro:SLLN_indic} and Lemmas~\ref{lem:reduc}(i) and \ref{lem:muComp}(i). The existence of a common set $\Omega_0$ for all such test functions follows from Lemma~\ref{lem:omega0}.
\end{proof}

\begin{proof}[Proof of Theorem~\ref{thm:SLLN}]
The $L^1(P_{\mu})$-convergence in \eqref{eq:SLLN} was proved in Theorem~\ref{thm:L1}, the remainder follows from Theorem~\ref{thm:SLLN_gen}(iii).
\end{proof}

\begin{proof}[Proof of Corollary~\ref{cor:SLLN}]
The claim follows immediately from Theorem~\ref{thm:SLLN}, \eqref{eq:mtoX} and \eqref{eq:ergodic}.
\end{proof}

\section{Examples}\label{sec:examples}

In this section, we explore our assumptions by verifying them for many classical examples of superdiffusions from the literature. 
Moreover, we give several examples to illustrate the implications and boundaries of the SLLN. 
For all examples considered, this article proves the SLLN, and for some even the WLLN, for the first time. 

\subsection{Spatially independent branching mechanisms}\label{sec:nonspatial}

In this section, we consider superdiffusions with a conservative motion and a spatially independent branching mechanism and write $\psi(z)=\psi_{\beta}(x,z)$ to simplify notation. 
Under these conditions, the total mass process $(\pint{\mathbf{1},X_t}\colon t\ge 0)$ is a continuous state branching process (CSBP) with branching mechanism $\psi$, cf.\ \cite{She97,BerKypMur11}. 
We exclude the trivial case of a linear branching mechanism $\psi(z)=-\beta z$ (see Remark~\ref{rem:detCase} for the result in this situation).
Since $\psi$ is strictly convex, $\psi(\infty):=\lim_{z \to \infty}\psi(z)$ exists in $[-\infty,0)\cup \{\infty\}$. Writing $z^*=\sup\{z\ge 0\colon \psi(z) \le 0\}$, we have $z^* \in (0,\infty)$ if and only if $\beta>0$ and $\psi(\infty)=\infty$ and in that case (cf.\ Proposition~1.1 in \cite{She97})
\[
P_{\mu}\big[e^{-z^*\pint{\mathbf{1},X_t}}\big]=e^{-z^* \pint{\mathbf{1},\mu}} \qquad \text{for all }\mu \in \M_f(D), t \ge 0.
\]
In particular, Assumption~\ref{as:skeleton} is satisfied with $w(x)=z^*$ for all $x \in D$. In this CSBP context, the skeleton decomposition was proved by Berestycki et al.\ in \cite{BerKypMur11} a few years before \cite{KypPerRen13pre}. 
The martingale function $z^*$ is related to the event of weak extinction $\extinguishing=\{\lim_{t \to \infty} \pint{\mathbf{1},X_t}=0\}$ by the identity $P_{\delta_x}(\extinguishing)=e^{-z^*}$ which holds even if $\beta \le 0$ or $\psi(\infty)<0$. 
To compare the martingale function $w(x)=-\log P_{\delta_x}(\extinguishing)$ to the classical choice $w(x)=-\log P_{\delta_x}(\extinction)$, where $\extinction$ denotes the event of extinction after finite time, notice that $\extinction \subseteq \extinguishing$, and for all $\mu \in \M_f(D)$,
\begin{equation}\label{eq:wExting}
P_{\mu}(\extinction)=P_{\mu}(\extinguishing)=e^{-z^*\pint{\mathbf{1},\mu}}>0  \qquad \text{if }\psi(\infty)=\infty \text{ and }\int^{\infty} \frac{1}{\psi(z)} \, dz<\infty.
\end{equation}
 Othwerwise, $P_{\mu}(\extinction)=0$ and on $\extinguishing$ the total mass of $X$ drifts to zero while staying positive at all finite times, cf.\ \cite{Gre74,She97}.

From now on, assume $\beta>0$, $\psi(\infty)=\infty$ and $w(x)=z^*$. In this case, Assumption~\ref{as:moment1} simplifies to
\[
\phi \text{ bounded } \quad \text{and} \quad \int_{(1,\infty)} y^p \, \Pi(dy)<\infty \quad\text{for some }p \in (1,2].
\]

In the following, we present two families of superprocesses for which the SLLN is proved by Theorem~\ref{thm:SLLN}. 
As far as we know, these results are new. Apart from the intrinsic interest, the results are very useful since the analysed processes are frequently employed to obtain further examples of superprocesses with interesting properties via $h$-transform. 
For those examples the SLLN follows from Lemma~\ref{lem:SLLN_h}.

We begin with the inward Ornstein-Uhlenbeck process (OU-process) which has attracted a wide interest in the literature. Specifically, its asymptotic behaviour is the subject of recent research articles \cite{Mil13pre,RenSonZha13pre}.

\begin{example}[Inward OU-process] \label{ex:inOU} Let $d \ge 1$, $D=\R^d$, $L=\frac{1}{2}\Delta -\OUp x \cdot \nabla$ with $\OUp>0$, $\psi$ spatially independent with $\beta \in (0,\infty)$, $\psi(\infty)=\infty$ and $\int_{(1,\infty)} y^p \, \Pi(dy)<\infty$ for some $p \in (1,2]$. 
Then Theorem~\ref{thm:SLLN} applies with $\phi=\mathbf{1}$, $\tphi(x)=(\OUp/\pi)^{d/2} e^{-\OUp \twonorm{x}^2}$ and $\lambda_c=\beta$.
\smallskip

The generator $L$ corresponds to the positive recurrent inward OU-process with transition density
\begin{equation}\label{eq:inOUdensity}
p_{{\sss \text{in-OU}}}(x,y,t)= \Big(\frac{\OUp}{\pi (1-e^{-2\OUp t})}\Big)^{d/2} \exp\Big(-\frac{\OUp}{1-e^{-2\OUp t}} \twonorm{y-e^{-\OUp t}x}^2\Big)  \qquad \text{for all }x,y \in\R^d, t>0.
\end{equation}
Hence, $\lambda_c=\lambda_c(L+\beta)=\beta >0$, $L$ is product $L^1$-critical, $\phi=\mathbf{1}$ and $\tphi(x)=(\OUp/\pi)^{d/2} e^{-\OUp \twonorm{x}^2}$ (cf.\ Chapter 4 in \cite{Pin95} or Example 3 in \cite{Pin96}). Thus, Assumptions~\ref{as:skeleton}--\ref{as:moment1} are satisfied. 
Using the estimate for  $p^{\phi}=p_{{\sss \text{in-OU}}}$ in \eqref{eq:density_bd}, we obtain that Condition~\eqref{eq:SupVeri} holds for $a(t)= \sqrt{(\lambda_c/\OUp+\delta) t}$ with $\delta>0$ (cf.\ Example~10 in \cite{EngHarKyp10}) and using \eqref{eq:inOUdensity}, we deduce that \eqref{eq:ErgoVeri} holds with $K=1$.
\end{example}

\begin{example}[Outward OU-process]\label{ex:outOU} Let $d \ge 1$, $D=\R^d$, $L=\frac{1}{2}\Delta +\OUp x \cdot \nabla$ with $\OUp>0$, $\psi$ spatially independent with $\beta \in (\OUp d,\infty)$, $\psi(\infty)=\infty$ and $\int_{(1,\infty)} y^p \, \Pi(dy)<\infty$ for some $p \in (1,2]$. 
Then Theorem~\ref{thm:SLLN} applies with $\phi(x)=(\OUp/\pi)^{d/2} e^{-\OUp \twonorm{x}^2}$, $\tphi=\mathbf{1}$ and $\lambda_c = \beta - \OUp d$.
\smallskip

The generator $L$ corresponds to the conservative, transient outward OU-process. The operator $L_1:=L+\OUp d$ is the formal adjoint of the inward OU-process with parameter $\OUp$. 
Hence, $L_1$ is critical with ground states $\phi_1(x)=(\OUp/\pi)^{d/2} e^{-\OUp \twonorm{x}^2}$ and $\tphi_1=\mathbf{1}$ by Example~\ref{ex:inOU} (see Theorem 4.3.3 in \cite{Pin95} or Example~2 in \cite{Pin96}). 
Writing $L_1=L+\beta - (\beta -\OUp d)$, we deduce that Assumptions~\ref{as:skeleton}--\ref{as:moment1} hold and $\phi,\tphi$ and $\lambda_c$ have been correctly identified. The corresponding ergodic motion is the inward OU-process with parameter $\OUp$. 
Thus, Conditions~\eqref{eq:SupVeri} and \eqref{eq:ErgoVeri} can be verified using \eqref{eq:inOUdensity}, \eqref{eq:density_bd}, $a(t)=e^{\OUp (1+\delta) t}$ for some $\delta>0$ and $K >1+\delta$.
\end{example}

The SLLN describes the asymptotic behaviour of the mass in compact sets. In general one cannot draw conclusions for the scaling of the total mass from the local behaviour, cf.\ \cite{EngKyp04,EngRenSon13pre}. 
Example~\ref{ex:outOU} illustrates this fact. Since the total mass process is a CSBP with branching mechanism $\psi$, $Y_t=e^{-\beta t} \pint{\mathbf{1},X_t}$ converges to a finite random variable $Y_{\infty}$ with $P_{\mu}(Y_{\infty}=0)=P_{\mu}(\extinguishing)$ if $\beta>0$ and $\int_{(1,\infty)}y \log y \, \Pi(dy)<\infty$, cf.\ \cite{Gre74}. 
In particular, in Example~\ref{ex:outOU}, the local growth rate $\lambda_c=\beta-\OUp d$ is strictly smaller than the global growth rate $\beta$. The reason is the transient nature of the underlying diffusion which allows mass to leave compact sets permanently and is reflected in the decay of $\phi$ at infinity. 
In particular, the function $\mathbf{1}$ is not an allowed test function in Theorem~\ref{thm:SLLN} but the focus is on functions of the form $\mathbbm{1}_B$ for $B$ a compact set.
\medskip

We call the diffusion corresponding to the generator $L=\frac{1}{2} \nabla \cdot a \nabla +b \cdot \nabla$ \emph{symmetric}, if $b=a\nabla Q$ for some $Q \in C^{2,\eta}(D)$.
The inward and outward OU-processes constitute examples of symmetric diffusions with 
$Q(x)=-\frac{\OUp}{2}\twonorm{x}^2$ and $Q(x)=\frac{\OUp}{2}\twonorm{x}^2$, respectively.
Chen et al.\ \cite{CheRenWan08} studied superdiffusions with a symmetric motion but insisted that  $Q$ is bounded. Hence, their results are not applicable to Examples~\ref{ex:inOU} and \ref{ex:outOU}. 
The result from Liu et al.\ \cite{LiuRenSon13} is not applicable since the domain is not of finite Lebesgue measure.

Engl\"ander and Winter \cite{EngWin06} proved convergence in probability in \eqref{eq:SLLN} for the situation of a quadratic branching mechanism. 
It is straightforward to extend their argument to general branching mechanisms but the method requires second moments. 
Hence, if $\int_{(1,\infty)} y^p \, \Pi(dy)<\infty$ for some $p\in (1,2)$ but not for $p=2$, then even the convergence in probability in Examples~\ref{ex:inOU} and \ref{ex:outOU} is new.

\subsection{Quadratic branching mechanisms}\label{sec:quadratic}

In this section, we consider the classical situation of a quadratic branching mechanism studied by Engl\"ander, Pinsky and Winter \cite{EngPin99,EngWin06} and Chen, Ren and Wang \cite{CheRenWan08}. 
Our assumptions on the branching mechanism in this section are $\alpha,\beta \in C^{\eta}(D)$, $\alpha(x)>0$ for all $ x \in D$, $\lambda_c:=\lambda_c(L+\beta)<\infty$ and $\Pi \equiv 0$. 
We write $\psi(x,z)=-\beta(x)z+\alpha(x) z^2$ and call $\psi$ a \emph{generalised quadratic branching mechanism} (GQBM). 
In Section~\ref{sec:model} we insisted that $\alpha$ and $\beta$ are bounded. This assumption can be relaxed as follows.
First suppose that $\beta$ is bounded from above but not necessarily from below. Engl\"ander and Pinsky \cite{EngPin99} showed that there is a unique $\M_f(D)$-valued Markov process $X=(X_t)_{t \ge 0}$ such that 
\[
P_{\mu}[e^{-\la f,X_t\ra}]=e^{-\la u_f(\cdot,t),\mu\ra} \qquad \text{for all continuous } f \in bp(D)\text{ and all } \mu \in \M_f(D),
\]
where $u_f$ is the minimal, nonnegative solution $u\in C(D\times [0,\infty))$, $(x,t) \mapsto u(x,t)$ twice continuously differentiable in $x\in D$ and once in $t \in (0,\infty)$, to
\begin{equation}\label{eq:pdeQuad}
\begin{split}
\partial_t u(x,t)&=Lu(x,t)-\psi(x,u(x,t))\qquad\text{for all } (x,t) \in D \times (0,\infty),\\
u(x,0)&=f(x) \hspace{3.47cm}\text{for all } x \in D.
\end{split}
\end{equation}
If $\alpha$ and $\beta$ are bounded, the minimal solution of \eqref{eq:pdeQuad} equals the unique solution to \eqref{eq:mild} by Lemma~A1 in \cite{EngPin99}. Hence, the two definitions are consistent.

Now let $\beta \in C^{\eta}(D)$ with $\lambda_c=\lambda_c(L+\beta)<\infty$ be not necessarily bounded from above.  
By definition~\eqref{eq:lambdacdef}, there exists $\lambda \in \R$ and $h \in C^{2,\eta}(D)$, $h>0$, such that $(L+\beta)h=\lambda h$. 
Recall the definition of $h$-transforms from Section~\ref{sec:moments}. An $(L,\psi;D)$-superprocess
can be defined by $X=\frac{1}{h}X^h$, where $X^h$ is the $(L_0^h,\psi^h;D)$ superprocess with $\beta^h=\lambda$ and $\alpha^h=\alpha h$, cf.\ \cite{EngPin99}.
Since $h$ is not necessarily bounded from below, the process $X$ may take values in the space of $\sigma$-finite measures $\M(D)$. 
While we have considered mainly finite measure-valued processes in this article, it is natural to consider also processes with values in the space $\M(D)$ via the branching property, and, as noted in Remark~\ref{rem:sigmaFinMu}, in our results the space of starting measures $\M_f^{\phi}(D)$ can be enlarged to the space of all $\mu \in \M(D)$ with $\pint{\phi,\mu}<\infty$.
\smallskip

Engl\"ander and Pinsky~\cite{EngPin99} proved the skeleton decomposition for supercritical superdiffusions with GQBMs long before \cite{KypPerRen13pre}. We only record the existence of a martingale function in the following lemma.
Recall the notation from Section~\ref{sec:skeleton} and that $\extinction$ denotes the event of extinction after a finite time. 

\begin{lemma}[{\bf{Engl\"ander and Pinsky \cite{EngPin99}}}]\label{lem:mgFctQuad}
Let $\psi$ be a GQBM and $\lambda_c>0$. The function $x\mapsto w(x):=-\log P_{\delta_x}(\extinction)$ is strictly positive, belongs to $C^{2,\eta}(D)$ and satisfies \eqref{eq:mgFct}.
\end{lemma}

\begin{proof}
By Theorem~3.1 and Corollary~4.2 in \cite{EngPin99}, $w \in C^{2,\eta}(D)$, $w(x)>0$ for all $x \in D$ and $P_{\mu}(\extinction)=e^{-\pint{w,\mu}}$ for all $\mu \in \M_c(D)$.
Let $B\subset \subset D$ be a domain and $\mu \in \M_f(D)$ with $\text{supp}(\mu) \subseteq B$. 
Then 
the exit measure 
$\widetilde{X}_t^B$ is $P_{\mu}$-almost surely supported on the boundary of $B\times [0,t)$ 
(recall the discussion around \eqref{eq:logLapExit}).
Since $\extinction$ is a tail event, the Markov property yields $P_{\mu}[e^{-\pint{\wtw, \widetilde{X}_t^B}}]= e^{-\pint{w,\mu}}$.
Choose a sequence of functions $w_j \in C_c^+(D)$ with $w_j \uparrow w$ pointwise and a sequence of domains $B_k \subset \subset D$, $B_k \subseteq B_{k+1}$, $D=\bigcup_{k=1}^{\infty}B_k$. 
By Lemma~\ref{lem:mild_exit_mono} and Lemma~A1 in \cite{EngPin99}, $\widetilde{u}_{w_j}^{B_k}$ is pointwise increasing in $j$ and $k$ with $\lim_{k \to \infty}\widetilde{u}_{w_j}^{B_k} = u_{w_j}$ pointwise and we obtain for all $\mu \in \M_c(D)$,
\[
P_{\mu}[e^{-\pint{w,X_t}}]= \lim_{j \to \infty} \lim_{k \to \infty} P_{\mu}[e^{-\pint{\wtw_j,\widetilde{X}_t^{B_k}}}] =\lim_{k \to \infty} \lim_{j \to \infty} P_{\mu}[e^{-\pint{\wtw_j,\widetilde{X}_t^{B_k}}}] = \lim_{k \to \infty} P_{\mu}[e^{-\pint{\wtw,\widetilde{X}_t^{B_k}}}] = e^{-\pint{w,\mu}}.\qedhere
\]
\end{proof}

In the remainder of this section, we choose $w$ to be the function $w(x)=-\log P_{\delta_x}(\extinction)$ and let $Z=(Z_t)_{t \ge 0}$ be a strictly dyadic branching particle diffusion, where the spatial motion is defined by \eqref{eq:Pwdef} and the branching rate is given by $q=\alpha w$ (in accordance with \eqref{eq:Generator}). 

One advantage of allowing unbounded $\alpha$ and $\beta$ is that the setup is now invariant under $h$-transforms: for any $h \in C^{2,\eta}(D)$, $h>0$, $\psi^h$ is a GQBM. Moreover,
\begin{equation}\label{eq:wh}
w^h(x):=-\log P_{\delta_x}^h(\exists t \ge 0\colon \pint{\mathbf{1},X_t}=0) =-\log P_{ h(x)^{-1}\delta_x}(\exists t \ge 0\colon \pint{h,X_t}=0)=w(x)/h(x),
\end{equation}
and Lemmas~\ref{lem:SLLN_h} and \ref{lem:invar} remain valid for GQBMs and $\mathbb{H}(\psi)=\{h \in C^{2,\eta}(D)\colon h>0\}$.
We record the following result.

\begin{theorem}\label{thm:Quad}
Let $\psi$ be a GQBM and suppose Assumption~\ref{as:criticality} holds and $\phi \alpha$ is bounded. Let $\mu \in \M_f^{\phi}(D)$.
\begin{itemize}
\item[{\rm (i)}] For all $f \in p(D)$ with $f/\phi$ bounded, the convergence in \eqref{eq:SLLN} holds in $L^1(P_{\mu})$
\item[{\rm (ii)}] If, in addition, Assumption~\ref{as:SLLN} holds, then there exists a measurable set $\Omega_0$ with $P_{\mu}(\Omega_0)=1$ and, on $\Omega_0$, the convergence in \eqref{eq:SLLN} holds for all $\ell$-almost everywhere continuous $f \in p(D)$ with $f/\phi$ bounded.
\end{itemize}
\end{theorem}

\begin{proof}
Let $X^{\phi}$ be an $(L_0^{\phi},\psi^{\phi};D)$-superprocess. Since $\beta^{\phi}=\lambda_c$ and $\alpha^{\phi}=\phi \alpha$ are bounded, $X^{\phi}$ satisfies the assumptions of Section~\ref{sec:model}. 
Moreover, $X^{\phi}$ satisfies Assumption~\ref{as:skeleton} by Lemma~\ref{lem:mgFctQuad}, Assumption~\ref{as:criticality} with $\phi^{\phi}=\mathbf{1}$ by Lemma~\ref{lem:SLLN_h}(i) and Assumption~\ref{as:moment_h}. 
Hence, Theorem~\ref{thm:SLLN_h}(i) and Lemma~\ref{lem:SLLN_h}(ii) yield the first part of the claim. 
Lemmas~\ref{lem:invar}, \ref{lem:SLLN_h}(ii), \ref{lem:omega0} and Theorem~\ref{thm:SLLN_h}(ii) yield the second part.
\end{proof}

The $h$-transforms are one way to relate two superprocesses to each other, another is monotonicity. 
\begin{lemma}\label{lem:mono}
Let $\psi_{\beta}$ and $\hat{\psi}_{\hat{\beta}}$ be two branching mechanisms as defined in Section~\ref{sec:model} with $\smash{\psi_{\beta}\ge \hat{\psi}_{\hat{\beta}}}$ pointwise. 
Let $X$ and $\smash{\hat{X}}$ be $\smash{(L,\psi_{\beta};D)}$- and $\smash{(L,\hat{\psi}_{\hat{\beta}};D)}$-superprocesses, respectively.
\begin{itemize}
\item[{\rm (i)}]
For all $\mu \in \M_f(D)$, $f\in bp(D)$, $t\ge 0$, $P_{\mu}[e^{-\pint{f,X_t}}] \ge P_{\mu}[e^{-\pint{f,\hat{X}_t}}]$.
\item[{\rm (ii)}] Let $w(x)=-\log P_{\delta_x}(\exists t \ge 0 \colon \pint{\mathbf{1},X_t}=0)$ and $\hat{w}(x)= -\log P_{\delta_x}(\exists t \ge 0 \colon \pint{\mathbf{1},\hat{X}_t}=0)$ for all $x\in D$. Then $w\le \hat{w}$ pointwise.
\end{itemize}
\end{lemma}

\begin{proof}
Part~(i) is proved in Appendix~\ref{sec:beta_unb}. Part~(ii) follows from part~(i) and the identity $w(x)=\lim_{t \to \infty} \lim_{ \theta \to \infty}-\log P_{\delta_x}[e^{-\theta \pint{\mathbf{1},X_t}}]$.
\end{proof}

We saw in Example~\ref{ex:outOU} that the SLLN describes the asymptotics of the mass in compact sets, not necessarily the global growth. 
A second distinction between the local and global behaviour can be observed on the event $\{\mglX=0\} \setminus \extinction$. 
Engl\"ander and Turaev \cite[Problem~14]{EngTur02} raised the question whether this event can have positive probability. 
Suppose Assumption~\ref{as:criticality} holds. Engl\"ander \cite{Eng04} observed that if $\lim_{t \to \infty} e^{-\lambda_c t}\pint{f,X_t}=\pint{f,\tphi} \mglX$ in distribution for all $f \in C_c^+(D)$, $\mu \in \M_c(D)$, and if the support of $X$ is transient, then 
\begin{equation}\label{eq:mglZero}
P_{\mu}(\mglX=0)>P_{\mu}(\extinction) \qquad \text{for all }\mu \in \M_c(D), \mu \not\equiv 0.
\end{equation}
Here the support of $(X;P_{\mu})$ is \emph{recurrent} if 
\[
P_{\mu}(X_t(B)>0 \text{ for some }t\ge 0\,|\,\extinction^c)=1 \qquad \text{for every open }B \subseteq D, B \not= \emptyset,
\]
and \emph{transient} otherwise. See \cite{EngPin99} for a detailed discussion of recurrence and transience of the support of superdiffusions.
\medskip

We study three examples in this section. 
In the first example, $\alpha$ and $\beta$ are bounded but $w$ is unbounded. In the second example $\alpha$ is bounded, but $\beta$, $\phi$ and $w$ are unbounded. 
Both examples are based on a recurrent motion but while the support of the superprocess is recurrent in the second, it is transient in the first example.
The third example considers a large class of processes containing super-Brownian motion with compactly supported growth rate $\beta$ and instances of non-symmetric underlying motions.

The domain for all these examples is $D=\R^d$ and, therefore, none of them is covered in Liu et al.'s\ \cite{LiuRenSon13} article. 
Chen et al.'s\ \cite{CheRenWan08} article is not applicable to the first two examples since they are based on the inward-OU process as underlying motion (i.e., like in Section~\ref{sec:nonspatial}, $Q$ is unbounded) and not to the third because the motion is non-symmetric (for some processes in the considered class) and the variance parameter $\alpha$ is unbounded, whereas \cite{CheRenWan08} require $\alpha$ to be bounded.

The motivation for the first example comes from Example~5.1 in \cite{EngPin99}. 

\begin{example}\label{ex:unb_w} Let $d \ge 1$, $D=\R^d$, $L=\frac{1}{2}\Delta -\OUp x \cdot \nabla$ with $\OUp>0$, $\beta \in (0,\infty)$ constant, $\alpha(x)=e^{-\OUp \twonorm{x}^2}$, $\Pi\equiv 0$. 
Then Theorem~\ref{thm:SLLN} applies with $\phi=\mathbf{1}$, $\tphi(x)=(\OUp/\pi)^{d/2} e^{-\OUp \twonorm{x}^2}$ and $\lambda_c=\beta$. Moreover, $w(x)=(\beta +\OUp d) e^{\OUp \twonorm{x}^2}$, the support of $X$ is transient and \eqref{eq:mglZero} holds.
\smallskip

There are two ways to prove \eqref{eq:SLLN} for this example. First, we perform an $h$-transform with $h(x)=(\OUp/\pi)^{-d/2} e^{\OUp \twonorm{x}^2}$ to obtain
\[
L_0^h=\frac{1}{2}\Delta +\OUp x \cdot \nabla, \quad \beta^h= \beta+\OUp d, \quad \alpha^h=(\pi/\OUp)^{d/2}.
\]
In Example~\ref{ex:outOU}, we showed that Theorem~\ref{thm:SLLN} applies to the $(L_0^h,\psi^h;\R^d)$-superprocess. 
The $(L,\psi;\R^d)$-superprocess can be recovered by an $h$-transform with $h_2=1/h$ and Lemma~\ref{lem:SLLN_h} yields that Assumption~\ref{as:criticality} is satisfied with the stated $\phi$, $\tphi$ and $\lambda_c$, and that \eqref{eq:SLLN} holds. 
Alternatively, we can deduce \eqref{eq:SLLN} by a direct application of Theorems~\ref{thm:SLLN} and \ref{thm:SLLN_skeleton}. 
Assumption~\ref{as:skeleton} holds by Lemma~\ref{lem:mgFctQuad}, and Assumption~\ref{as:moment1} holds since $\alpha$ and $\phi$ are bounded. 
To verify Assumption~\ref{as:SLLN}, 
we notice that $w^h(x)=\beta^h/\alpha^h=(\beta +\OUp d) (\OUp/\pi)^{d/2}$ by \eqref{eq:wExting}, and \eqref{eq:wh} yields 
\[
w(x)=w^h(x)h(x)=(\beta+\OUp d) e^{\OUp \twonorm{x}^2} \qquad \text{for all } x \in \R^d.
\]
Thus, $w$ is not bounded from above. The verification of \eqref{eq:SupVeri} and \eqref{eq:ErgoVeri} is the same as in Example~\ref{ex:outOU} since $w/\phi$ is of the same order and the ergodic motion is the same. 
Hence, the conditions hold with $a(t)=e^{\OUp (1+\delta) t}$ for some $\delta>0$ and $K >1+\delta$.

To see that the support of $X$ is transient notice that the support is invariant under $h$-transforms and the support of the $(L_0^h;\psi^h;\R^d)$-superprocess is transient by Theorem~4.6 in \cite{Pin95} and Example~2 in \cite{Pin96}.
\end{example}

Example~\ref{ex:unb_w} should be compared to Example~\ref{ex:outOU} for a quadratic branching mechanism. 
In both examples, the support of the superprocess is transient and the event $\{\mglX=0\}\setminus \extinction$ has positive probability. 
Hence, in both examples, mass can escape to infinity which is reflected in the SLLN by virtue of the fact that $\mglX=0$. 
However, the motion in Example~\ref{ex:unb_w} is recurrent and the SLLN captures not only the local but also the global growth of mass.
\smallskip

The unbounded $w$ in Example~\ref{ex:unb_w} can be interpreted as follows. Heuristically, since the local growth rate $\beta$ is bounded away from zero, on average a large population is generated everywhere in space. 
Risk for the branching process comes from areas of a relatively large variance for the total mass process. In contrast, when the variance parameter $\alpha$ is very small, then extinction is unlikely and $w$ becomes large. 
\smallskip

The motivation for the next example comes from Example~10 in \cite{EngHarKyp10}. For $B \in \B(D)$, $f_1,f_2 \in p(B)$, we write $f_1 \asymp f_2$ if there are constants $0<c\le C <\infty$ such that $c f_1(x) \le f_2(x) \le C f_1(x)$ for all $x \in B$.

\begin{example}\label{ex:betaUnbounded} Let $d \ge 1$, $D= \R^d$, $L=\frac{1}{2} \Delta -\OUp x \cdot \nabla$, $\beta(x)=c_1 \twonorm{x}^2+c_2$, where $\OUp, c_1,c_2 >0$, $\OUp >\sqrt{2c_1}$ and we let $\exPa:=\frac{1}{2} (\OUp -\sqrt{\OUp^2-2c_1})$. Then Assumption~\ref{as:criticality} holds with $\lambda_c=c_2+d \exPa$, $\phi(x)=e^{\exPa \twonorm{x}^2}$ and $\tphi(x)=c e^{(\exPa-\OUp) \twonorm{x}^2}$, where $c= (\frac{\OUp-2 \exPa}{\pi})^{d/2}$. Suppose that $\Pi\equiv 0$ and $\alpha \in C^{\eta}(D)$ with $\alpha \asymp 1/\phi$ on $\R^d$. Then Theorem~\ref{thm:Quad} applies, $w\asymp \phi$ and the support of $X$ is recurrent.
\smallskip

Let $h(x)=e^{\exPa \twonorm{x}^2}$. Using $-\OUp +2\exPa=-\sqrt{\OUp^2-2c_2}$ and $\exPa^2-\OUp \exPa +\frac{c_1}{2}=0$, we observe that
\[
L_0^h=\frac{1}{2} \Delta -\sqrt{\OUp^2-2c_1} x \cdot \nabla, \quad \beta^h=\exPa d +c_2, \quad \alpha^h(x)=h(x)\alpha(x) \asymp 1 \quad \text{on }\R^d.
\]
The $(L_0^h,\psi^h;\R^d)$-superprocess, denoted by $X^h$, satisfies Assumption~\ref{as:criticality} by 
Example~\ref{ex:inOU} with $\phi^h=\mathbf{1}$, $\tphi^h \asymp e^{(2\exPa-\OUp)\twonorm{x}^2}$ and $\lambda_c^h=\exPa d+c_2$. 
Hence, Lemma~\ref{lem:SLLN_h}(i) shows that $X$ satisfies Assumption~\ref{as:criticality}, $\phi$, $\tphi$ and $\lambda_c$ have been correctly identified and $\phi \alpha$ is bounded. 
When we have verified Assumption~\ref{as:SLLN} for $X^h$, then Lemma~\ref{lem:invar} will yield Assumption~\ref{as:SLLN} for $X$ and the claim is established

To this end, choose constants $c_3,c_4>0$ such that $c_3/h \le \alpha \le c_4/h$ pointwise. Let $\overline{\psi}(x,z):=-\beta^hz+c_3 z^2$ and $\underline{\psi}(x,z):=-\beta^hz +c_4 z^2$, and denote by $\overline{w}$ and $\underline{w}$ the martingale functions corresponding to the event of extinction after finite time for the $(L_0^h,\overline{\psi};\R^d)$ and $(L_0^h,\underline{\psi};\R^d)$-superprocesses, respectively. Since $\overline{\psi} \le \psi^h \le \underline{\psi}$ pointwise, Lemma~\ref{lem:mono}(ii) and \eqref{eq:wExting} imply
\[
\beta^h/c_3=\overline{w}(x) \ge w^h(x) \ge \underline{w}(x)=\beta^h/c_4 \qquad \text{for all }x\in D.
\]
Hence, $w^h \asymp \mathbf{1} = \phi^h$. Now the verification of \eqref{eq:SupVeri} and \eqref{eq:ErgoVeri} for $X^h$ is the same as in Example~\ref{ex:inOU} and we can choose $a(t)= \sqrt{(\lambda_c/\gamma+\delta) t}$ for some $\delta>0$ and $K=1$.

The support of $X$ is recurrent since the support is invariant under $h$-transforms and the support of the $(L_0^h, \psi^h;\R^d)$-superprocess is recurrent according to Theorem~4.4(b) in \cite{EngPin99}.
\end{example}

The next example covers a large class of processes. The underlying motion is a Brownian motion with or without a compactly supported drift term. 
Depending on the choice of that drift, the underlying motion can be symmetric or non-symmetric. 
For a choice of $b$ which makes $L$ non-symmetric see Example~13 in \cite{EngHarKyp10}. The article by Chen et al.\ \cite{CheRenWan08} excludes non-symmetric motions. 
The example is motivated by Example~22 in \cite{EngTur02} and Examples~12 and 13 in \cite{EngHarKyp10}.

\begin{example}\label{ex:nonsymm}
Let $d \in \{1,2\}$, $D=\R^d$, $L=\frac{1}{2}\Delta +b \cdot \nabla$ where all components of $b$ belong to $C^{1,\eta}(\R^d)$ for some $\eta \in (0,1]$ and are of compact support, $\beta_0 \in C^{\eta}(\R^d)$ nonnegative and of compact support, $\beta_0\not=\mathbf{0}$. There exists $\theta>0$ such that $\lambda_c(L+\theta \beta_0)>0$ and we let $\beta=\theta \beta_0$, $\lambda_c=\lambda_c(L+\beta)$. We further write 
\[
\varrho(x)=\twonorm{x}^{(1-d)/2} e^{-\sqrt{2 \lambda_c} \twonorm{x}}, \qquad \text{for all }x\in \R^d\setminus \{0\}.
\]
Let $\alpha \in C^{\eta}(D)$, $\alpha(x)>0$ for all $x \in \R^d$ and $\alpha \asymp 1/\varrho$ on $\R^d \setminus B$ for an open ball $B$ around the origin. Then Theorem~\ref{thm:Quad} applies with $\phi, \tphi,w \asymp \varrho$ on $\R^d\setminus B$ and the support of $X$ is recurrent.
\smallskip

The existence of $\theta$ is proven in Theorems~4.6.3 and 4.6.4 of Pinsky's book \cite{Pin95} and $L+\beta-\lambda_c$ is critical by Theorem~4.6.7 in the same bock. 
Denote by $G$ the Green's function corresponding to the operator $L-\lambda_c$. 
Then $\phi\asymp G(\cdot,0)$ on $\R^d\setminus B$ by Theorems~4.6.3 and 7.3.8 in \cite{Pin95}. 
Pinsky showed in Example~7.3.11 that the Green's function $G_1$ of $\frac{1}{2}\Delta - \lambda_c$ satisfies $G_1(\cdot,0) \asymp \varrho$ on $\R^d\setminus B$. 
Since $b$ is compactly supported, $G_1(\cdot,0)\asymp G(\cdot,0)$ on $\R^d \setminus B$ and the estimate for $\phi$ is established. 
The same argument yields the same estimate for $\tphi$ and Assumption \ref{as:criticality} holds. Moreover, $\phi\alpha$ is bounded.

To check Assumption~\ref{as:SLLN} we use Theorem~\ref{thm:SLLN_skeleton}. An $h$-transform of the $(L,\psi;\R^d)$-superprocess with $h=\phi$ gives an $(L_0^{\phi},\psi^{\phi};\R^d)$-superprocess, where $L_0^{\phi}$ corresponds to a conservative, positive recurrent motion and $\psi^{\phi}(x,z)=-\lambda_c z+ \phi(x)\alpha(x) z^2$. 
Since $\phi \alpha \asymp \mathbf{1}$, $w^{\phi} \asymp \mathbf{1}$ by the same argument as in Example~\ref{ex:betaUnbounded}. 
Hence, \eqref{eq:wh} implies $w/\phi \asymp \mathbf{1}$ and conditions~(i) and (ii) of Theorem~\ref{thm:SLLN_skeleton} have been verified in Examples~12 and 13 of \cite{EngHarKyp10} with $a(t)=\sqrt{2 \supnorm{\beta}} t$, $K > 1/\sqrt{2 \lambda_c}$. 
Since $w$ is bounded and the underlying diffusion is recurrent, Theorem~4.4 in \cite{EngPin99} shows that the support of $X$ is recurrent.
\end{example}

\subsection{Bounded domains}\label{sec:bd_domain}

In the situation that $D$ is a bounded Lipschitz domain and $L$ is a uniformly elliptic operator with smooth coefficients, Liu et al.~\cite{LiuRenSon13} prove the SLLN for a general branching mechanism with $\beta$ any bounded, measurable function on $D$ and $\alpha$ and $\Pi$ as in Section~\ref{sec:model}.  

However, the Wright-Fisher diffusion on domain $D=(0,1)$ is a diffusion process whose diffusion matrix $a(x)=x(1-x)$ is not uniformly elliptic. 
The process has attracted a wide interest in the literature (see for example \cite{GreKleWak01,FleSwa03,Ber09}). 
Fleischmann and Swart \cite{FleSwa03} studied the large-time behaviour of the corresponding superprocess with spatially independent, quadratic branching mechanism on $[0,1]$. 
They conjecture a SLLN for the process restricted to $D=(0,1)$ (see above (23) in \cite{FleSwa03}) but prove only convergence in $L^2$. 
The Wright-Fisher diffusion is not conservative, so the arguments in Section~\ref{sec:nonspatial} are not applicable. 
However, Theorem~\ref{thm:SLLN} is and the following theorem proves the conjecture for all $\ell$-almost everywhere continuous test functions $f \in p(D)$ with $f/\phi$ bounded (Fleischmann and Swart do not assume any continuity).

\begin{theorem}[Super-Wright-Fisher diffusion]
Let $D=(0,1)$, $\beta\in (1,\infty)$, $\alpha>0$, $\Pi\equiv 0$ and
\[
L=\frac{1}{2} x(1-x) \frac{d^2}{dx^2}=\frac{1}{2} \frac{d}{dx} x(1-x) \frac{d}{dx}+ \frac{2x-1}{2} \frac{d}{dx}.
\]
Then Theorem~\ref{thm:SLLN} applies with $\phi(x)=6 x (1-x)$, $\tphi=\mathbf{1}$ and $\lambda_c=\beta-1$.
\end{theorem}

\begin{proof} 
Let $h(x) = 6 x (1-x)$. Fleischmann and Swart proved in Lemma~20 of \cite{FleSwa03} that the generator
\[
L_0^h= \frac{1}{2} \frac{d}{dx} x(1-x) \frac{d}{dx} + \frac{1-2x}{2} \frac{d}{dx} 
\]
corresponds to an ergodic diffusion with invariant law $h(x) \ell(dx)$ on $D$. Using $\beta^h=\beta-1$, we deduce that $\lambda_c(L_0^h+\beta^h)=\beta-1$, $\phi^h =\mathbf{1}$, $\tphi^h=h$ and, using Lemma~\ref{lem:SLLN_h}, Assumption~\ref{as:criticality} for the $(L,\psi;D)$ superprocess as well as the stated identities for $\phi$, $\tphi$ and $\lambda_c$ are established. Assumption~\ref{as:skeleton} holds by Lemma~\ref{lem:mgFctQuad}; 
the boundedness of $\alpha$ and $\phi$ implies that Assumption~\ref{as:moment1} is satisfied. To verify Assumption~\ref{as:SLLN}, we notice that Condition~(i) of Theorem~\ref{thm:SLLN_skeleton} is trivially satisfied for $D_t=D$ and \eqref{eq:ErgodicSpeed} for $D_t=D$ and $K=1$ has been proved in Lemma~20 of \cite{FleSwa03}. 
Hence, Assumption~\ref{as:SLLN} follows from Theorem~\ref{thm:SLLN_skeleton} and Theorem~\ref{thm:SLLN} applies.
\end{proof}

\appendix
\section{Appendix}
\label{appendix}

\subsection{Feynman-Kac arguments}\label{sec:FeynKac}
In this section, we prove an integral identity that is used several times in this article. Versions of this result appeared in Lemma~A.I.1.5 of \cite{Dyn93} and Lemma 4.1.1 of \cite{Dyn02} but the format and assumptions are different. Like in the remainder of the article, $(\xi=(\xi_t)_{t\ge 0}\colon (\P_x)_{x\in D})$ is a diffusion as described in Section~\ref{sec:model}.

\begin{lemma}\label{lem:semi}
Let $T>0$ and either $B=D$ or $B\subset \subset D$ open. Write $\tau=\inf\{t \ge 0\colon \xi_t \not\in B\}$ and $A=D$ if $B=D$; $A=\overline{B}$ if $B \subset \subset D$.
\smallskip

\noindent
\begin{itemize}
\item[{\rm (i)}] 
Let $f_1\in b(A)$, $g_1\colon A \times [0,T] \to \R$ measurable and bounded from above and $f_2,g_2\in b(A\times [0,T])$. If for all $(x,t) \in A \times [0,T]$,
\begin{equation}\label{eq:semi}
v(x,t)=\P_x\Big[e^{\int_0^{t\land \tau} (g_1+g_2)(\xi_r,t-r)\, dr} f_1(\xi_{t\land \tau})\Big]+\P_x\Big[\int_0^{t\land \tau}  e^{\int_0^s (g_1+g_2)(\xi_r,t-r)\, dr} f_2(\xi_s, t-s)\, ds\Big],
\end{equation}
then, for all $(x,t) \in A \times [0,T]$,
\[
\hspace{-0.77cm}v(x,t)= \P_x\Big[e^{\int_0^{t\land \tau} g_1(\xi_r,t-r)\, dr} f_1(\xi_{t\land \tau})\Big]+\P_x\Big[\int_0^{t\land \tau} e^{\int_0^s g_1(\xi_r,t-r)\, dr} \Big(f_2(\xi_s, t-s)+g_2(\xi_s,t-s)v(\xi_s,t-s)\Big)\, ds\Big].
\]
\item[{\rm (ii)}]
The statement of (i) remains valid when $f_1 \in bp(A)$, $f_2,g_1,g_2 \colon A\times [0,T] \to \R$ measurable with $g_1$ bounded from above, $g_2$ nonnegative, $f_2$ nonpositive and $g_1+g_2$ bounded from above. Notice that in this case, $v$ might attain the value $-\infty$.
\end{itemize}
\end{lemma}

\begin{proof}
For all $t \ge 0$, write
\begin{align*}
\mathcal{Y}_t=e^{\int_0^{t \land \tau} (g_1+g_2)(\xi_r,t-r) \, dr} f_1(\xi_{t\land \tau}), \quad\text{and}\quad \mathcal{Z}_t=\int_0^{t \land \tau} e^{\int_0^s (g_1+g_2)(\xi_r,t-r)\, dr} f_2(\xi_s ,t-s)\, ds.
\end{align*}
By assumption, $v(x,t)=\P_x[\mathcal{Y}_t+\mathcal{Z}_t]$ for all $(x,t) \in A \times [0,T]$. The Markov property implies
\begin{align*}
\int_0^t \P_x\Big[\mathbbm{1}_{\{s <\tau\}}& e^{\int_0^s g_1(\xi_r,t-r) \, dr} g_2(\xi_s,t-s) \P_{\xi_s}[\mathcal{Y}_{t-s}]\Big] \, ds\\
&=\int_0^t \P_x\Big[\mathbbm{1}_{\{s <\tau\}} e^{\int_0^s g_1(\xi_r,t-r) \, dr} g_2(\xi_s,t-s) e^{\int_s^{t\land \tau} (g_1+g_2)(\xi_r,t-r) \, dr} f_1(\xi_{t\land \tau}) \Big]\, ds.
\end{align*}
If $g_2$ is bounded, then Fubini's theorem and the fundamental theorem of calculus (FTC) for Lebesgue integrals imply that the right-hand side equals
\begin{align*}
\P_x\Big[f_1(\xi_{t\land \tau}) e^{\int_0^{t\land \tau} g_1(\xi_r,t-r) \, dr} &\int_0^{t\land \tau}  g_2(\xi_s,t-s) e^{\int_s^{t\land \tau} g_2(\xi_r,t-r) \, dr} \, ds\Big]\\
&=\P_x\Big[f_1(\xi_{t\land \tau}) e^{\int_0^{t\land \tau} g_1(\xi_r,t-r) \, dr} \Big(e^{\int_0^{t\land \tau} g_2(\xi_r,t-r) \, dr} -1\Big)\Big].
\end{align*}
In the situation of (ii), the same identity can be obtained by truncating $g_2$ before the application of FTC and using the monotone convergence theorem afterwards.
The Markov property and Fubini's theorem (in case (ii) its application is justified by the nonpositivity of the integrand) yield
\begin{align*}
&\int_0^t \P_x\Big[\mathbbm{1}_{\{s <\tau\}} e^{\int_0^s g_1(\xi_r,t-r) \, dr} g_2(\xi_s,t-s) \P_{\xi_s}[\mathcal{Z}_{t-s}]\Big] \, ds\\
&=\int_0^t \P_x\Big[\mathbbm{1}_{\{s <\tau\}} e^{\int_0^s g_1(\xi_r,t-r) \, dr} g_2(\xi_s,t-s) \int_s^{t\land \tau} e^{\int_s^u (g_1+g_2)(\xi_r,t-r)\, dr} f_2(\xi_u ,t-u)\, du \Big] \, ds\\
&= \P_x\Big[\int_0^{t\land \tau} e^{\int_0^u g_1(\xi_r,t-r) \, dr} f_2(\xi_u ,t-u) \int_0^u g_2(\xi_s,t-s) e^{\int_s^u g_2(\xi_r,t-r)\, dr}  \, ds\, du\Big].
\end{align*}
As above, the FTC implies that the right-hand side equals
\[
\P_x\Big[\int_0^{t\land \tau} e^{\int_0^u g_1(\xi_r,t-r) \, dr} f_2(\xi_u ,t-u) \Big(e^{\int_0^u g_2(\xi_r,t-r) \, dr}-1\Big)\, du\Big].
\]
Since $v(x,t)=\P_x[\mathcal{Y}_t+\mathcal{Z}_t]$ for all $(x,t) \in A \times [0,T]$, we conclude that in the situation of (i),
\begin{align*}
&  \P_x\Big[\int_0^{t\land \tau} e^{\int_0^s g_1(\xi_r,t-r) \, dr} \Big(f_2(\xi_s,t-s)+g_2(\xi_s,t-s) v(\xi_s,t-s)\Big) \, ds \Big]\\
&=\P_x\Big[f_1(\xi_{t\land \tau}) e^{\int_0^{t\land \tau} g_1(\xi_r,t-r) \, dr} \Big(e^{\int_0^{t\land \tau} g_2(\xi_r,t-r) \, dr} -1\Big)+\int_0^{t\land \tau} e^{\int_0^s (g_1+g_2)(\xi_r,t-r) \, dr} f_2(\xi_s ,t-s) \,ds\Big].
\end{align*}
In the situation of (ii), this use of linearity is justified since none of the summed integrals can take the value $+\infty$. Since $f_1$ is bounded and $g_1,g_1+g_2$ are bounded from above, the first summand on the right can be written as the difference of two finite integrals and \eqref{eq:semi} yields the claim.
\end{proof}

\subsection{Integral equations}\label{sec:beta_unb}

In this section, we study a generalised version of the mild equation \eqref{eq:mild} and of the corresponding equation for exit measures \eqref{eq:mildExit}. 
We only assume that $\beta$ is bounded above, not necessarily from below. More specifically, the setup is as follows.

Let $\beta \colon D \to \R$ be measurable with $\bbeta=\sup_{x \in D} \beta(x)<\infty$, $\alpha \in bp(D)$, $\Pi$ a kernel from $D$ to $(0,\infty)$ such that $x \mapsto \int_{(0,\infty)} (y \land y^2) \,\Pi(x,dy)$ belongs to $bp(D)$, and let $(\xi;\P)$ be a diffusion as described in Section~\ref{sec:model}. 
We denote by $\locbdpos(D\times [0,\infty))$ the space of all functions $f \in p(D\times [0,\infty))$ with $\supTnorm{f}:=\sup_{t \in [0,T]} \supnorm{f(\cdot,t)}<\infty$ for all $T>0$.

For $f \in bp(D)$ and $g\in \locbdpos(D\times [0,\infty))$, we are interested in solutions to the integral equation
\begin{equation}\label{eq:mild2}
u(x,t)+\int_0^t S_s\big[\psi_0(\cdot,u(\cdot,t-s))\big](x) \, ds=S_tf(x)+\int_0^t S_s[g(\cdot,t-s)](x) \, ds, \quad (x,t) \in D\times [0,\infty),
\end{equation}
where $S_t\hat{g}(x)=\P_x[e^{\int_0^t \beta(\xi_s) \,ds} \hat{g}(\xi_t)]$ for $\hat{g} \in p(D)$.

In this section, we prove existence and uniqueness of the solution to \eqref{eq:mild2}. 
Choosing $g=\mathbf{0}$, this proves Lemma~\ref{lem:mildStar} but is also of intrinsic interest when one wants to extend the definition of the superprocess to branching mechanisms with $\beta$ only bounded from above. 
The more general right-hand side in \eqref{eq:mild2} allows us to obtain Lemma~\ref{lem:mono}. 
Taking advantage of this analysis, we prove monotonicity of the solution $\widetilde{u}_f^B$ of \eqref{eq:mildExit} in the domain $B$ when $f$ is compactly supported.
\medskip

Note that $z \mapsto \psi_0(x,z)$, defined in \eqref{eq:psiDef} is increasing, convex, and nonnegative. In particular, any nonnegative solution $u$ to \eqref{eq:mild2}, satisfies
\[
0 \le u(x,t) \le e^{\bbeta t}\supnorm{f}+ \int_0^t e^{\bbeta s} \supnorm{g(\cdot,t-s)} \, ds \quad \text{for all } (x,t) \in D\times [0,\infty).
\]
Hence, any nonnegative solution to \eqref{eq:mild2} is an element of $\locbdpos(D\times[0,\infty))$. 
Moreover, $\psi_0$ is locally Lipschitz continuous in the sense that for every fixed $c>0$ there exists $\mathcal{L}(c)\in [0,\infty)$ such that
\begin{equation}\label{eq:psi0_Lip}
|\psi_0(x,z_1)-\psi_0(x,z_2)| \le \mathcal{L}(c) |z_1-z_2| \qquad \text{for all } z_1,z_2 \in [0,c], x \in D.
\end{equation}
We use the following version of Gronwall's Lemma. For a proof see Theorem~A.5.1 in \cite{EthKur05}.

\begin{lemma}[{\bf Gronwall's Lemma}]\label{lem:Gron}
Let $T>0$, $C,\rho \ge 0$ and $h\in b([0,T])$. If
\[
h(t) \le C+\rho \int_0^t h(s) \, ds \qquad \text{for all } t \in [0,T],
\]
then $h(t) \le C e^{\rho t}$ for all $t \in [0,T]$.
\end{lemma}

\begin{lemma}[{\bf Uniqueness}]\label{lem:mild_uni} Let $f, \hat{f} \in bp(D)$,  $g,\hat{g} \in \locbdpos(D\times [0,\infty))$, and suppose that $u$ and $\hat{u}$ are nonnegative solutions to \eqref{eq:mild2} for $(f,g)$ and $(\hat{f},\hat{g})$, respectively. 
Then there exist for every $T>0$, constants $C,\rho>0$ such that
\[
\supTnorm{u-\hat{u}} \le C \big(\supnorm{f-\hat{f}}+\supTnorm{g-\hat{g}}\big) e^{\rho T}.
\]
In particular, the solution to \eqref{eq:mild2} is unique.
\end{lemma}

\begin{proof}
Fix $T>0$, and let $c\ge \max\{\supTnorm{u},\supTnorm{\hat{u}}\}$. Then \eqref{eq:psi0_Lip} yields
\[
|\psi_0(x,\hat{u}(x,t)) - \psi_0(x,u(x,t))| \le \mathcal{L}(c) |\hat{u}(x,t)-u(x,t)| \quad \text{for all } (x,t) \in D\times [0,T].
\]
Writing $h(x,t)=|u(x,t)-\hat{u}(x,t)|$ and $M=\max\{e^{\bbeta T},1\}$, we find for all $(x,t) \in D \times [0,T]$,
\begin{align*}
h(x,t)\le M\supnorm{f-\hat{f}}+MT \supTnorm{g-\hat{g}} + \int_0^t M\mathcal{L}(c) \supnorm{h(\cdot,s)} \, ds.
\end{align*}
Lemma~\ref{lem:Gron} yields the claim.
\end{proof}

The following lemma in the case $\beta=\mathbf{0}$ and $g=\mathbf{0}$ is Theorem~4.3.1 in \cite{Dyn02}.
\begin{lemma}[{\bf{Existence}}]\label{lem:mild_exist} Let $f \in bp(D)$ and $g \in \locbdpos(D\times [0,\infty))$. There exists a nonnegative solution $u\in \locbdpos(D\times [0,\infty))$ to \eqref{eq:mild2}.
\end{lemma}

\begin{lemma}[{\bf{Monotonicity in $f$}}]\label{lem:mild_mono}  Let $f,\hat{f} \in bp(D)$, $g,\hat{g} \in \locbdpos(D\times [0,\infty))$ with $f \le \hat{f}$ and $g\le \hat{g}$ pointwise and denote by $u$ and $\hat{u}$ the unique solutions to \eqref{eq:mild2} corresponding to $(f,g)$ and $(\hat{f},\hat{g})$, respectively. 
Then $u \le \hat{u}$ pointwise.
\end{lemma}

\begin{proof}[Proof of Lemma~\ref{lem:mild_exist}]
Fix $T>0$ and let $M=\max\{e^{\bbeta T},1\}$. For $k \in [0,\infty)$ and $u \in \locbdpos(D \times [0,T])$, i.e.\ $u \in p(D \times [0,T])$ with $\supTnorm{u}<\infty$, we define for all $(x,t) \in D\times [0,T]$,
\[
F_ku(x,t)=e^{-kt} S_tf(x)+ \int_0^t e^{-ks} S_s[g(\cdot,t-s)](x) \,ds + \int_0^t e^{-ks} S_s\big[ku(\cdot,t-s)-\psi_0(\cdot,u(\cdot,t-s))\big](x) \, ds.
\]
Let $c \ge M \supnorm{f} + MT \supTnorm{g}$ and $k \ge \mathcal{L}(c)$. Write $v(x,t)=e^{\bbeta t}\supnorm{f}+\int_0^t e^{\bbeta s}\, ds\supTnorm{g}$ for all $x\in D$, $t\in [0,T]$. We show the following:
\begin{itemize}
\item[(i)] $\mathbf{0} \le F_k\mathbf{0} \le v$ pointwise on $D \times [0,T]$.
\item[(ii)] If $\mathbf{0} \le u_1\le u_2 \le v$ pointwise on $D \times [0,T]$, then $F_ku_1 \le F_ku_2$ pointwise on $D \times [0,T]$.
\item[(iii)] $F_kv \le v$ pointwise on $D \times [0,T]$.
\end{itemize}
Indeed, $F_k\mathbf{0}(x,t)=e^{-kt} S_tf(x) +\int_0^t e^{-ks} S_s[g(\cdot,t-s)](x) \,ds  \in [0, v(x,t)]$ since $f$ and $g$ are nonnegative and $k \ge 0$. For (ii), we use that $v(x,t) \le c$ for all $(x,t) \in D \times [0,T]$ and \eqref{eq:psi0_Lip} to obtain 
\begin{align*}
F_ku_2(x,t)-F_ku_1(x,t)&=\int_0^t e^{-ks} S_s\big[k(u_2-u_1)(\cdot,t-s)-\big(\psi_0(\cdot,u_2(\cdot,t-s))-\psi_0(\cdot,u_1(\cdot,t-s))\big)\big](x) \, ds \\
&\ge \int_0^t e^{-ks} S_s\big[(k-\mathcal{L}(c))(u_2-u_1)(\cdot,t-s)\big](x) \, ds \ge 0.
\end{align*}
To show (iii), we use that $\psi_0$ is nonnegative, the definition of $v$ and Fubini's theorem to obtain
\begin{align*}
F_kv(x,t)&\le e^{-kt} e^{\bbeta t} \supnorm{f}+ \int_0^t e^{-ks} e^{\bbeta s} \supTnorm{g} \, ds+ \int_0^t e^{-ks} e^{\bbeta s}k \Big(e^{\bbeta(t-s)}\supnorm{f}+\int_0^{t-s} e^{\bbeta r} \,dr \supTnorm{g}\Big) \, ds\\
&= \Big(e^{-kt}+\int_0^t k e^{-ks} \,ds\Big)e^{\bbeta t} \supnorm{f} + \Big( \int_0^t e^{(\bbeta-k)s}\,ds+  \int_0^t k e^{-ks} \int_s^t e^{\bbeta r} \,dr \, ds\Big)\supTnorm{g}=  v(x,t).
\end{align*}
In the next step, we construct a solution to \eqref{eq:mild2} via a Picard iteration. Let $u_0=\mathbf{0}$ and $u_n=F_ku_{n-1}$ for all $n \in \N$. 
We show by induction that $\mathbf{0} \le u_{n-1} \le u_n \le v$ pointwise on $D\times [0,T]$ for all $n \in \N$. For $n=1$, this is statement~(i). 
The induction step follows from~(ii)--(iii). In particular, $(u_n)_{n \in \N_0}$ has a pointwise limit $u$ which is a fixed point of $F_k$ by the dominated convergence theorem. 
Lemma~\ref{lem:semi}(i) applied to $g_1=\beta$, which is bounded from above, and the bounded functions $g_2=-k$, $f_1=f$ and $f_2(x,t)=g(x,t)+ku(x,t)-\psi_0(x,u(x,t))$, shows that $u$ solves \eqref{eq:mild2}.
\end{proof}

\begin{proof}[Proof of Lemma~\ref{lem:mild_mono}] Since the solution is unique according to Lemma~\ref{lem:mild_uni}, the claim follows immediately from the construction of the solution via Picard iteration in the proof of Lemma~\ref{lem:mild_exist}.
\end{proof}

\begin{proof}[Proof of Lemma~\ref{lem:mono}(i)]
According to \eqref{eq:mild}, $u_f$ is the unique solution to \eqref{eq:mild2} with $g=\mathbf{0}$. Moreover, \eqref{eq:mild} for $\hat{u}_f$ and Lemma~\ref{lem:semi}(i) applied to $g_1=\beta$, $g_2=\hat{\beta}-\beta$, $f_1=f$ and $f_2(x,t)=\hat{\psi}_0(x,\hat{u}_f(x,t))$ implies that $\hat{u}_f$ satisfies
\begin{align*}
\hat{u}_f(x,t)&=S_tf(x)+\int_0^t S_s\big[-\hat{\psi}_0(\cdot,\hat{u}_f(\cdot,t-s))+(\hat{\beta}-\beta)\hat{u}_f(\cdot,t-s)\big](x) \, ds\\
&=S_tf(x)-\int_0^t S_s\big[\psi_0(\cdot,\hat{u}_f(\cdot,t-s))\big](x) \,ds + \int_0^t S_s\big[\psi_{\beta}(\cdot,\hat{u}_f(\cdot,t-s))-\hat{\psi}_{\hat{\beta}}(\cdot,\hat{u}_f(\cdot,t-s))\big](x) \, ds.
\end{align*}
In particular, $\hat{u}_f$ solves \eqref{eq:mild2} with $g(x,t)=\psi_{\beta}(x,\hat{u}_f(x,t))-\hat{\psi}_{\hat{\beta}}(x,\hat{u}_f(x,t)) \ge 0$. Now Lemma~\ref{lem:mild_mono} yields the claim.
\end{proof}

\begin{lemma}[{\bf{Monotonicity in $B$}}]\label{lem:mild_exit_mono}  Let $B \subset \subset D$ and $f \in bp(D)$ such that the support of $f$, $\text{supp}(f)$, is a subset of $B$. There exists a unique nonnegative solution $u_f^B \in \locbdpos(D\times [0,\infty))$ to 
\begin{equation}\label{eq:mild_stop}
u(x,t)=\P_x\Big[e^{\int_0^{t \land \tau_B} \beta(\xi_s) \, ds} f(\xi_{t \land \tau_B})\Big] -\P_x\Big[\int_0^{t \land \tau_B} e^{\int_0^s \beta(\xi_r) \, dr} \psi_0(\xi_s,u(\xi_s,t-s))\, ds \Big].
\end{equation}
Moreover, if $B_1$ and $B_2$ are domains with $\text{supp}(f) \subseteq B_1 \subseteq B_2$, then $u_f^{B_1} \le u_f^{B_2}$ pointwise.
\end{lemma}

\begin{proof}
Let $T>0$, $M=\max\{e^{\bbeta T},1\}$, $c\ge M \supnorm{f}$, $k\ge \mathcal{L}(c)$ and for $u\in \locbdpos(D \times [0,T])$, define
\begin{align*}
F_ku(x,t)&=\P_x\Big[e^{\int_0^t [\beta(\xi_s)-k] \, ds} f(\xi_t) \, \mathbbm{1}_{\{t <\tau_{B}\}}\Big]\\
&\phantom{spa}+ \int_0^t \P_x\Big[e^{\int_0^s [\beta(\xi_r)-k] \, dr}\big[ku(\xi_s,t-s)- \psi_0(\xi_s,u(\xi_s,t-s))\big] \mathbbm{1}_{\{s \le \tau_B\}}\Big] \, ds.
\end{align*}
Since $kz-\psi_0(x,z) \ge kz-\mathcal{L}(c) z \ge 0$ for all $z \in [0,c]$, $F_ku$ is pointwise increasing in $B$ for all $u$ with $u(x,t) \le e^{\bbeta t}\supnorm{f} =:v(x,t)$. 
As in Lemmas~\ref{lem:mild_uni} and~\ref{lem:mild_exist}, the unique solution to \eqref{eq:mild_stop} can be obtained as a pointwise limit of the increasing sequence $u_0=\mathbf{0}$, $u_{n+1}=F_ku_n$ with $u_n \le v$ pointwise for all $n$. 
Denote by $u_n^{\sss (1)}$ and $u_n^{\sss (2)}$ the iterates for the operators $F_k^{\sss (1)}$, $F_k^{\sss (2)}$ corresponding to $B_1$ and $B_2$, respectively. We show by induction that $u_n^{\sss (1)} \le u_n^{\sss (2)}$ pointwise for every $n$. 
For $n=0$ this is trivial. For the induction step we first use the induction hypothesis and monotonicity of $F_k^{\sss(1)}$ (see (ii) in the proof of Lemma~\ref{lem:mild_exist}) and then the monotonicity of $F_k$ in $B$ to deduce
\[
u_{n+1}^{\sss (1)}= F_k^{\sss(1)}u_n^{\sss(1)} \le F_k^{\sss (1)}u_n^{\sss (2)} \le F_k^{\sss (2)}u_n^{\sss (2)} =u_{n+1}^{\sss (2)}.\qedhere
\]
\end{proof}

{\bf{Acknowledgement:}} The research has been initiated at the workshop 'Branching random walks and searching in trees (10w5085)' at Banff International Research Station (BIRS). 
Further parts of this research have been completed at the Institute for Mathematical Research (FIM) -- ETH Zurich and at the Mathematisches Forschungsinstitut Oberwolfach (MFO). AEK and MW thank BIRS and ME and AEK thank FIM and MFO for their hospitality.


\end{document}